\documentclass[11pt]{article}

\usepackage{amsmath,amssymb,amsthm}
\usepackage{mathrsfs}
\usepackage{cite}
\usepackage{geometry}
\newtheorem{theorem}{Theorem}[section]
\newtheorem{proposition}[theorem]{Proposition}
\newtheorem{lemma}[theorem]{Lemma}
\newtheorem{corollary}[theorem]{Corollary}

\newtheorem{definition}[theorem]{Definition}
\numberwithin{equation}{section}

\begin{document}

\title{Concentration, Local Uniqueness, and Morse Index of Multi-bubble Solutions for a Critical Exponential Biharmonic Choquard Equation}

\author{Wenjing Chen and Shengbing Deng\\
School of Mathematics and Statistics, Southwest University, Chongqing, China\\
Emails: wjchen@swu.edu.cn and shbdeng@swu.edu.cn}

\date{}

\maketitle

\begin{abstract}
Let $\Omega\subset\mathbb R^4$ be a bounded domain of class $C^6$,
let $0<\alpha<4$, and let $K\in C^4(\overline\Omega)$ be positive. We
study the Navier problem
\[
\Delta^2u=\varepsilon^{8-\alpha}K(x)e^{u(x)}
\left(\int_\Omega\frac{K(y)e^{u(y)}}{|x-y|^\alpha}\,dy\right),
\qquad
u=\Delta u=0\quad\text{on }\partial\Omega.
\]
Let $G$ be the Navier Green function, let $H$ be its regular part, and
put $M_\alpha=8\pi^2(8-\alpha)$. The concentration points are governed
by
\[
\mathcal F_m(\boldsymbol\xi)
=\sum_{i=1}^m
\left[\log K(\xi_i)+\frac{M_\alpha}{2}H(\xi_i,\xi_i)\right]
+M_\alpha\sum_{i<j}G(\xi_i,\xi_j).
\]
Every $C^1$-stable critical point of $\mathcal F_m$ produces a
positive $m$-bubble solution whose scales are of order
$\varepsilon^{-1}$ and whose nonlinear source converges to
$M_\alpha\sum_i\delta_{\xi_i^*}$. If the critical point is
nondegenerate, the corresponding $m$-bubble solution is locally
unique, modulo permutations, in a fixed scaled modulation
neighborhood. The linearized operator is
nondegenerate on $H^2(\Omega)\cap H_0^1(\Omega)$, and
\[
\operatorname{ind}(u_\varepsilon)
=m+\operatorname{ind}\!\left(-D^2\mathcal F_m(\boldsymbol\xi^*)\right).
\]
A critical four-dimensional capacity controls the scale directions.
The dilation block of the reduced Hessian is positive and equals
$8\pi^2b_\alpha^2|\log\varepsilon|^{-1}I_m+o(|\log\varepsilon|^{-1})$,
where $b_\alpha=(8-\alpha)/2$.
\end{abstract}



\section{Introduction}

Two-dimensional elliptic equations with exponential nonlinearities are
critical with respect to the Trudinger--Moser inequality.  A basic
model is
\[
-\Delta u=V(x)e^u
\quad\hbox{in }\Omega\subset\mathbb R^2.
\]
The blow-up theory starts with the a priori estimates of
Brezis--Merle and the quantization theorem of Li--Shafrir
\cite{BrezisMerle,LiShafrir}.  Classification results for the limiting
Liouville equation are given in \cite{ChenLi,ChenLiOu}.  Concentrating
solutions, multi-peak configurations, local uniqueness, and Morse
index formulas were later obtained by finite-dimensional reduction;
see \cite{DelPinoKowalczykMusso,GladialiGrossiOhtsukaSuzuki}.

The corresponding fourth-order critical dimension is four.  The
quadratic form is
\[
\int_\Omega |\Delta u|^2\,dx,
\]
and the endpoint embedding is the Adams inequality
\cite{Adams}.  The limiting local equation is the constant
$Q$-curvature equation.  Entire classifications and blow-up
compactness were established in
\cite{Lin,WeiXu,Martinazzi,MartinazziCC}.  Concentration estimates and
finite-dimensional interactions for fourth-order exponential
problems were developed in
\cite{DruetRobert,LinWeiFourthOrder}.  The logarithmic fundamental
solution is the main difference from subcritical dimensions.  It also
produces a critical capacitary effect in scale directions.

We next recall the second-order Choquard theory with exponential
critical growth.  In the plane, a representative equation is
\[
-\varepsilon^2\Delta u+V(x)u
=\varepsilon^{\mu-2}
\bigl(I_\mu*F(u)\bigr)f(u),
\qquad 0<\mu<2.
\]
Existence and concentration of semiclassical states were proved by
Alves--Cassani--Tarsi--Yang and by Yang
\cite{AlvesCassaniTarsiYang,YangPlanarChoquard}.  Indefinite and
strongly indefinite potentials were treated in
\cite{QinTang,GaoChenQinWu}.  Further ground-state and concentration
results appear in
\cite{ZhangLiaoTangQin,ChenLiWangQuasilinear,DengTianXiong}.
For the logarithmic Newton kernel, Cassani--Tarsi introduced a
Pohozaev--Trudinger weighted framework
\cite{CassaniTarsi}.  Critical-growth Schr\"odinger--Newton systems
and positive logarithmic Choquard solutions were studied in
\cite{LiuRadulescuZhang,CassaniDuLiu}.
Sharp asymptotics and nondegeneracy of the planar logarithmic
Choquard ground state were obtained in
\cite{BonheureCingolaniVanSchaftingen}.  Recent work also concerns
nondegeneracy and blow-up quantization for nonlocal exponential
equations \cite{GaoLiMa,Gluck}.  For power-type critical Choquard problems, bubble
nondegeneracy, local uniqueness, and concentrating solutions were
studied in
\cite{Lenzmann,DuYang,SquassinaYangZhao,LiLiuTangXu,GhimentiHuangPistoia,CannoneCingolaniYangZhao,ChenWangSubcritical,ChenWangMultiple}.
A recent fourth-order Hartree construction in a pierced domain is
given in \cite{TanYangBubbleTower}.

The fourth-order exponential Choquard equation has a different local
profile. In the companion manuscript
\cite[Theorems 1.1 and 1.2]{DengSpectrum}, which is included in the
submission as a related manuscript and is available to the editors and
referees, we studied
\[
\Delta^2U
=\left(\int_{\mathbb R^4}
\frac{e^{U(y)}}{|x-y|^\alpha}\,dy\right)e^{U(x)},
\qquad 0<\alpha<4.
\]
There, normal finite-mass solutions are reduced to the known
classification theory, and the complete linearized spectrum is
computed on $\mathbb S^4$. The kernel is generated by four
translations and one dilation, the negative eigenspace is
one-dimensional, and the quadratic form is coercive above the
negative and neutral spaces. These spectral facts, rather than the
classification reduction itself, are the whole-space input used
below. The classification is consistent with Niu's general
higher-order theorem \cite{NiuClassification}. Whole-space
variational results for biharmonic exponential Choquard problems were
obtained in \cite{ChenLiWang}.

The aim of this paper is to develop a complete finite-dimensional
reduction for the bounded Navier problem. We construct positive
multi-bubble solutions and determine their local uniqueness,
nondegeneracy, and Morse index. Two mechanisms distinguish the present
problem from the usual second-order reductions. First, the direct
Riesz interaction of two separated cores is lower order; the leading
pair term is generated by the regular part of the Navier projection.
Second, differentiation in a scale direction exposes the negative
constant mode of the limiting operator. The orthogonal correction
cancels its order-one contribution, and the first nonzero scale
energy is the four-dimensional biharmonic capacity of a small ball.

We study
\begin{equation}\label{eq:perturbed-problem}
\begin{cases}
\displaystyle
\Delta^{2}u=\varepsilon^{8-\alpha}K(x)e^{u(x)}
\left(\int_{\Omega}\frac{K(y)e^{u(y)}}{|x-y|^{\alpha}}\,dy\right)
&\text{in }\Omega,\\[1.2ex]
u=0,\quad \Delta u=0
&\text{on }\partial\Omega,
\end{cases}
\end{equation}
where $\Omega\subset\mathbb R^4$ is bounded and of class $C^6$, and
\begin{equation*}
K\in C^4(\overline\Omega),
\qquad \min_{\overline\Omega}K>0.
\end{equation*}
The parameter $\varepsilon>0$ will tend to zero.

Set
\begin{equation*}
b_\alpha=\frac{8-\alpha}{2},
\qquad M_\alpha=8\pi^2(8-\alpha),
\qquad
C_\alpha=
\left(\frac{12(8-\alpha)(6-\alpha)(4-\alpha)}{\pi^2}\right)^{\!\frac1{8-\alpha}}.
\end{equation*}
Let $G$ be the Navier Green function; see \cite{GazzolaGrunauSweers}:
\begin{equation*}
\begin{cases}
\Delta_x^2G(x,\xi)=\delta_\xi&\text{in }\Omega,\\
G(x,\xi)=\Delta_xG(x,\xi)=0&\text{on }\partial\Omega.
\end{cases}
\end{equation*}
We write
\begin{equation*}
G(x,\xi)=-\frac1{8\pi^2}\log|x-\xi|+H(x,\xi).
\end{equation*}
The functions $G$ and $H$ are symmetric. On every compact subset of
$\mathcal C_2(\Omega)$, all derivatives used below are uniformly
bounded; the diagonal value $H(\xi,\xi)$ is the Navier Robin
function.
For $m\geq1$, define
\begin{equation*}
\mathcal C_m(\Omega)
=\{(\xi_1,\ldots,\xi_m)\in\Omega^m:\xi_i\neq\xi_j\text{ for }i\neq j\}
\end{equation*}
and
\begin{equation*}
\begin{aligned}
\mathcal F_m(\boldsymbol\xi)
={}\sum_{i=1}^m
\left[\log K(\xi_i)+\frac{M_\alpha}{2}H(\xi_i,\xi_i)\right]+M_\alpha\sum_{1\leq i<j\leq m}G(\xi_i,\xi_j).
\end{aligned}
\end{equation*}
The function $\mathcal F_m$ contains the coefficient, the Robin term,
and the pair interactions. The direct Riesz interaction between two
separated cores is lower order. The leading pair term comes from the
Navier projection. We write
\[
\mathcal H=H^2(\Omega)\cap H_0^1(\Omega),
\qquad
\|v\|_{\mathcal H}^2=\int_\Omega|\Delta v|^2\,dx.
\]

We use the following stability condition.

\begin{definition}\label{def:stable-critical-point}
An isolated critical point $\boldsymbol\xi^{*}\in\mathcal C_{m}(\Omega)$ of $\mathcal F_{m}$ is called $C^{1}$-stable if there is a bounded open neighborhood $D\Subset\mathcal C_{m}(\Omega)$ such that $\boldsymbol\xi^{*}$ is the only critical point of $\mathcal F_m$ in $\overline D$, $\nabla\mathcal F_{m}\neq0$ on $\partial D$, and
\[
\deg(\nabla\mathcal F_{m},D,0)\neq0.
\]
Equivalently, the local Brouwer degree of $\nabla\mathcal F_m$ at $\boldsymbol\xi^*$ is nonzero. A nondegenerate critical point is $C^{1}$-stable after $D$ is chosen sufficiently small.
\end{definition}

The existence statement is as follows.

\begin{theorem}\label{thm:existence}
Fix
\begin{equation}\label{eq:tau-fixed-intro}
0<\tau<\min\{1,\alpha\}.
\end{equation}
Let $m\geq1$. Assume that
$\boldsymbol\xi^{*}=(\xi_{1}^{*},\ldots,\xi_{m}^{*})$ is a
$C^{1}$-stable critical point of $\mathcal F_m$. Then there is
$\varepsilon_{0}>0$ such that \eqref{eq:perturbed-problem} has a
positive solution $u_\varepsilon$ for
$0<\varepsilon<\varepsilon_{0}$. Moreover,
$u_\varepsilon\in C^{4,\gamma}(\overline\Omega)$ for every fixed
$0<\gamma<\min\{1,4-\alpha\}$. There are points
\begin{equation}\label{eq:center-convergence-intro}
\boldsymbol\xi_\varepsilon
=(\xi_{1,\varepsilon},\ldots,\xi_{m,\varepsilon})
\longrightarrow\boldsymbol\xi^{*}
\end{equation}
and numbers
\begin{equation*}
t_{i,\varepsilon}
=O(\varepsilon^\tau|\log\varepsilon|)
\end{equation*}
such that
\begin{equation}\label{eq:scales-intro}
\begin{aligned}
\mu_{i,\varepsilon}
={}\frac{C_\alpha}{\varepsilon}
K(\xi_{i,\varepsilon})^{-\frac{2}{8-\alpha}}
\exp\left\{
-\frac{2M_\alpha}{8-\alpha}
\left[
H(\xi_{i,\varepsilon},\xi_{i,\varepsilon})
+\sum_{j\ne i}G(\xi_{i,\varepsilon},\xi_{j,\varepsilon})
\right]
+t_{i,\varepsilon}
\right\}.
\end{aligned}
\end{equation}
If $PU_{\mu,\xi}$ is the Navier projection in
\eqref{eq:projection-definition}, then
\begin{equation}\label{eq:solution-expansion-intro}
u_\varepsilon
=\sum_{i=1}^{m}PU_{\mu_{i,\varepsilon},\xi_{i,\varepsilon}}
+\phi_\varepsilon,
\qquad
\|\phi_\varepsilon\|_{L^\infty(\Omega)}
+\|\phi_\varepsilon\|_{\mathcal H}
\leq C\varepsilon^\tau|\log\varepsilon|.
\end{equation}
The nonlinear source satisfies
\begin{equation}\label{eq:measure-convergence-intro}
\varepsilon^{8-\alpha}K(x)e^{u_\varepsilon(x)}
\int_\Omega
\frac{K(y)e^{u_\varepsilon(y)}}{|x-y|^\alpha}\,dy
\rightharpoonup
M_\alpha\sum_{i=1}^{m}\delta_{\xi_i^{*}}
\end{equation}
as measures. For every
\begin{equation}\label{eq:outer-beta-range}
0<\beta<\min\{1,4-\alpha\},
\end{equation}
one has
\begin{equation}\label{eq:outer-convergence-intro}
u_\varepsilon
\longrightarrow
M_\alpha\sum_{i=1}^{m}G(\cdot,\xi_i^{*})
\quad\hbox{in }C^{4,\beta}_{\mathrm{loc}}
\left(\overline\Omega\setminus
\{\xi_{1}^{*},\ldots,\xi_{m}^{*}\}\right).
\end{equation}
The values at the modulation centers satisfy
\begin{equation}\label{eq:peak-height-intro}
\begin{aligned}
u_\varepsilon(\xi_{i,\varepsilon})
={}&(8-\alpha)\log\frac{C_\alpha}{\varepsilon}
-2\log K(\xi_{i,\varepsilon})-M_\alpha
\left[
H(\xi_{i,\varepsilon},\xi_{i,\varepsilon})
+\sum_{j\ne i}G(\xi_{i,\varepsilon},\xi_{j,\varepsilon})
\right]
+o(1).
\end{aligned}
\end{equation}
\end{theorem}

The next theorem gives the qualitative information requested in the
title. At a solution $u$, the linearized Navier operator is
\[
\begin{aligned}
\mathscr L_{\varepsilon,u}\phi
={}&\Delta^2\phi
-\varepsilon^{8-\alpha}K(x)e^{u(x)}\phi(x)
\left(\int_\Omega\frac{K(y)e^{u(y)}}{|x-y|^\alpha}\,dy\right)\\
&-\varepsilon^{8-\alpha}K(x)e^{u(x)}
\int_\Omega\frac{K(y)e^{u(y)}\phi(y)}{|x-y|^\alpha}\,dy,
\end{aligned}
\]
with $\phi=\Delta\phi=0$ on $\partial\Omega$. The nonlocal part is
compact relative to the Navier biharmonic form. Hence the associated
realization on $L^2(\Omega)$ is self-adjoint with compact resolvent.
Its Morse index is the maximal dimension of a subspace of
$\mathcal H$ on which the second variation of $J_\varepsilon$ is
negative definite; equivalently, it is the number of negative
eigenvalues counted with multiplicity. For a critical point of a
$C^2$ function, $\operatorname{ind}$ denotes the number of negative
eigenvalues of its Hessian, counted with multiplicity.

The qualitative properties of the constructed solutions are summarized next.

\begin{theorem}\label{thm:qualitative}
Assume that $\boldsymbol\xi^{*}$ is a nondegenerate critical point
of $\mathcal F_m$. There are constants $C_0>0$, $r>0$, and
$\varepsilon_0>0$ with the following property. Let
\[
v_\varepsilon
=W_{\boldsymbol s,\boldsymbol\eta}+\psi
\]
be another solution of \eqref{eq:perturbed-problem}, written with an
admissible separated parameter $(\boldsymbol s,\boldsymbol\eta)$.
Assume that $\psi$ satisfies the orthogonality conditions
\eqref{eq:orthogonality} associated with these parameters, that
\begin{equation}\label{eq:uniqueness-scale-class}
|\boldsymbol s|
\leq C_0\varepsilon^\tau|\log\varepsilon|,
\end{equation}
and that
\begin{equation}\label{eq:branch-distance}
\begin{aligned}
d_\varepsilon(v_\varepsilon,u_\varepsilon)
:=\min_{\sigma\in S_m}\Bigg\{
&\sum_{i=1}^{m}|s_{\sigma(i)}-t_{i,\varepsilon}|+\sum_{i=1}^{m}\mu_{i,\varepsilon}
|\eta_{\sigma(i)}-\xi_{i,\varepsilon}| +\|\psi-\phi_\varepsilon\|_{L^\infty(\Omega)}
+\|\psi-\phi_\varepsilon\|_{\mathcal H}
\Bigg\}<r.
\end{aligned}
\end{equation}
Then $v_\varepsilon=u_\varepsilon$ after a permutation of the
peaks.

For $0<\varepsilon<\varepsilon_0$, the operator
$\mathscr L_{\varepsilon,u_\varepsilon}$ has trivial kernel in
$\mathcal H$. Its Morse index is
\begin{equation}\label{eq:Morse-formula-intro}
\operatorname{ind}(u_\varepsilon)
=m+\operatorname{ind}
\left(-D^2\mathcal F_m(\boldsymbol\xi^{*})\right).
\end{equation}
Equivalently,
\begin{equation}\label{eq:Morse-equivalent}
\operatorname{ind}(u_\varepsilon)
=5m-\operatorname{ind}
\left(D^2\mathcal F_m(\boldsymbol\xi^{*})\right).
\end{equation}
\end{theorem}

For one peak,
\begin{equation*}
\mathcal F_{1}(\xi)
=\log K(\xi)+\frac{M_{\alpha}}{2}H(\xi,\xi).
\end{equation*}
For one concentration point, the preceding theorems give the following consequence.

\begin{corollary}\label{cor:single-peak}
Every $C^1$-stable critical point of $\mathcal F_{1}$ produces a positive
single-peak solution. If the critical point is nondegenerate, the
corresponding solution is locally unique and nondegenerate. A strict nondegenerate
local maximum produces a solution of Morse index one. A strict
nondegenerate local minimum produces a solution of Morse index five.
\end{corollary}

The proof uses a finite-dimensional Lyapunov--Schmidt reduction
\cite{AmbrosettiMalchiodi}. Three points require separate arguments.
First, the nonlocal error must be estimated simultaneously in the
core, neck, and exterior regions. On bounded scales in the $i$-th
core,
\begin{equation*}
\frac{\text{field produced by the $j$-th core}}
{\text{self field of the $i$-th core}}
=O(\mu_i^{-\alpha})=O(\varepsilon^\alpha),
\qquad i\ne j.
\end{equation*}
Second,
\begin{equation*}
\partial_{\log\mu}PU_{\mu,\xi}
=Z_0(\mu(\cdot-\xi))+b_\alpha+o(1).
\end{equation*}
The derivative of the orthogonal correction cancels the constant
negative mode.  The remaining dilation energy is
\begin{equation*}
\frac{8\pi^2b_\alpha^2}{|\log\varepsilon|}
+o(|\log\varepsilon|^{-1}).
\end{equation*}
Third, the Morse index requires a localization formula for the
nonlocal quadratic form.  Related low-mode arguments appear in
\cite{BatesShi,ChoiKimLee,GladialiGrossiOhtsukaSuzuki}.  Each bubble contributes one negative normal
direction.  The dilation block is positive.  The translation block is
congruent to $-M_\alpha D^2\mathcal F_m(\boldsymbol\xi^*)$.

The paper is organized as follows.  Section 2 contains the analytic
preliminaries and the approximate solution.  Section 3 proves the
projected linear estimate.  Section 4 solves the nonlinear projected
problem.  Section 5 studies the reduced functional and proves Theorem
\ref{thm:existence}.  Section 6 proves local uniqueness and
nondegeneracy.  Section 7 gives the Morse index.  Section 8 contains
the technical estimates.

\section{Preliminaries and construction of the approximate solution}

\subsection{Variational setting and analytic estimates}

Set
\begin{equation}\label{eq:energy-space}
\mathcal H:=H^2(\Omega)\cap H_0^1(\Omega),
\qquad
\|v\|_{\mathcal H}^2:=\int_\Omega|\Delta v|^2\,dx.
\end{equation}
This norm is equivalent to the usual $H^2$ norm on $\mathcal H$.

We first record the domain version of the Hardy--Littlewood--Sobolev estimate.

\begin{lemma}\label{lem:HLS-domain}
Let $p=8/(8-\alpha)$. For $f,g\in L^p(\Omega)$,
\begin{equation}\label{eq:HLS-domain}
\int_\Omega\int_\Omega
\frac{|f(x)g(y)|}{|x-y|^\alpha}\,dxdy
\leq C_{\Omega,\alpha}\|f\|_{L^p(\Omega)}
\|g\|_{L^p(\Omega)}.
\end{equation}
\end{lemma}

\begin{proof}
Extend $f$ and $g$ by zero to $\mathbb R^4$. Since
\[
\frac1p+\frac1p+\frac\alpha4=2,
\]
the Hardy--Littlewood--Sobolev inequality on $\mathbb R^4$ gives
\eqref{eq:HLS-domain}. The constant depends only on $\alpha$ and the
normalization of the Riesz kernel. Restriction to $\Omega$ completes
the proof.
\end{proof}

The Adams inequality gives the exponential integrability needed below.

\begin{lemma}\label{lem:Adams-consequence}
For every $R>0$ and every $q<\infty$, there is $C=C(R,q,\Omega)$ such
that
\begin{equation*}
\sup_{\|u\|_{\mathcal H}\leq R}
\|e^{|u|}\|_{L^q(\Omega)}\leq C.
\end{equation*}
Moreover, multiplication by finitely many functions of $\mathcal H$
preserves the integrability required in \eqref{eq:HLS-domain}.
\end{lemma}

\begin{proof}
Fix $q<\infty$. For every $\delta>0$,
\begin{equation*}
q|s|\leq\delta s^2+\frac{q^2}{4\delta}.
\end{equation*}
Choose $\delta$ so that $\delta R^2$ is below the Adams threshold.
Then
\[
\int_\Omega e^{q|u|}\,dx
\leq e^{q^2/(4\delta)}\int_\Omega e^{\delta u^2}\,dx
\leq C.
\]
The embedding $\mathcal H\hookrightarrow L^r(\Omega)$ holds for
every finite $r$. H\"older's inequality therefore permits any fixed
finite product of $\mathcal H$ functions. This proves the last
assertion.
\end{proof}

We shall also use a boundary version of the local H\"older estimate for
Riesz potentials.

\begin{lemma}\label{lem:boundary-Riesz-holder}
Let $f\in C^{0,\beta}(\overline\Omega)$ for some $0<\beta<1$ and set
\[
\mathcal I_\alpha f(x)
:=\int_\Omega\frac{f(y)}{|x-y|^\alpha}\,dy.
\]
Then $\mathcal I_\alpha f\in C^{0,\gamma}(\overline\Omega)$ for every
$0<\gamma<\min\{1,4-\alpha\}$. Moreover,
\[
\|\mathcal I_\alpha f\|_{C^{0,\gamma}(\overline\Omega)}
\leq C\|f\|_{C^{0,\beta}(\overline\Omega)},
\]
where $C$ depends on $\Omega$, $\alpha$, $\beta$, and $\gamma$.
\end{lemma}

\begin{proof}
Extend $f$ to a compactly supported function in
$C^{0,\beta}(\mathbb R^4)$. Let $h=x'-x$ and split the difference of the
potentials into $|x-y|\leq2|h|$ and its complement. The near part is
bounded by
\[
C\|f\|_\infty\int_0^{3|h|}r^{3-\alpha}\,dr
\leq C\|f\|_\infty |h|^{4-\alpha}.
\]
On the complementary region, the mean value theorem gives
\[
\bigl||x'-y|^{-\alpha}-|x-y|^{-\alpha}\bigr|
\leq C|h||x-y|^{-\alpha-1}.
\]
Integration from $2|h|$ to a fixed radius gives $C|h|$ if
$\alpha<3$, $C|h||\log|h||$ if $\alpha=3$, and
$C|h|^{4-\alpha}$ if $3<\alpha<4$. The remaining far part is
Lipschitz. Hence the difference is bounded by $C|h|^\gamma$ for every
$\gamma<\min\{1,4-\alpha\}$, uniformly for $x,x'\in\overline\Omega$.
The supremum estimate follows from the local integrability of
$|x-y|^{-\alpha}$. This proves the assertion.
\end{proof}

Problem \eqref{eq:perturbed-problem} is the Euler equation of
\begin{equation*}
\begin{aligned}
J_\varepsilon(u)
={}\frac12\int_\Omega|\Delta u|^2\,dx-\frac{\varepsilon^{8-\alpha}}2
\int_\Omega\int_\Omega
\frac{K(x)K(y)e^{u(x)+u(y)}}{|x-y|^\alpha}\,dxdy.
\end{aligned}
\end{equation*}

We next verify the smooth variational structure of the problem.

\begin{lemma}
The functional $J_\varepsilon$ belongs to $C^\infty(\mathcal H,\mathbb R)$.
Its first derivative is
\begin{equation}\label{eq:first-variation}
\begin{aligned}
J_\varepsilon'(u)[\phi]
={}\int_\Omega\Delta u\Delta\phi\,dx-\varepsilon^{8-\alpha}
\int_\Omega\int_\Omega
\frac{K(x)K(y)e^{u(x)+u(y)}\phi(x)}{|x-y|^\alpha}\,dxdy.
\end{aligned}
\end{equation}
Weak critical points are weak solutions of
\eqref{eq:perturbed-problem}.
\end{lemma}

\begin{proof}
Let $B_R$ be a bounded ball of $\mathcal H$. Lemma
\ref{lem:Adams-consequence} gives
\begin{equation}\label{eq:energy-exp-integrability}
\sup_{u\in B_R}\|e^u\|_{L^q(\Omega)}<\infty
\quad\text{for every }q<\infty.
\end{equation}
For $n\geq1$ and $\phi_1,\ldots,\phi_n\in\mathcal H$, the $n$-th
derivative of the nonlocal part is a finite sum of terms
\begin{equation*}
\int_\Omega\int_\Omega
\frac{K(x)K(y)e^{u(x)+u(y)}
\prod_{j=1}^n(\phi_j(x)+\phi_j(y))}
{|x-y|^\alpha}\,dxdy.
\end{equation*}
Choose $q$ large. Apply H\"older's inequality separately in $x$ and
$y$. Use $\mathcal H\hookrightarrow L^r(\Omega)$ for every finite
$r$. The two remaining factors belong to $L^{8/(8-\alpha)}$ by
\eqref{eq:energy-exp-integrability}. Lemma \ref{lem:HLS-domain} then
gives
\begin{equation*}
|D^nJ_\varepsilon(u)[\phi_1,\ldots,\phi_n]|
\leq C_{R,n}\prod_{j=1}^n\|\phi_j\|_{\mathcal H}.
\end{equation*}
The same estimate applied to differences proves continuity of
$D^nJ_\varepsilon$. Thus $J_\varepsilon\in C^\infty$.

The symmetry in $x$ and $y$ combines the two first-variation terms
and gives \eqref{eq:first-variation}. For
$\phi\in C_c^\infty(\Omega)$, integration by parts yields
\[
\int_\Omega\Delta u\Delta\phi
=\langle\Delta^2u,\phi\rangle.
\]
The weak boundary space is $\mathcal H$. Hence the natural boundary
conditions are $u=\Delta u=0$. The Euler equation is
\eqref{eq:perturbed-problem}.
\end{proof}

\subsection{Entire bubbles and the limiting linear theory}

The limiting equation is
\begin{equation*}
\Delta^2U=
\left(\int_{\mathbb R^4}\frac{e^{U(y)}}{|x-y|^\alpha}\,dy\right)e^{U(x)}
\quad\hbox{in }\mathbb R^4.
\end{equation*}
For $\mu>0$ and $\xi\in\mathbb R^4$, set
\begin{equation*}
U_{\mu,\xi}(x)
=b_\alpha\log\left(
\frac{C_\alpha\mu}{1+\mu^2|x-\xi|^2}\right).
\end{equation*}
Then
\begin{equation*}
\Delta^2U_{\mu,\xi}(x)
=\frac{48(8-\alpha)\mu^4}{(1+\mu^2|x-\xi|^2)^4}
=:F_{\mu,\xi}(x),
\end{equation*}
\begin{equation}\label{eq:bubble-mass}
\int_{\mathbb R^4}F_{\mu,\xi}\,dx=M_\alpha,
\end{equation}
and
\begin{equation}\label{eq:bubble-nonlocal}
\left(\int_{\mathbb R^4}
\frac{e^{U_{\mu,\xi}(y)}}{|x-y|^\alpha}\,dy\right)e^{U_{\mu,\xi}(x)}
=F_{\mu,\xi}(x).
\end{equation}

We record the part of \cite{DengSpectrum} used in this paper. Let
\begin{equation*}
U(x):=U_{1,0}(x)
=b_{\alpha}\log\frac{C_{\alpha}}{1+|x|^{2}},
\qquad F:=F_{1,0}.
\end{equation*}
The linearized operator is
\begin{equation*}
\begin{aligned}
\mathcal L_{\infty}\phi(x)
={}&\Delta^{2}\phi(x)
-e^{U(x)}\int_{\mathbb R^{4}}
\frac{e^{U(y)}\phi(y)}{|x-y|^{\alpha}}\,dy-e^{U(x)}\phi(x)
\int_{\mathbb R^{4}}
\frac{e^{U(y)}}{|x-y|^{\alpha}}\,dy.
\end{aligned}
\end{equation*}
Its symmetric quadratic form is
\begin{equation*}
\begin{aligned}
\mathcal Q_\infty(\phi,\psi)
={}\int_{\mathbb R^4}\Delta\phi\,\Delta\psi\,dx-\frac12\int_{\mathbb R^4}\int_{\mathbb R^4}
\frac{e^{U(x)+U(y)}}{|x-y|^\alpha}
(\phi(x)+\phi(y))(\psi(x)+\psi(y))\,dxdy.
\end{aligned}
\end{equation*}
The five kernel functions are
\begin{equation*}
Z_{\ell}(z)=\frac{(8-\alpha)z_{\ell}}{1+|z|^{2}},
\qquad 1\leq\ell\leq4,\qquad
Z_{0}(z)=b_{\alpha}\frac{1-|z|^{2}}{1+|z|^{2}}.
\end{equation*}
The index $0$ is used for dilation in this paper.

The conformal energy space used below is
\begin{equation*}
\mathcal E_\alpha
:=\left\{\phi\in H^2_{\rm loc}(\mathbb R^4):
\Delta\phi\in L^2(\mathbb R^4),\
\int_{\mathbb R^4}F\phi^2\,dx<+\infty\right\},
\end{equation*}
endowed with
\begin{equation*}
\|\phi\|_{\mathcal E_\alpha}^2
:=\int_{\mathbb R^4}|\Delta\phi|^2\,dx
+\int_{\mathbb R^4}F\phi^2\,dx.
\end{equation*}
Constants belong to this weighted space. Thus the negative constant
mode is admissible. The Hardy--Littlewood--Sobolev inequality and the
stereographic representation show that $\mathcal Q_\infty$ extends
continuously to $\mathcal E_\alpha\times\mathcal E_\alpha$.

The companion spectral analysis supplies the exact low-mode information
needed in the reduction.

\begin{theorem}\label{thm:limit-input}
If $\phi$ is a bounded distributional solution of
$\mathcal L_\infty\phi=0$, then
\begin{equation*}
\phi\in\operatorname{span}\{Z_0,Z_1,Z_2,Z_3,Z_4\}.
\end{equation*}
The same kernel statement holds for solutions in
$\mathcal E_\alpha$. Under stereographic projection to $\mathbb S^4$,
the negative eigenspace is $\mathcal H_0$ and the kernel is
$\mathcal H_1$. In Euclidean variables,
\begin{equation}\label{eq:negative-mode-value}
\mathcal Q_\infty(1,1)=-2M_\alpha.
\end{equation}
There is $c_\alpha>0$ such that
\begin{equation}\label{eq:limit-coercivity}
\mathcal Q_\infty(\phi,\phi)
\geq c_\alpha\int_{\mathbb R^4}|\Delta\phi|^2\,dx
\end{equation}
for every $\phi\in\mathcal E_\alpha$ satisfying
\begin{equation}\label{eq:limit-coercivity-orthogonality}
\int_{\mathbb R^4}F\phi\,dx=0,
\qquad
\int_{\mathbb R^4}FZ_k\phi\,dx=0,
\quad 0\leq k\leq4.
\end{equation}
The kernel, the negative direction, and the coercivity constant are
transported by translations and dilations.
\end{theorem}

\begin{proof}
The spectral theorem in \cite[Theorem 1.2]{DengSpectrum} conjugates
$\mathcal L_\infty$ to a self-adjoint operator on $\mathbb S^4$. The
spherical harmonics of degree zero give the simple negative
eigenvalue, degree one gives the five conformal kernel functions, and
all modes of degree at least two are separated from zero. The
weighted $L^2$ space on the sphere is equivalent, under
stereographic projection, to the $F$-weighted term in
$\mathcal E_\alpha$. Hence both bounded kernel elements and
$\mathcal E_\alpha$ kernel elements are exhausted by the five
conformal modes. Pulling the quadratic decomposition back gives
\eqref{eq:negative-mode-value} and \eqref{eq:limit-coercivity}. The
$\mathcal H_0$ and $\mathcal H_1$ orthogonality conditions become
exactly \eqref{eq:limit-coercivity-orthogonality}. Translation and
dilation preserve this decomposition.
\end{proof}

The exact spherical eigenvalues are stated in \cite[Theorem 1.2]{DengSpectrum}. The first three spectral levels are enough here. The negative mode is simple. The kernel has dimension five. The next eigenvalue is separated from zero. This separation is uniform under translations and dilations.

The spherical spectral gap also controls the weighted low modes.

\begin{lemma}\label{lem:weighted-low-mode-control}
There is a constant $C>0$ such that every
$\phi\in\mathcal E_\alpha$ satisfies
\begin{equation}\label{eq:weighted-low-mode-control}
\begin{aligned}
\int_{\mathbb R^4}F\phi^2\,dx
\leq C\Bigg[
&\int_{\mathbb R^4}|\Delta\phi|^2\,dx
+\left|\int_{\mathbb R^4}F\phi\,dx\right|^2+\sum_{k=0}^{4}
\left|\int_{\mathbb R^4}FZ_k\phi\,dx\right|^2
\Bigg].
\end{aligned}
\end{equation}
The same estimate holds for every translated and dilated bubble, with
the same constant after the natural change of variables.
\end{lemma}

\begin{proof}
Use stereographic projection.  The weighted integral on the left is
equivalent to the $L^2(\mathbb S^4)$ norm of the spherical
representative.  Decompose this representative into
$\mathcal H_0\oplus\mathcal H_1$ and its orthogonal complement.  The
coefficients of the $\mathcal H_0$ component and of the five
$\mathcal H_1$ components are equivalent to the six moments on the
right-hand side of \eqref{eq:weighted-low-mode-control}.  On the
orthogonal complement, the Paneitz energy controls the spherical
$L^2$ norm by the spectral gap in Theorem
\ref{thm:limit-input}.  Pulling the estimate back to $\mathbb R^4$
gives \eqref{eq:weighted-low-mode-control}.  Translation and
dilation preserve the conformal decomposition.
\end{proof}

We also use the following integral identities. They follow from scaling and symmetry.

\begin{lemma}\label{lem:bubble-integrals}
With $F=F_{1,0}$ as above,
\begin{equation}\label{eq:moment-zero}
\int_{\mathbb R^{4}}z_{k}F(z)\,dz=0,
\qquad 1\leq k\leq4,
\end{equation}
\begin{equation*}
\int_{\mathbb R^{4}}|z|^{2}F(z)\,dz<+\infty,
\end{equation*}
and
\begin{equation*}
B_{\alpha}:=\int_{\mathbb R^{4}}F(z)U(z)\,dz
\end{equation*}
 is finite. Moreover,
\begin{equation}\label{eq:self-energy-scaling}
\int_{\mathbb R^{4}}F_{\mu,\xi}(x)U_{\mu,\xi}(x)\,dx
=B_{\alpha}+b_{\alpha}M_{\alpha}\log\mu.
\end{equation}
\end{lemma}

\begin{proof}
The density $F$ is radial. This gives \eqref{eq:moment-zero}. It decays like $|z|^{-8}$. Hence the second moment is finite in dimension four. The function $U$ grows logarithmically. Thus $FU$ is integrable. Finally, put $z=\mu(x-\xi)$. Then
\[
F_{\mu,\xi}(x)\,dx=F(z)\,dz
\]
and
\[
U_{\mu,\xi}(x)=U(z)+b_{\alpha}\log\mu.
\]
Equation \eqref{eq:self-energy-scaling} follows from \eqref{eq:bubble-mass}.
\end{proof}

\subsection{Green projection}

For $\xi\in\Omega$ and $\mu>0$, define the Navier projection by
\begin{equation}\label{eq:projection-definition}
\begin{cases}
\Delta^{2}P U_{\mu,\xi}=F_{\mu,\xi}
&\text{in }\Omega,\\
P U_{\mu,\xi}=\Delta P U_{\mu,\xi}=0
&\text{on }\partial\Omega.
\end{cases}
\end{equation}
Thus
\begin{equation}\label{eq:projection-Green}
P U_{\mu,\xi}(x)
=\int_{\Omega}G(x,y)F_{\mu,\xi}(y)\,dy.
\end{equation}
We work with separated configurations. For $\eta>0$, set
\begin{equation*}
\mathcal C_{m,\eta}(\Omega)
:=\left\{\boldsymbol\xi\in\mathcal C_{m}(\Omega):
\operatorname{dist}(\xi_{i},\partial\Omega)\geq\eta,
\ |\xi_{i}-\xi_{j}|\geq\eta
\right\}.
\end{equation*}

We now expand the Navier projection of one entire bubble.

\begin{lemma}\label{lem:projection-expansion}
Fix $\eta>0$. Uniformly for $\xi$ with $\operatorname{dist}(\xi,\partial\Omega)\geq\eta$ and $\mu\to+\infty$,
\begin{equation}\label{eq:projection-expansion}
P U_{\mu,\xi}(x)
=U_{\mu,\xi}(x)
+b_{\alpha}\log\mu-b_{\alpha}\log C_{\alpha}
+M_{\alpha}H(x,\xi)
+R_{\mu,\xi}(x),
\end{equation}
where
\begin{equation}\label{eq:projection-remainder}
\|R_{\mu,\xi}\|_{C^{2}(\overline\Omega)}
\leq C\mu^{-2}.
\end{equation}
The same estimate holds after one derivative with respect to $\xi$. After one derivative with respect to $\log\mu$, the right-hand side is $O(\mu^{-2})$ in $C^{2}$.
\end{lemma}

\begin{proof}
By the whole-space logarithmic representation,
\begin{equation}\label{eq:bubble-log-representation}
U_{\mu,\xi}(x)
=-\frac1{8\pi^2}\int_{\mathbb R^4}
\log|x-y|F_{\mu,\xi}(y)\,dy
+b_\alpha\log C_\alpha-b_\alpha\log\mu.
\end{equation}
Subtract \eqref{eq:bubble-log-representation} from
\eqref{eq:projection-Green}. Since
$G(x,y)=-(8\pi^2)^{-1}\log|x-y|+H(x,y)$, we obtain
\begin{equation}\label{eq:projection-remainder-formula}
\begin{aligned}
R_{\mu,\xi}(x)
={}&\int_\Omega[H(x,y)-H(x,\xi)]F_{\mu,\xi}(y)\,dy+\frac1{8\pi^2}\int_{\mathbb R^4\setminus\Omega}
\log|x-y|F_{\mu,\xi}(y)\,dy\\
&-H(x,\xi)\int_{\mathbb R^4\setminus\Omega}
F_{\mu,\xi}(y)\,dy.
\end{aligned}
\end{equation}
The center is at distance at least $\eta$ from the boundary. Hence,
for $y\notin\Omega$,
$|y-\xi|\geq\eta$. The explicit source gives
\begin{equation}\label{eq:projection-tail}
\int_{\mathbb R^4\setminus\Omega}
(1+|\log|x-y||)F_{\mu,\xi}(y)\,dy
\leq C\mu^{-4}\log(2+\mu).
\end{equation}
The estimate holds after at most two $x$-derivatives.

For the first term in \eqref{eq:projection-remainder-formula}, set
$y=\xi+z/\mu$. Then
$F_{\mu,\xi}(y)\,dy=F(z)\,dz$. Taylor's formula gives
\begin{equation}\label{eq:H-Taylor}
H(x,\xi+z/\mu)-H(x,\xi)
=\mu^{-1}\nabla_yH(x,\xi)\cdot z
+\frac1{2\mu^2}z^TD_y^2H(x,\xi+\theta z/\mu)z.
\end{equation}
The first integral vanishes by \eqref{eq:moment-zero}. The second is
bounded by $C\mu^{-2}\int|z|^2F(z)\,dz$. Thus
\eqref{eq:projection-remainder} follows. The same computation is valid
for $D_x^jR$, $0\leq j\leq2$.

Differentiate \eqref{eq:projection-remainder-formula} with respect to
$\xi$. After the change of variables, the derivative falls on the
smooth factor in \eqref{eq:H-Taylor}; the odd first moment still
vanishes. Hence the bound remains $O(\mu^{-2})$. A derivative with
respect to $\log\mu$ differentiates $z/\mu$ and the tail. The resulting
integrands are bounded by $C\mu^{-2}|z|^2F(z)$ and by the right-hand
side of \eqref{eq:projection-tail}. This proves the parameter
estimates.
\end{proof}

The projection expansion immediately gives the interaction between separated bubbles.

\begin{lemma}\label{lem:pair-interaction}
Let $\boldsymbol\xi\in\mathcal C_{m,\eta}(\Omega)$ and let the scales be comparable and tend to infinity. If $i\neq j$, then
\begin{equation}\label{eq:pair-projection}
P U_{\mu_{j},\xi_{j}}(x)
=M_{\alpha}G(x,\xi_{j})+O(\mu_{j}^{-2})
\end{equation}
uniformly for $x\in B_{\eta/4}(\xi_{i})$. Moreover,
\begin{equation}\label{eq:quadratic-cross}
\int_{\Omega}
\Delta P U_{\mu_{i},\xi_{i}}
\Delta P U_{\mu_{j},\xi_{j}}\,dx
=M_{\alpha}^{2}G(\xi_{i},\xi_{j})+O(\mu_{i}^{-2}+\mu_{j}^{-2}).
\end{equation}
\end{lemma}

\begin{proof}
Let $x\in B_{\eta/4}(\xi_i)$ and $i\neq j$. The separation condition
implies $|x-y|\geq\eta/2$ for
$y\in B_{\eta/4}(\xi_j)$. From
\eqref{eq:projection-Green},
\[
PU_{\mu_j,\xi_j}(x)-M_\alpha G(x,\xi_j)
=\int_\Omega[G(x,y)-G(x,\xi_j)]F_{\mu_j,\xi_j}(y)\,dy
-G(x,\xi_j)\int_{\mathbb R^4\setminus\Omega}F_{\mu_j,\xi_j}.
\]
Expand $G(x,y)$ at $y=\xi_j$. The linear term integrates to zero by
\eqref{eq:moment-zero}. The quadratic remainder is bounded by
$C\mu_j^{-2}\int|z|^2F(z)\,dz$. The exterior mass is
$O(\mu_j^{-4})$. This proves \eqref{eq:pair-projection}. The argument
is unchanged after two $x$-derivatives and one scaled parameter
derivative.

For the energy interaction, use the Navier boundary conditions twice:
\begin{equation*}
\int_\Omega\Delta PU_{\mu_i,\xi_i}
\Delta PU_{\mu_j,\xi_j}\,dx
=\int_\Omega F_{\mu_i,\xi_i}(x)
PU_{\mu_j,\xi_j}(x)\,dx.
\end{equation*}
Split the last integral into $B_{\eta/4}(\xi_i)$ and its complement.
On the ball, insert \eqref{eq:pair-projection} and expand
$G(x,\xi_j)$ at $x=\xi_i$. The first moment vanishes. The error is
$O(\mu_i^{-2}+\mu_j^{-2})$. Outside the ball,
$F_{\mu_i,\xi_i}$ has mass $O(\mu_i^{-4})$. Therefore
\eqref{eq:quadratic-cross} follows.
\end{proof}

\subsection{Construction of the approximate solution and residual estimates}

\subsubsection{Choice of scales}

Fix $\eta>0$ and $\boldsymbol\xi\in\mathcal C_{m,\eta}(\Omega)$. Put
\begin{equation*}
S_i(\boldsymbol\xi)
:=H(\xi_i,\xi_i)+\sum_{j\ne i}G(\xi_i,\xi_j).
\end{equation*}
Define
\begin{equation}\label{eq:leading-scale}
\overline\mu_i(\varepsilon,\boldsymbol\xi)
:=\frac{C_\alpha}{\varepsilon}
K(\xi_i)^{-\frac{2}{8-\alpha}}
\exp\left\{-\frac{2M_\alpha}{8-\alpha}S_i(\boldsymbol\xi)\right\}.
\end{equation}
Let
\begin{equation}\label{eq:modulated-scale}
\mu_i=\overline\mu_i e^{t_i},
\qquad |\boldsymbol t|\leq t_0,
\end{equation}
where $t_0>0$ is fixed and small. Uniformly in the separated configuration set,
\begin{equation*}
c\varepsilon^{-1}\leq\mu_i\leq C\varepsilon^{-1}.
\end{equation*}
Set
\begin{equation*}
W_{\boldsymbol t,\boldsymbol\xi}
:=\sum_{i=1}^{m}PU_{\mu_i,\xi_i}.
\end{equation*}
The following expansion identifies the matching constant in each core.

\begin{lemma}
For $x\in B_{\eta/4}(\xi_i)$,
\begin{equation}\label{eq:local-W}
\begin{aligned}
W_{\boldsymbol t,\boldsymbol\xi}(x)
={}&U_{\mu_i,\xi_i}(x)+A_i
+M_\alpha\nabla_x\left[H(x,\xi_i)+\sum_{j\ne i}G(x,\xi_j)\right]_{x=\xi_i}\cdot(x-\xi_i)\\
&+O(|x-\xi_i|^2+\varepsilon^2),
\end{aligned}
\end{equation}
uniformly for separated configurations, where
\begin{equation}\label{eq:A-i}
A_i=b_\alpha\log\mu_i-b_\alpha\log C_\alpha+M_\alpha S_i(\boldsymbol\xi).
\end{equation}
Moreover,
\begin{equation}\label{eq:balance}
\varepsilon^{8-\alpha}K(\xi_i)^2e^{2A_i}
=e^{(8-\alpha)t_i}=e^{2b_\alpha t_i}.
\end{equation}
The estimates remain valid after one scaled center derivative and one
logarithmic scale derivative.
\end{lemma}

\begin{proof}
Apply Lemma \ref{lem:projection-expansion} to the $i$-th bubble:
\[
PU_{\mu_i,\xi_i}(x)
=U_{\mu_i,\xi_i}(x)+b_\alpha\log\mu_i-b_\alpha\log C_\alpha
+M_\alpha H(x,\xi_i)+O(\mu_i^{-2}).
\]
For $j\neq i$, Lemma \ref{lem:pair-interaction} gives
\[
PU_{\mu_j,\xi_j}(x)=M_\alpha G(x,\xi_j)+O(\mu_j^{-2}).
\]
Taylor-expand the regular functions at $x=\xi_i$:
\begin{align*}
H(x,\xi_i)&=H(\xi_i,\xi_i)
+\nabla_xH(\xi_i,\xi_i)\cdot(x-\xi_i)+O(|x-\xi_i|^2),\\
G(x,\xi_j)&=G(\xi_i,\xi_j)
+\nabla_xG(\xi_i,\xi_j)\cdot(x-\xi_i)+O(|x-\xi_i|^2).
\end{align*}
Since $\mu_j^{-2}=O(\varepsilon^2)$, summation gives
\eqref{eq:local-W} and \eqref{eq:A-i}.

From \eqref{eq:leading-scale} and
$\mu_i=\overline\mu_ie^{t_i}$,
\[
(8-\alpha)\log\mu_i
=(8-\alpha)\log C_\alpha-(8-\alpha)\log\varepsilon
-2\log K(\xi_i)-2M_\alpha S_i+(8-\alpha)t_i.
\]
Insert this identity into $2A_i$. Exponentiation gives
\eqref{eq:balance}. Parameter derivatives follow from the
corresponding estimates in Lemmas \ref{lem:projection-expansion} and
\ref{lem:pair-interaction}. The center derivatives are divided by
$\mu_i$, which is the natural core normalization.
\end{proof}

\subsubsection{Weighted norms and parameter derivatives}

Throughout the reduction, $\tau$ is the number fixed in
\eqref{eq:tau-fixed-intro}. Define
\begin{equation}\label{eq:error-weight}
\Theta(x):=\sum_{i=1}^{m}
\frac{\mu_i^4}{(1+\mu_i|x-\xi_i|)^{8-\tau}},
\end{equation}
and
\begin{equation*}
\|h\|_{**}:=\sup_{x\in\Omega}\frac{|h(x)|}{\Theta(x)},
\qquad
\|\phi\|_*:=\|\phi\|_{L^\infty(\Omega)}.
\end{equation*}
For a parameter-dependent function $h$, we also use
\begin{equation*}
\begin{aligned}
\|h\|_{**,1}:={}&\|h\|_{**}
+\sum_{i=1}^{m}\|\partial_{t_i}h\|_{**}
+\sum_{i=1}^{m}\sum_{\ell=1}^{4}
\mu_i^{-1}\|\partial_{\xi_{i,\ell}}h\|_{**}.
\end{aligned}
\end{equation*}
The factor $\mu_i^{-1}$ is part of the definition. A displacement of order $\mu_i^{-1}$ is a displacement of order one in the $i$-th core.

Define
\begin{equation*}
\mathcal N_\varepsilon(v)(x)
:=\varepsilon^{8-\alpha}K(x)e^{v(x)}
\int_\Omega\frac{K(y)e^{v(y)}}{|x-y|^\alpha}\,dy
\end{equation*}
and
\begin{equation*}
R_{\boldsymbol t,\boldsymbol\xi}
:=\Delta^2W_{\boldsymbol t,\boldsymbol\xi}
-\mathcal N_\varepsilon(W_{\boldsymbol t,\boldsymbol\xi}).
\end{equation*}

We next estimate the residual and all parameter derivatives needed in the reduction.

\begin{proposition}\label{prop:residual}
Put
\[
E_{\boldsymbol t,\boldsymbol\xi}
:=R_{\boldsymbol t,\boldsymbol\xi}
-\sum_{i=1}^{m}(1-e^{2b_\alpha t_i})F_{\mu_i,\xi_i}.
\]
Uniformly for $|\boldsymbol t|\leq t_0$ and
$\boldsymbol\xi\in\mathcal C_{m,\eta}(\Omega)$,
\begin{equation}\label{eq:residual-estimate}
\|E_{\boldsymbol t,\boldsymbol\xi}\|_{**}\leq C\varepsilon^\tau,
\end{equation}
\begin{equation}\label{eq:residual-t-derivative}
\|\partial_{t_i}E_{\boldsymbol t,\boldsymbol\xi}\|_{**}
\leq C\varepsilon^\tau,
\end{equation}
and
\begin{equation}\label{eq:residual-xi-derivative}
\mu_i^{-1}\|\partial_{\xi_{i,\ell}}E_{\boldsymbol t,\boldsymbol\xi}\|_{**}
\leq C\varepsilon^\tau.
\end{equation}
The same estimate holds for every second parameter derivative, after
multiplying by one factor $\mu_i^{-1}$ for each derivative with
respect to $\xi_{i,\ell}$. The constants are uniform when two
derivatives act on different peaks. If
$|\boldsymbol t|\leq C\varepsilon^\tau|\log\varepsilon|$, then
\begin{equation}\label{eq:residual-small}
\|R_{\boldsymbol t,\boldsymbol\xi}\|_{**}
\leq C\varepsilon^\tau|\log\varepsilon|.
\end{equation}
\end{proposition}

\begin{proof}
Fix $0<\delta<1/4$. We split each peak ball into
\[
\mathcal C_i^{\rm in}:=
\{|x-\xi_i|\leq\mu_i^{-1}\varepsilon^{-\delta}\},
\qquad
\mathcal C_i^{\rm mid}:=
\{\mu_i^{-1}\varepsilon^{-\delta}<|x-\xi_i|<\eta/4\},
\]
and set
\[
\mathcal C^{\rm out}:=
\Omega\setminus\bigcup_{i=1}^{m}B_{\eta/4}(\xi_i).
\]
We choose $\delta$ so small that
\[
\delta(2+\tau)<1-\tau.
\]

Let $x=\xi_i+z/\mu_i$ in $\mathcal C_i^{\rm in}$. Formula \eqref{eq:local-W} gives
\begin{equation}\label{eq:core-exponential-expansion}
\begin{aligned}
e^{W(x)}
=e^{A_i}e^{U_{\mu_i,\xi_i}(x)}
\bigl[1+\mathcal E_i(x)\bigr],
\qquad
|\mathcal E_i(x)|
\leq C\mu_i^{-1}(1+|z|)+C\mu_i^{-2}(1+|z|^2).
\end{aligned}
\end{equation}
The right-hand side is $O(\varepsilon^{1-\delta})$. It is $O(\varepsilon^\tau(1+|z|)^\tau)$ after decreasing $\delta$. Taylor's formula can therefore be used without a truncation.

We split the Riesz potential into the $i$-th core, the other cores, and the complement. On the $i$-th core, the change of variables $y=\xi_i+w/\mu_i$, the $C^1$ expansion of $K$, and Lemma \ref{lem:weighted-Riesz} give
\begin{equation}\label{eq:self-potential-detailed}
\begin{aligned}
&\int_{B_{\eta/4}(\xi_i)}
\frac{K(y)e^{W(y)}}{|x-y|^\alpha}\,dy=K(\xi_i)e^{A_i}
\int_{\mathbb R^4}
\frac{e^{U_{\mu_i,\xi_i}(y)}}{|x-y|^\alpha}\,dy
\bigl[1+O(\varepsilon^\tau(1+|z|)^\tau)\bigr].
\end{aligned}
\end{equation}
The omitted tail satisfies
\[
\int_{\mathbb R^4\setminus B_{\eta/4}(\xi_i)}
\frac{e^{U_{\mu_i,\xi_i}(y)}}{|x-y|^\alpha}\,dy
\leq C\mu_i^{\alpha/2-4}.
\]
Relative to the whole-space self potential, this is $O(\mu_i^{-4})$ in the inner core.

For $j\ne i$, separation and Lemma \ref{lem:separated-core-estimates} give
\begin{equation}\label{eq:cross-potential-detailed}
\int_{B_{\eta/4}(\xi_j)}
\frac{K(y)e^{W(y)}}{|x-y|^\alpha}\,dy
\leq C\mu_j^{4-\alpha}.
\end{equation}
The self potential has size
$\mu_i^4(1+|z|)^{-\alpha}$. Hence the quotient of \eqref{eq:cross-potential-detailed} by the self term is bounded by
$C\varepsilon^\alpha(1+|z|)^\alpha$. This is absorbed by
$C\varepsilon^\tau(1+|z|)^\tau$ in the range where the latter is used. The exterior contribution is estimated in the same way.

Multiplying \eqref{eq:self-potential-detailed} by \eqref{eq:core-exponential-expansion}, and using \eqref{eq:bubble-nonlocal} and \eqref{eq:balance}, we obtain
\begin{equation}\label{eq:core-nonlinearity}
\mathcal N_\varepsilon(W)(x)
=e^{2b_\alpha t_i}F_{\mu_i,\xi_i}(x)
+O(\varepsilon^\tau\Theta(x))
\qquad\hbox{in }\mathcal C_i^{\rm in}.
\end{equation}
Since
$\Delta^2W=\sum_jF_{\mu_j,\xi_j}$, this proves
\eqref{eq:residual-estimate} in the inner core.

On $\mathcal C_i^{\rm mid}$, the Taylor factor in
\eqref{eq:core-exponential-expansion} need not be small. We do not use it. The projection expansion gives
\[
e^{W(x)}\leq C e^{A_i}e^{U_{\mu_i,\xi_i}(x)}
\exp(C|x-\xi_i|).
\]
The exponential factor is bounded. The bubble density satisfies
\[
F_{\mu_i,\xi_i}(x)
\leq C\mu_i^4(1+\mu_i|x-\xi_i|)^{-8}.
\]
Every first-order coefficient error is bounded by
$C|x-\xi_i|$. Thus
\[
|x-\xi_i|F_{\mu_i,\xi_i}(x)
\leq C\varepsilon^\tau
\frac{\mu_i^4}{(1+\mu_i|x-\xi_i|)^{8-\tau}}.
\]
The cross potentials are handled by \eqref{eq:cross-potential-detailed}. This proves the same estimate in the middle core.

On $\mathcal C^{\rm out}$,
\[
F_{\mu_i,\xi_i}(x)\leq C\mu_i^{-4}\leq C\varepsilon^4.
\]
The projected bubbles are bounded there. A peak contributes
$O(\mu_i^{4-\alpha})$ to the Riesz potential. Therefore
$\mathcal N_\varepsilon(W)=O(\varepsilon^4)$. On the other hand,
$\varepsilon^\tau\Theta(x)\geq c\varepsilon^4$ after the constants are adjusted. This proves \eqref{eq:residual-estimate} on the outer region.

For the differentiated estimates, write
$z=\mu_i(x-\xi_i)$. Direct differentiation gives
\begin{equation}\label{eq:bubble-parameter-derivative-bounds}
\left|\partial_{t_i}F_{\mu_i,\xi_i}(x)\right|
\leq CF_{\mu_i,\xi_i}(x),
\qquad
\mu_i^{-1}
\left|\partial_{\xi_{i,\ell}}F_{\mu_i,\xi_i}(x)\right|
\leq CF_{\mu_i,\xi_i}(x).
\end{equation}
The same bounds hold for $e^{U_{\mu_i,\xi_i}}$ after replacing the
right-hand side by the corresponding bubble weight. Differentiate
the self-potential formula under the integral sign. Lemma
\ref{lem:weighted-Riesz} controls every differentiated kernel
integral. In the inner region, the additional factors are bounded by
$C(1+|z|)$ and are absorbed by the choice of $\delta$. In the middle
region, \eqref{eq:bubble-parameter-derivative-bounds} and the
algebraic decay give the same weight $\Theta$. In the outer region,
all differentiated bubble terms are $O(\varepsilon^4)$ after the
normalization by $\mu_i^{-1}$ in a position derivative. This proves
\eqref{eq:residual-t-derivative} and
\eqref{eq:residual-xi-derivative}. Repeating the differentiation once
more produces at most quadratic polynomials in the core variables.
The inner--middle--outer decomposition above, with the same choice of
$\delta$, absorbs these factors. Each center derivative is normalized
by the corresponding $\mu_i^{-1}$, and derivatives falling on two
different cores are smaller by separation. This proves the asserted
second-derivative bounds. Finally,
$|1-e^{2b_\alpha t_i}|\leq C|t_i|$ for small $t_i$.
Equation \eqref{eq:residual-small} follows.
\end{proof}

\section{The projected linear problem}

The space $\mathcal H$ and its norm are defined in
\eqref{eq:energy-space}.

\subsection{Approximate kernel functions and the capacitary dilation tail}

Let $\chi\in C_c^\infty(B_{\eta/3}(0))$ satisfy
$\chi=1$ in $B_{\eta/4}(0)$. Put
\begin{equation*}
\chi_i(x):=\chi(x-\xi_i),
\end{equation*}
and
\begin{equation*}
Z_{i,k}(x):=Z_k(\mu_i(x-\xi_i)),
\qquad 0\leq k\leq4.
\end{equation*}
The orthogonality conditions are
\begin{equation}\label{eq:orthogonality}
\int_\Omega\chi_iF_{\mu_i,\xi_i}Z_{i,k}\phi\,dx=0,
\qquad 1\leq i\leq m,
\quad 0\leq k\leq4.
\end{equation}

We next define the fourth-order capacity used for the dilation
mode. Let $0<\rho<\eta/8$ and
$\operatorname{dist}(\xi,\partial\Omega)>2\rho$. Set
\begin{equation*}
\begin{aligned}
\mathcal A_{\rho,\xi}:=
\{v\in H^2(\Omega\setminus B_\rho(\xi)):\;&
v=-b_\alpha,\quad \partial_\nu v=0
\quad\hbox{on }\partial B_\rho(\xi),\ v=0\quad\hbox{on }\partial\Omega\}.
\end{aligned}
\end{equation*}
All traces are understood in the Sobolev sense.

The dilation direction requires the critical biharmonic capacity of a small ball.

\begin{lemma}\label{lem:biharmonic-capacity}
The minimization problem
\begin{equation}\label{eq:capacity-minimization}
\inf_{v\in\mathcal A_{\rho,\xi}}
\int_{\Omega\setminus B_\rho(\xi)}|\Delta v|^2\,dx
\end{equation}
has a unique minimizer $V_{\rho,\xi}$. It satisfies
\begin{equation}\label{eq:capacity-Euler}
\begin{cases}
\Delta^2V_{\rho,\xi}=0
&\text{in }\Omega\setminus\overline{B_\rho(\xi)},\\
V_{\rho,\xi}=-b_\alpha,\quad
\partial_\nu V_{\rho,\xi}=0
&\text{on }\partial B_\rho(\xi),\\
V_{\rho,\xi}=\Delta V_{\rho,\xi}=0
&\text{on }\partial\Omega.
\end{cases}
\end{equation}
Uniformly for $\xi$ in a compact subset of $\Omega$,
\begin{equation}\label{eq:capacity-asymptotic}
\int_{\Omega\setminus B_\rho(\xi)}
|\Delta V_{\rho,\xi}|^2\,dx
=\frac{8\pi^2b_\alpha^2}{|\log\rho|}
+O(|\log\rho|^{-2}).
\end{equation}
Fix $0<r_0<R$ and assume $B_{4R}(\xi)\Subset\Omega$. Then
there is a smooth function $\mathcal B_\xi$ such that
\begin{equation}\label{eq:capacity-profile-fixed}
V_{\rho,\xi}(x)
=\frac{b_\alpha\log|x-\xi|+\mathcal B_\xi(x)}
{|\log\rho|}
+O(|\log\rho|^{-2})
\end{equation}
in $C^3(B_R(\xi)\setminus B_{r_0}(\xi))$. On the inner transition
annulus,
\begin{equation}\label{eq:capacity-boundary-layer}
\sup_{\rho\leq|x-\xi|\leq2\rho}
\left(
|V_{\rho,\xi}(x)+b_\alpha|
+\rho|\nabla V_{\rho,\xi}(x)|
+\rho^2|D^2V_{\rho,\xi}(x)|
\right)
\leq\frac{C}{|\log\rho|}.
\end{equation}
\end{lemma}

\begin{proof}
The functional in \eqref{eq:capacity-minimization} is coercive on
the affine class. Indeed, the difference of two admissible functions
has zero Dirichlet trace on both boundary components and zero normal
trace on the inner sphere. If its Laplacian vanishes, the function is
harmonic with zero boundary trace and is therefore zero. The direct
method gives a minimizer. Strict convexity gives uniqueness.

Interior variations give the biharmonic equation. Variations whose
outer Dirichlet trace vanishes but whose outer normal derivative is
arbitrary give $\Delta V_{\rho,\xi}=0$ on $\partial\Omega$. This proves
\eqref{eq:capacity-Euler}.

Fix $R>0$ with $B_{4R}(\xi)\Subset\Omega$. A radial biharmonic
function on $B_R(\xi)\setminus B_\rho(\xi)$ has the form
\begin{equation}\label{eq:radial-biharmonic-form}
A+B\log r+Cr^2+Dr^{-2}.
\end{equation}
The two inner conditions imply
\begin{equation}\label{eq:radial-inner-relations}
D=\frac12B\rho^2+C\rho^4,
\qquad
A=-b_\alpha-B\log\rho-\frac12B-2C\rho^2.
\end{equation}
For $(a,d)\in\mathbb R^2$, let
$\mathcal E_{\xi,R}(a,d)$ be the least exterior energy among
functions on $\Omega\setminus B_R(\xi)$ with Cauchy data
$w=a$, $\partial_\nu w=d$ on $\partial B_R(\xi)$ and
$w=0$ on $\partial\Omega$. The minimizer is biharmonic and satisfies
the natural condition $\Delta w=0$ on $\partial\Omega$.
Elliptic estimates show that
$\mathcal E_{\xi,R}$ is a positive quadratic form in $(a,d)$ whose
coefficients are uniformly bounded for $\xi$ in a compact set.

Attach the radial function \eqref{eq:radial-biharmonic-form} to this
exterior minimizer by matching its value and radial derivative at
$r=R$. Minimize the sum of the annular and exterior energies subject
to \eqref{eq:radial-inner-relations}. The Euler equations form a
uniformly invertible linear system. With
$L_\rho=|\log\rho|$, they give
\begin{equation}\label{eq:capacity-coefficients}
B=\frac{b_\alpha}{L_\rho}+O(L_\rho^{-2}),
\quad
C=O(L_\rho^{-1}),
\quad
D=\frac{b_\alpha\rho^2}{2L_\rho}
+O(\rho^2L_\rho^{-2}).
\end{equation}
The resulting piecewise function has matching value and normal
derivative at $r=R$, hence belongs to $H^2(\Omega\setminus
B_\rho(\xi))$ and is admissible. Its logarithmic
part has energy
\begin{equation}\label{eq:capacity-upper-log}
\begin{aligned}
4B^2|\mathbb S^3|\int_\rho^R\frac{dr}{r}
=8\pi^2B^2\log(R/\rho)=\frac{8\pi^2b_\alpha^2}{|\log\rho|}
+O(|\log\rho|^{-2}).
\end{aligned}
\end{equation}
The remaining radial terms, their cross terms, and the exterior
correction are of smaller order. More precisely, their total
contribution is
\begin{equation*}
O(|\log\rho|^{-2}).
\end{equation*}
This proves the upper bound.

For the lower bound, take the spherical mean
$\overline V_{\rho,\xi}$ on
$B_R(\xi)\setminus B_\rho(\xi)$. Orthogonality of spherical
harmonics yields
\begin{equation*}
\int_{B_R(\xi)\setminus B_\rho(\xi)}
|\Delta V_{\rho,\xi}|^2\,dx
\geq
\int_{B_R(\xi)\setminus B_\rho(\xi)}
|\Delta\overline V_{\rho,\xi}|^2\,dx.
\end{equation*}
For fixed Cauchy data of the spherical mean at $r=R$, the exterior
part of the minimizer has energy at least
$\mathcal E_{\xi,R}(a,d)$. Thus the full energy is bounded below by
the same one-dimensional quadratic minimization used for the
comparison function. Its Euler system is the system above. Hence its
logarithmic coefficient satisfies the first formula in
\eqref{eq:capacity-coefficients}. Substitution in
\eqref{eq:capacity-upper-log} gives the matching lower bound. This
proves \eqref{eq:capacity-asymptotic}.

On every set with $|x-\xi|\geq r_0$, the term
$D|x-\xi|^{-2}$ is $O(\rho^2|\log\rho|^{-1})$. The remaining
coefficients and the exterior biharmonic correction give
\eqref{eq:capacity-profile-fixed} by interior estimates. On
$\rho\leq r\leq2\rho$, use
\eqref{eq:radial-inner-relations}--\eqref{eq:capacity-coefficients}.
The $Dr^{-2}$ term must be retained there. Direct differentiation
gives \eqref{eq:capacity-boundary-layer}. Uniformity follows from
compactness of the set of centers.
\end{proof}

Extend $V_{\rho,\xi}$ by the constant $-b_\alpha$ inside
$B_\rho(\xi)$. The value and the normal derivative match, so the
extension belongs to $H^2(\Omega)$.

Choose
\begin{equation}\label{eq:R-epsilon-choice}
R_\varepsilon\longrightarrow+\infty,
\qquad
\log R_\varepsilon=o(|\log\varepsilon|),
\qquad
\varepsilon R_\varepsilon\longrightarrow0,
\qquad
\rho_{i,\varepsilon}=\frac{R_\varepsilon}{\mu_i}.
\end{equation}
All estimates below are uniform for choices satisfying
\eqref{eq:R-epsilon-choice}. The explicit logarithmic choice of $R_\varepsilon$ will be fixed
in Lemma \ref{lem:scale-tangent-exterior}.
Let $\zeta\in C^\infty([0,\infty))$ satisfy
$\zeta=1$ on $[0,1]$ and $\zeta=0$ on $[2,\infty)$. Define
\begin{equation}\label{eq:dilation-test}
\widehat Z_{i,0}(x)
:=\zeta\left(\frac{|x-\xi_i|}{\rho_{i,\varepsilon}}\right)
Z_{i,0}(x)
+\left[1-\zeta\left(\frac{|x-\xi_i|}{\rho_{i,\varepsilon}}\right)\right]
V_{\rho_{i,\varepsilon},\xi_i}(x).
\end{equation}
For $1\leq k\leq4$, set
\begin{equation}\label{eq:translation-test}
\widehat Z_{i,k}:=\chi_iZ_{i,k}.
\end{equation}

The capacity profile is combined with the kernel functions in the following test estimates.

\begin{lemma}\label{lem:test-functions}
The functions in \eqref{eq:dilation-test}--\eqref{eq:translation-test}
belong to $\mathcal H$ and satisfy
\begin{equation}\label{eq:test-H-bound}
\|\widehat Z_{i,k}\|_{\mathcal H}
+\|\widehat Z_{i,k}\|_{L^\infty(\Omega)}
\leq C.
\end{equation}
Let $L_\varepsilon$ be defined in \eqref{eq:linearized-W}. Then
\begin{equation}\label{eq:test-L-small}
\sup_{\|v\|_{\mathcal H}=1}
\left|
\int_\Omega vL_\varepsilon\widehat Z_{i,k}\,dx
\right|
=o(|\log\varepsilon|^{-1/2})
\end{equation}
for $0\leq k\leq4$. Moreover,
\begin{equation}\label{eq:dilation-test-energy}
\int_\Omega|\Delta\widehat Z_{i,0}|^2\,dx\leq C,
\qquad
\int_{\Omega\setminus B_{2\rho_{i,\varepsilon}}(\xi_i)}
|\Delta\widehat Z_{i,0}|^2\,dx
=\frac{8\pi^2b_\alpha^2}{|\log\varepsilon|}
+o(|\log\varepsilon|^{-1}).
\end{equation}
If $i\ne j$, then
\begin{equation}\label{eq:dilation-test-cross-energy}
\int_\Omega
\Delta\widehat Z_{i,0}\Delta\widehat Z_{j,0}\,dx
=O(|\log\varepsilon|^{-2}).
\end{equation}
\end{lemma}

\begin{proof}
For $1\leq k\leq4$, the translation kernels satisfy
$Z_k(z)=O(|z|^{-1})$ and
$\nabla^jZ_k(z)=O(|z|^{-1-j})$. The cutoff commutator is supported
where $|x-\xi_i|\asymp1$. The bubble coefficients are
$O(\varepsilon^4)$ on this set. Rescaling and the limiting kernel
equation give \eqref{eq:test-H-bound} and
\eqref{eq:test-L-small} for translations.

Consider the dilation mode. On the transition annulus
$\rho_{i,\varepsilon}<|x-\xi_i|<2\rho_{i,\varepsilon}$,
\begin{equation}\label{eq:dilation-core-mismatch}
Z_{i,0}+b_\alpha=O(R_\varepsilon^{-2}),
\qquad
\rho_{i,\varepsilon}|\nabla Z_{i,0}|
+\rho_{i,\varepsilon}^2|D^2Z_{i,0}|
=O(R_\varepsilon^{-2}).
\end{equation}
By \eqref{eq:capacity-boundary-layer},
\begin{equation}\label{eq:dilation-capacity-mismatch}
|V_{\rho_{i,\varepsilon},\xi_i}+b_\alpha|
+\rho_{i,\varepsilon}|\nabla V_{\rho_{i,\varepsilon},\xi_i}|
+\rho_{i,\varepsilon}^2
|D^2V_{\rho_{i,\varepsilon},\xi_i}|
=O(|\log\varepsilon|^{-1}).
\end{equation}
Two derivatives of the cutoff produce
$\rho_{i,\varepsilon}^{-2}$. The transition annulus has volume
$O(\rho_{i,\varepsilon}^4)$. Hence the interpolation energy is
$O(|\log\varepsilon|^{-2})$. Lemma
\ref{lem:biharmonic-capacity} gives
\eqref{eq:test-H-bound} and \eqref{eq:dilation-test-energy}.

Inside $B_{\rho_{i,\varepsilon}}(\xi_i)$, rescaling gives
$\mathcal L_\infty Z_0=0$. Proposition \ref{prop:residual} bounds
the coefficient error by $O(\varepsilon^\tau)$. Outside
$B_{2\rho_{i,\varepsilon}}(\xi_i)$, the capacity profile is
biharmonic. The two nonlocal coefficients have uniformly bounded
total mass and are concentrated in the bubble cores. H\"older's
inequality, the Hardy--Littlewood--Sobolev inequality, and
\eqref{eq:capacity-asymptotic} show that their pairing with a unit
vector of $\mathcal H$ is
$O(|\log\varepsilon|^{-1})+O(\varepsilon^\tau)$. The transition
commutator has the same bound by
\eqref{eq:dilation-core-mismatch}--\eqref{eq:dilation-capacity-mismatch}.
This proves \eqref{eq:test-L-small}.

For $i\ne j$, each capacity tail has amplitude
$O(|\log\varepsilon|^{-1})$ on a fixed neighborhood of the other
center. Their core supports are disjoint. Green's identity and
\eqref{eq:capacity-profile-fixed} give
\eqref{eq:dilation-test-cross-energy}.
\end{proof}

The linearized operator at $W=W_{\boldsymbol t,\boldsymbol\xi}$ is
\begin{equation}\label{eq:linearized-W}
\begin{aligned}
L_\varepsilon\phi(x)
={}&\Delta^2\phi(x)
-\varepsilon^{8-\alpha}K(x)e^{W(x)}\phi(x)
\int_\Omega\frac{K(y)e^{W(y)}}{|x-y|^\alpha}\,dy\\
&-\varepsilon^{8-\alpha}K(x)e^{W(x)}
\int_\Omega\frac{K(y)e^{W(y)}\phi(y)}{|x-y|^\alpha}\,dy.
\end{aligned}
\end{equation}
We solve
\begin{equation}\label{eq:projected-linear-problem}
\begin{cases}
\displaystyle
L_\varepsilon\phi
=h+\displaystyle\sum_{i=1}^{m}\sum_{k=0}^{4}
 c_{i,k}\chi_iF_{\mu_i,\xi_i}Z_{i,k}
&\text{in }\Omega,\\
\phi=\Delta\phi=0&\text{on }\partial\Omega,\\
\phi\text{ satisfies }\eqref{eq:orthogonality}.&
\end{cases}
\end{equation}

The main linear estimate is the uniform invertibility of the projected problem.

\begin{proposition}\label{prop:linear-estimate}
There are $\varepsilon_0>0$ and $C>0$ such that every solution of
\eqref{eq:projected-linear-problem} satisfies
\begin{equation}\label{eq:linear-estimate}
\|\phi\|_*+\|\phi\|_\mathcal H
+\sum_{i=1}^{m}\sum_{k=0}^{4}|c_{i,k}|
\leq C\|h\|_{**}
\end{equation}
for $0<\varepsilon<\varepsilon_0$, $|\boldsymbol t|\leq t_0$, and
$\boldsymbol\xi\in\mathcal C_{m,\eta}(\Omega)$.
\end{proposition}

\begin{proof}
We first estimate the multipliers. Test
\eqref{eq:projected-linear-problem} with
$\widehat Z_{j,\ell}$. The operator $L_\varepsilon$ is symmetric.
Hence
\begin{equation*}
\int_\Omega L_\varepsilon\phi\,\widehat Z_{j,\ell}\,dx
=
\int_\Omega\phi L_\varepsilon\widehat Z_{j,\ell}\,dx.
\end{equation*}
Lemma \ref{lem:test-functions} gives
\begin{equation*}
\left|
\int_\Omega\phi L_\varepsilon\widehat Z_{j,\ell}\,dx
\right|
=o(|\log\varepsilon|^{-1/2})\|\phi\|_{\mathcal H}.
\end{equation*}
Since
\begin{equation}\label{eq:theta-L1}
\int_\Omega\Theta(x)\,dx\leq C,
\end{equation}
one also has
\begin{equation*}
\left|\int_\Omega h\widehat Z_{j,\ell}\,dx\right|
\leq C\|h\|_{**}.
\end{equation*}
The multiplier matrix is
\begin{equation}\label{eq:linear-multiplier-matrix}
\mathcal A_{(i,k),(j,\ell)}
:=
\int_\Omega
\chi_iF_{\mu_i,\xi_i}Z_{i,k}\widehat Z_{j,\ell}\,dx.
\end{equation}
Its off-diagonal peak blocks tend to zero. On a diagonal block,
rescaling and $R_\varepsilon\to\infty$ give
\begin{equation*}
\mathcal A_{(i,k),(i,\ell)}
=
\int_{\mathbb R^4}FZ_kZ_\ell\,dz+o(1).
\end{equation*}
Lemma \ref{lem:kernel-Gram} implies uniform invertibility. Therefore
\begin{equation}\label{eq:coefficient-preestimate}
\sum_{i,k}|c_{i,k}|
\leq C\|h\|_{**}
+o(|\log\varepsilon|^{-1/2})\|\phi\|_{\mathcal H}.
\end{equation}

We next derive an energy bound in terms of $\|\phi\|_*$.
Testing \eqref{eq:projected-linear-problem} with $\phi$ gives
\begin{equation*}
\|\phi\|_{\mathcal H}^2
=\mathcal B_\varepsilon(\phi)
+\int_\Omega h\phi\,dx
+\sum_{i,k}c_{i,k}
\int_\Omega\chi_iF_{\mu_i,\xi_i}Z_{i,k}\phi\,dx,
\end{equation*}
where $\mathcal B_\varepsilon(\phi)$ is the sum of the two nonlocal
linearized terms. The total nonlinear mass is uniformly bounded.
Thus the Hardy--Littlewood--Sobolev inequality gives
\begin{equation}\label{eq:linear-nonlocal-energy-bound}
|\mathcal B_\varepsilon(\phi)|
\leq C\|\phi\|_*^2.
\end{equation}
By \eqref{eq:theta-L1},
\begin{equation}\label{eq:linear-source-energy-bound}
\left|\int_\Omega h\phi\,dx\right|
\leq C\|h\|_{**}\|\phi\|_*.
\end{equation}
The multiplier term vanishes exactly by the orthogonality
conditions \eqref{eq:orthogonality}. Combining
\eqref{eq:linear-nonlocal-energy-bound} and
\eqref{eq:linear-source-energy-bound}, and then using Young's
inequality, yields
\begin{equation}\label{eq:linear-H-from-Linfty}
\|\phi\|_{\mathcal H}
\leq C\bigl(\|\phi\|_*+\|h\|_{**}\bigr).
\end{equation}

Suppose that the $L^\infty$ estimate is false. Then there are
$\varepsilon_n\to0$, admissible parameters, and solutions such that
\begin{equation}\label{eq:linear-contradiction-normalization}
\|\phi_n\|_*=1,
\qquad
\|h_n\|_{**}\to0.
\end{equation}
Equations \eqref{eq:linear-H-from-Linfty} and
\eqref{eq:coefficient-preestimate} give a uniform $\mathcal H$ bound
and $c_{i,k,n}\to0$.

Rescale around the $i$-th center:
\begin{equation*}
\widetilde\phi_{i,n}(z)
:=\phi_n(\xi_{i,n}+z/\mu_{i,n}).
\end{equation*}
Local $W^{4,p}$ estimates, valid for every finite $p$, give
convergence in $C^{3,\beta}_{\rm loc}(\mathbb R^4)$ for every
$0<\beta<1$, after passage to a subsequence. Lemma
\ref{lem:weighted-Riesz} permits passage to the nonlocal terms. The
limit is bounded and solves
$\mathcal L_\infty\widetilde\phi_i=0$. The orthogonality conditions
give
\begin{equation*}
\int_{\mathbb R^4}FZ_k\widetilde\phi_i\,dz=0,
\qquad 0\leq k\leq4.
\end{equation*}
Theorem \ref{thm:limit-input} and Lemma \ref{lem:kernel-Gram} imply
$\widetilde\phi_i=0$.

On compact subsets away from the limiting centers, the Green
representation and the uniform $\mathcal H$ bound give compactness in
$C^{3,\beta}$. The source measures are uniformly bounded. Any weak
limit is biharmonic away from the centers and may contain only point
masses at those centers. A nonzero point mass would generate a
nonzero multiple of $-(8\pi^2)^{-1}\log|x-\xi_i|$, which is
incompatible with the uniform $L^\infty$ normalization. Thus the
limit solves the homogeneous Navier problem on all of $\Omega$ and is
zero. Lemma \ref{lem:remove-defects} then excludes maxima escaping
through the annular necks. It follows that
$\|\phi_n\|_\infty\to0$, contrary to
\eqref{eq:linear-contradiction-normalization}. Thus
\begin{equation*}
\|\phi\|_*\leq C\|h\|_{**}.
\end{equation*}
Finally, \eqref{eq:linear-H-from-Linfty} and
\eqref{eq:coefficient-preestimate} give
\eqref{eq:linear-estimate}.
\end{proof}

The a priori estimate yields solvability and smooth parameter dependence.

\begin{proposition}\label{prop:linear-solvability}
For small $\varepsilon$, problem \eqref{eq:projected-linear-problem} has a unique solution. The solution operator is bounded from the $\|\cdot\|_{**}$ space to
$\mathcal H\cap L^\infty(\Omega)$. If the source and the parameters are $C^2$, then the solution, the multipliers, and the orthogonal projections are $C^2$ in
$(\boldsymbol t,\boldsymbol\xi)$. After the natural rescaling of position derivatives, the first and second parameter derivatives satisfy the estimate in Proposition \ref{prop:linear-estimate}.
\end{proposition}

\begin{proof}
For fixed parameters, let $\mathcal X(\boldsymbol t,\boldsymbol\xi)$
be the subspace defined by \eqref{eq:orthogonality}. The Navier
biharmonic operator is an isomorphism from $\mathcal H$ to
$\mathcal H'$. The two nonlocal terms are compact from $\mathcal H$
to $\mathcal H'$. This follows from the compact embedding
$\mathcal H\hookrightarrow L^q(\Omega)$ for every finite $q$ and
the HLS inequality. Hence the projected system on
$\mathcal X(\boldsymbol t,\boldsymbol\xi)\times\mathbb R^{5m}$ is
Fredholm of index zero. Proposition \ref{prop:linear-estimate}
excludes a nonzero homogeneous solution. Fredholm's alternative gives
existence and uniqueness.

We record the parameter argument. Let $p$ be one of the $t_i$ or a
normalized position coordinate. Differentiate the equation and the
orthogonality conditions. The derivative $(\partial_p\phi,
\partial_pc)$ solves the same projected operator with source
\begin{equation*}
\partial_ph-(\partial_pL_\varepsilon)\phi
+\text{terms produced by the differentiated constraints}.
\end{equation*}
Proposition \ref{prop:residual}, the bubble derivative bounds
\eqref{eq:bubble-parameter-derivative-bounds}, and Lemma
\ref{lem:weighted-Riesz} show that this source has uniformly bounded
$\|\cdot\|_{**}$ norm. Proposition
\ref{prop:linear-estimate} therefore controls the first derivative.

Differentiate once more. The second derivative solves the same
projected operator. Its source contains
\begin{equation*}
\partial_{pq}^2h
-(\partial_{pq}^2L_\varepsilon)\phi
-(\partial_pL_\varepsilon)\partial_q\phi
-(\partial_qL_\varepsilon)\partial_p\phi
\end{equation*}
and the twice differentiated constraints. The first-derivative
estimate and the same weighted bounds control every term. A position
derivative is multiplied by $\mu_i^{-1}$, and a second position
derivative by the product of the corresponding factors. Applying
Proposition \ref{prop:linear-estimate} again gives the second
derivative estimate.

Equivalently, identify the moving orthogonal spaces with a fixed
reference space by the uniformly invertible Gram projection. The
projected operator is a $C^2$ family of isomorphisms. The standard
inverse-map theorem then gives the asserted $C^2$ dependence.
\end{proof}

\section{The nonlinear projected problem}

Put
\begin{equation*}
Q_\varepsilon(\phi)
:=\mathcal N_\varepsilon(W+\phi)
-\mathcal N_\varepsilon(W)-D\mathcal N_\varepsilon(W)[\phi].
\end{equation*}
We solve
\begin{equation}\label{eq:projected-nonlinear-problem}
\begin{cases}
\displaystyle
L_\varepsilon\phi
=-R_{\boldsymbol t,\boldsymbol\xi}+Q_\varepsilon(\phi)
+\displaystyle\sum_{i=1}^{m}\sum_{k=0}^{4}
 c_{i,k}\chi_iF_{\mu_i,\xi_i}Z_{i,k}
&\text{in }\Omega,\\
\phi=\Delta\phi=0&\text{on }\partial\Omega,\\
\phi\text{ satisfies }\eqref{eq:orthogonality}.&
\end{cases}
\end{equation}

We record the quadratic estimates for the nonlinear remainder.

\begin{lemma}\label{lem:nonlinear-estimates}
There is $C>0$ such that, if $\|\phi\|_*+\|\psi\|_*\leq1$, then
\begin{equation}\label{eq:quadratic-estimate}
\|Q_\varepsilon(\phi)\|_{**}
\leq C\|\phi\|_*^2,
\end{equation}
and
\begin{equation}\label{eq:Lipschitz-estimate}
\|Q_\varepsilon(\phi)-Q_\varepsilon(\psi)\|_{**}
\leq C(\|\phi\|_*+\|\psi\|_*)\|\phi-\psi\|_*.
\end{equation}
The same estimates hold after one or two parameter derivatives, with the natural scaling for the position variables.
\end{lemma}

\begin{proof}
Write
\[
E(s)=e^s-1-s.
\]
For $|s|,|t|\leq1$,
\begin{equation}\label{eq:scalar-exponential-remainders}
|E(s)|\leq Cs^2,
\qquad
|E(s)-E(t)|\leq C(|s|+|t|)|s-t|.
\end{equation}
Using the symmetry of the two exponential factors, we have
\begin{equation}\label{eq:nonlinear-remainder-exact}
\begin{aligned}
\mathcal Q_\varepsilon(\phi)(x)
={}&\varepsilon^{8-\alpha}K(x)e^{W(x)}
\int_\Omega\frac{K(y)e^{W(y)}}{|x-y|^\alpha}\\
&\times\bigl[E(\phi(x))+E(\phi(y))
+(e^{\phi(x)}-1)(e^{\phi(y)}-1)\bigr]dy.
\end{aligned}
\end{equation}
Thus, for $\|\phi\|_*\leq1$,
\[
|\mathcal Q_\varepsilon(\phi)(x)|
\leq C\|\phi\|_*^2\mathcal N_\varepsilon(W)(x).
\]
Proposition \ref{prop:residual} and
$\Delta^2W=\sum_iF_{\mu_i,\xi_i}$ give
\begin{equation*}
\mathcal N_\varepsilon(W)(x)\leq C\Theta(x).
\end{equation*}
This proves \eqref{eq:quadratic-estimate}.

Subtract two copies of \eqref{eq:nonlinear-remainder-exact}. Apply the
second inequality in \eqref{eq:scalar-exponential-remainders} and
$|e^s-e^t|\leq C|s-t|$. We obtain
\[
|\mathcal Q_\varepsilon(\phi)-\mathcal Q_\varepsilon(\psi)|
\leq C(\|\phi\|_*+\|\psi\|_*)
\|\phi-\psi\|_*\Theta.
\]
This is \eqref{eq:Lipschitz-estimate}.

A parameter derivative of \eqref{eq:nonlinear-remainder-exact}
contains a derivative of $W$, of the scale, or of the correction.
After multiplying a center derivative by $\mu_i^{-1}$, all bubble
derivatives are bounded by a fixed polynomial in the core variable
times the weight in \eqref{eq:error-weight}. The first and second
parameter estimates in Propositions \ref{prop:residual} and
\ref{prop:linear-solvability} therefore apply term by term. This proves
the differentiated estimates.
\end{proof}

The projected nonlinear problem can now be solved uniformly.

\begin{proposition}\label{prop:nonlinear-correction}
There are $t_0>0$, $\varepsilon_0>0$, and $C>0$ such that, for every admissible configuration and every $|\boldsymbol t|\leq t_0$, problem \eqref{eq:projected-nonlinear-problem} has a unique solution
$\phi_\varepsilon(\boldsymbol t,\boldsymbol\xi)$ in the ball
\[
\|\phi\|_*+\|\phi\|_\mathcal H
\leq C(\varepsilon^\tau+|\boldsymbol t|).
\]
It satisfies
\begin{equation}\label{eq:correction-estimate}
\|\phi_\varepsilon\|_*+\|\phi_\varepsilon\|_\mathcal H
+\sum_{i,k}|c_{i,k}|
\leq C(\varepsilon^\tau+|\boldsymbol t|).
\end{equation}
If
$|\boldsymbol t|\leq C\varepsilon^\tau|\log\varepsilon|$, then
\begin{equation}\label{eq:correction-sharp}
\|\phi_\varepsilon\|_*+\|\phi_\varepsilon\|_\mathcal H
+\sum_{i,k}|c_{i,k}|
\leq C\varepsilon^\tau|\log\varepsilon|.
\end{equation}
The map $(\boldsymbol t,\boldsymbol\xi)\mapsto\phi_\varepsilon$ is $C^2$. Its first and second derivatives satisfy the bounds obtained by differentiating \eqref{eq:correction-estimate}; every position derivative is measured after multiplication by the corresponding factor $\mu_i^{-1}$.
\end{proposition}

\begin{proof}
Let $T_\varepsilon$ denote the solution operator in Proposition
\ref{prop:linear-solvability}. Put
\begin{equation*}
\mathscr T_\varepsilon(\phi)
:=T_\varepsilon[-R_{\boldsymbol t,\boldsymbol\xi}
+Q_\varepsilon(\phi)].
\end{equation*}
Choose $C_0>2C$, where $C$ is the constant in Propositions
\ref{prop:linear-estimate} and \ref{prop:residual}. Set
\begin{equation*}
\rho_\varepsilon
=C_0(\varepsilon^\tau+|\boldsymbol t|)
\end{equation*}
and
\[
\mathbb B_\varepsilon
=\{\phi\in\mathcal H:\phi\text{ satisfies }
\eqref{eq:orthogonality},\
\|\phi\|_*+\|\phi\|_\mathcal H\leq\rho_\varepsilon\}.
\]
For $\phi\in\mathbb B_\varepsilon$, Propositions
\ref{prop:linear-estimate}, \ref{prop:residual}, and Lemma
\ref{lem:nonlinear-estimates} give
\begin{equation*}
\begin{aligned}
\|\mathscr T_\varepsilon(\phi)\|_*
+\|\mathscr T_\varepsilon(\phi)\|_\mathcal H
\leq C(\varepsilon^\tau+|\boldsymbol t|)
+C\rho_\varepsilon^2\leq\rho_\varepsilon,
\end{aligned}
\end{equation*}
provided $t_0$ and then $\varepsilon_0$ are small. For
$\phi,\psi\in\mathbb B_\varepsilon$,
\begin{equation}\label{eq:fixed-point-contraction}
\begin{aligned}
\|\mathscr T_\varepsilon(\phi)
-\mathscr T_\varepsilon(\psi)\|_*
+\|\mathscr T_\varepsilon(\phi)
-\mathscr T_\varepsilon(\psi)\|_\mathcal H
\leq C\rho_\varepsilon\|\phi-\psi\|_*.
\end{aligned}
\end{equation}
After reducing $t_0$ and $\varepsilon_0$, the right-hand coefficient
is less than $1/2$. Banach's fixed-point theorem gives a unique
solution in $\mathbb B_\varepsilon$. The multiplier estimate is part
of Proposition \ref{prop:linear-estimate}. This proves
\eqref{eq:correction-estimate}.

If
$|\boldsymbol t|\leq C\varepsilon^\tau|\log\varepsilon|$, use
\eqref{eq:residual-small} instead of the first residual bound. The
same calculation, with radius
$C_0\varepsilon^\tau|\log\varepsilon|$, gives
\eqref{eq:correction-sharp}.

We finally prove parameter regularity. Define
\[
\mathscr G(\phi,\boldsymbol t,\boldsymbol\xi)
=\phi-T_\varepsilon[-R_{\boldsymbol t,\boldsymbol\xi}
+Q_\varepsilon(\phi)].
\]
At the fixed point,
\begin{equation*}
D_\phi\mathscr G
=I-T_\varepsilon DQ_\varepsilon(\phi).
\end{equation*}
By \eqref{eq:fixed-point-contraction}, this operator is invertible and
its inverse is the convergent Neumann series. Propositions
\ref{prop:residual}, \ref{prop:linear-solvability}, and Lemma
\ref{lem:nonlinear-estimates} show that $\mathscr G$ is $C^2$ in all
variables after the natural center normalization. The implicit
function theorem gives $C^2$ dependence. Differentiating
$\mathscr G=0$ once and twice and applying
$(D_\phi\mathscr G)^{-1}$ gives the stated derivative estimates.
\end{proof}

\section{Finite-dimensional reduction and proof of the existence theorem}

Define
\begin{equation}\label{eq:reduced-energy}
\widetilde J_\varepsilon(\boldsymbol t,\boldsymbol\xi)
:=J_\varepsilon(W_{\boldsymbol t,\boldsymbol\xi}
+\phi_\varepsilon(\boldsymbol t,\boldsymbol\xi)).
\end{equation}

\subsection{Energy expansion for the approximate solution}

We begin the finite-dimensional reduction with the energy of the approximate solution.

\begin{proposition}\label{prop:energy-expansion}
There is a function
\[
\mathcal R_\varepsilon^{\rm app}
\in C^2([-t_0,t_0]^m\times\mathcal C_{m,\eta}(\Omega))
\]
such that
\begin{equation}\label{eq:energy-expansion}
\begin{aligned}
J_\varepsilon(W_{\boldsymbol t,\boldsymbol\xi})
={}&A_{m,\varepsilon}-M_\alpha\mathcal F_m(\boldsymbol\xi)+\sum_{i=1}^{m}\left[
M_\alpha b_\alpha t_i
-\frac{M_\alpha}{2}(e^{2b_\alpha t_i}-1)
\right]
+\mathcal R_\varepsilon^{\rm app}(\boldsymbol t,\boldsymbol\xi),
\end{aligned}
\end{equation}
where
\begin{equation}\label{eq:energy-constant}
A_{m,\varepsilon}
=m\left[
\frac{B_\alpha}{2}
+\frac{M_\alpha b_\alpha}{2}\log C_\alpha
-M_\alpha b_\alpha\log\varepsilon
-\frac{M_\alpha}{2}
\right].
\end{equation}
Uniformly for $|\boldsymbol t|\leq t_0$ and
$\boldsymbol\xi\in\mathcal C_{m,\eta}(\Omega)$,
\begin{equation}\label{eq:energy-remainder-C1}
|\mathcal R_\varepsilon^{\rm app}|
+|\nabla_{\boldsymbol t}\mathcal R_\varepsilon^{\rm app}|
+|\nabla_{\boldsymbol\xi}\mathcal R_\varepsilon^{\rm app}|
\leq C\varepsilon^\tau,
\end{equation}
and
\begin{equation}\label{eq:energy-remainder-C2}
\|D_{\boldsymbol t\boldsymbol t}^2
\mathcal R_\varepsilon^{\rm app}\|
+\|D_{\boldsymbol t\boldsymbol\xi}^2
\mathcal R_\varepsilon^{\rm app}\|
+\|D_{\boldsymbol\xi\boldsymbol\xi}^2
\mathcal R_\varepsilon^{\rm app}\|
\leq C\varepsilon^\tau.
\end{equation}
The derivatives in \eqref{eq:energy-remainder-C1}--\eqref{eq:energy-remainder-C2}
are taken in the original variables
$(\boldsymbol t,\boldsymbol\xi)$.
\end{proposition}

\begin{proof}
We first expand the quadratic part.  For one projected bubble,
\begin{equation*}
\int_\Omega|\Delta PU_{\mu_i,\xi_i}|^2\,dx
=
\int_\Omega F_{\mu_i,\xi_i}PU_{\mu_i,\xi_i}\,dx.
\end{equation*}
Insert \eqref{eq:projection-expansion}.  The tail outside $\Omega$
is $O(\mu_i^{-4}\log\mu_i)$.  Lemma
\ref{lem:bubble-integrals} gives
\begin{equation}\label{eq:projected-self-energy-expansion}
\begin{aligned}
\int_\Omega|\Delta PU_{\mu_i,\xi_i}|^2\,dx
={}B_\alpha+2b_\alpha M_\alpha\log\mu_i
-b_\alpha M_\alpha\log C_\alpha+M_\alpha^2H(\xi_i,\xi_i)+O(\varepsilon^2).
\end{aligned}
\end{equation}
For $i\ne j$, Lemma \ref{lem:pair-interaction} gives
\begin{equation*}
\int_\Omega
\Delta PU_{\mu_i,\xi_i}\Delta PU_{\mu_j,\xi_j}\,dx
=
M_\alpha^2G(\xi_i,\xi_j)+O(\varepsilon^2).
\end{equation*}

We next treat the nonlocal part.  Put
$B_i=B_{\eta/4}(\xi_i)$.  If $x,y\in B_i$, use
\eqref{eq:balance}, \eqref{eq:core-exponential-expansion}, and
\eqref{eq:bubble-nonlocal}.  Then
\begin{equation*}
\varepsilon^{8-\alpha}
\int_{B_i}\int_{B_i}
\frac{K(x)K(y)e^{W(x)+W(y)}}{|x-y|^\alpha}\,dxdy
=
M_\alpha e^{2b_\alpha t_i}+O(\varepsilon^\tau).
\end{equation*}
If $x\in B_i$ and $y\in B_j$, $i\ne j$, then
\eqref{eq:separated-double} gives $O(\varepsilon^\alpha)$.
Terms with one variable outside the core balls are
$O(\varepsilon^4)$.  Therefore
\begin{equation}\label{eq:nonlocal-total-energy-expansion}
\varepsilon^{8-\alpha}
\int_\Omega\int_\Omega
\frac{K(x)K(y)e^{W(x)+W(y)}}{|x-y|^\alpha}\,dxdy
=
M_\alpha\sum_{i=1}^{m}e^{2b_\alpha t_i}
+O(\varepsilon^\tau).
\end{equation}

The scale identity is
\begin{equation}\label{eq:log-scale-energy-substitution}
\begin{aligned}
2b_\alpha\log\mu_i
={}&2b_\alpha\log C_\alpha
-2b_\alpha\log\varepsilon
-2\log K(\xi_i)\\
&-2M_\alpha H(\xi_i,\xi_i)
-2M_\alpha\sum_{j\ne i}G(\xi_i,\xi_j)
+2b_\alpha t_i.
\end{aligned}
\end{equation}
Insert \eqref{eq:log-scale-energy-substitution} in the quadratic
energy.  Each pair $G(\xi_i,\xi_j)$ occurs twice in the scale terms
and once in the quadratic cross term.  Its remaining coefficient is
$-M_\alpha^2$.  The remaining self terms are
$-M_\alpha\log K(\xi_i)$ and
$-M_\alpha^2H(\xi_i,\xi_i)/2$.  This gives
$-M_\alpha\mathcal F_m(\boldsymbol\xi)$.  Equations
\eqref{eq:projected-self-energy-expansion}--
\eqref{eq:nonlocal-total-energy-expansion} prove
\eqref{eq:energy-expansion} and \eqref{eq:energy-constant}.

We prove the derivative estimates.  It is enough to consider one
core integral.  Set
\[
x=\xi_i+\frac z{\mu_i},
\qquad
y=\xi_i+\frac w{\mu_i}.
\]
After this change of variables, the leading integrand is independent
of $(\boldsymbol t,\boldsymbol\xi)$.  A $t_i$ derivative acts on the
factor $e^{2b_\alpha t_i}$ and on a uniformly integrable rational
profile.  A center derivative acts on
\[
K\left(\xi_i+\frac z{\mu_i}\right),
\quad
H\left(\xi_i+\frac z{\mu_i},\xi_i\right),
\quad
G\left(\xi_i+\frac z{\mu_i},\xi_j\right).
\]
The zero first moment in \eqref{eq:moment-zero} removes all
terms which are linear in $z$ or $w$.  The remaining terms contain
either $\mu_i^{-2}(|z|^2+|w|^2)$, an exterior tail, or a separated-core
factor $\varepsilon^\alpha$.  Lemmas
\ref{lem:weighted-Riesz} and \ref{lem:separated-core-estimates}
give an integrable majorant.  The same argument applies after two
parameter derivatives.  In particular,
\begin{equation}\label{eq:energy-core-C2-estimate}
\left|
\partial_{\boldsymbol t}^{a}
\partial_{\boldsymbol\xi}^{b}
\mathcal E_{i,\varepsilon}
\right|
\leq C\varepsilon^\tau,
\qquad |a|+|b|\leq2,
\end{equation}
for the difference between the $i$-th core integral and its displayed
leading term.

For an interaction of two distinct cores, the kernel is smooth on
$B_i\times B_j$.  Two center derivatives are therefore harmless.
Equation \eqref{eq:separated-double} gives
\begin{equation}\label{eq:energy-cross-C2-estimate}
\left|
\partial_{\boldsymbol t}^{a}
\partial_{\boldsymbol\xi}^{b}
\mathcal E_{ij,\varepsilon}
\right|
\leq C\varepsilon^\alpha,
\qquad |a|+|b|\leq2.
\end{equation}
The exterior terms satisfy the same estimate with
$\varepsilon^\alpha$ replaced by $\varepsilon^4|\log\varepsilon|$.
Since $\tau<\min\{1,\alpha\}$, equations
\eqref{eq:energy-core-C2-estimate}--
\eqref{eq:energy-cross-C2-estimate} imply
\eqref{eq:energy-remainder-C1} and
\eqref{eq:energy-remainder-C2}.
\end{proof}

The raw scale Hessian is negative:
\begin{equation}\label{eq:raw-scale-Hessian}
\left.\frac{\partial^2}{\partial t_i\partial t_j}
J_\varepsilon(W_{\boldsymbol t,\boldsymbol\xi})\right|_{\boldsymbol t=0}
=-2M_\alpha b_\alpha^2\delta_{ij}+O(\varepsilon^\tau).
\end{equation}
This is not the scale Hessian of the reduced manifold. The derivative of the projected correction has an $O(1)$ component. It removes the negative constant mode in \eqref{eq:raw-scale-Hessian}.

\subsection{Dilation equations and the capacity correction}

Let
\[
u_{\boldsymbol t,\boldsymbol\xi}
:=W_{\boldsymbol t,\boldsymbol\xi}
+\phi_\varepsilon(\boldsymbol t,\boldsymbol\xi).
\]
At $\boldsymbol t=0$, set
\[
\Phi_{i,\varepsilon}:=
\partial_{t_i}\phi_\varepsilon(0,\boldsymbol\xi),
\qquad
Y_{i,\varepsilon}:=
\partial_{t_i}u_{\boldsymbol t,\boldsymbol\xi}|_{\boldsymbol t=0}.
\]

Differentiation in a scale direction produces the following core correction and cancellation.

\begin{lemma}\label{lem:scale-corrector}
For every $0<\beta<1$, the following convergences hold in the
$i$-th rescaled variables:
\begin{equation}\label{eq:scale-corrector-limit}
\Phi_{i,\varepsilon}(\xi_i+z/\mu_i)
\longrightarrow-b_\alpha
\quad\hbox{in }C^{3,\beta}_{\rm loc}(\mathbb R^4),
\end{equation}
and
\begin{equation}\label{eq:scale-tangent-limit}
Y_{i,\varepsilon}(\xi_i+z/\mu_i)
\longrightarrow Z_0(z)
\quad\hbox{in }C^{3,\beta}_{\rm loc}(\mathbb R^4).
\end{equation}
At every other core, both rescaled limits are zero. Moreover,
\begin{equation}\label{eq:core-cancellation}
\begin{aligned}
\mathcal Q_\infty(Z_0+b_\alpha,Z_0+b_\alpha)
+2\mathcal Q_\infty(Z_0+b_\alpha,-b_\alpha)
+\mathcal Q_\infty(-b_\alpha,-b_\alpha)=0.
\end{aligned}
\end{equation}
The three terms are
\begin{equation*}
\begin{aligned}
\mathcal Q_\infty(Z_0+b_\alpha,Z_0+b_\alpha)
&=-2M_\alpha b_\alpha^2,\\
2\mathcal Q_\infty(Z_0+b_\alpha,-b_\alpha)
&=4M_\alpha b_\alpha^2,\\
\mathcal Q_\infty(-b_\alpha,-b_\alpha)
&=-2M_\alpha b_\alpha^2.
\end{aligned}
\end{equation*}
\end{lemma}

\begin{proof}
Differentiate the projected nonlinear problem at $\boldsymbol t=0$. In the $i$-th core, Proposition \ref{prop:residual} gives, after passage to the limit,
\[
\mathcal L_\infty\widetilde\Phi_i
=2b_\alpha F+\sum_{k=0}^{4}d_kFZ_k.
\]
Differentiated orthogonality gives
\[
\int_{\mathbb R^4}FZ_k\widetilde\Phi_i=0,
\qquad 0\leq k\leq4,
\]
up to terms tending to zero. Test the limiting equation with every $Z_\ell$. Self-adjointness and Lemma \ref{lem:kernel-Gram} imply $d_k=0$. The constant function satisfies
\[
\mathcal L_\infty(-b_\alpha)=2b_\alpha F.
\]
The difference $\widetilde\Phi_i+b_\alpha$ belongs to the limiting kernel. The five orthogonality conditions remove all its components. This proves \eqref{eq:scale-corrector-limit}. At another peak the limiting right-hand side is zero. The same argument gives zero.

The projection expansion gives
\[
\partial_{t_i}PU_{\mu_i,\xi_i}
=Z_{i,0}+b_\alpha+o(1)
\]
in core variables. Formula \eqref{eq:scale-tangent-limit} follows.

Finally, $\mathcal L_\infty1=-2F$ and
$\mathcal L_\infty Z_0=0$. Hence
\[
\mathcal Q_\infty(1,1)=-2M_\alpha,
\qquad
\mathcal Q_\infty(Z_0,v)=0.
\]
Substitute $Z_0+b_\alpha$ and $-b_\alpha$. The three displayed values follow. Their sum is zero. This is the exact cancellation of the $O(1)$ negative term in \eqref{eq:raw-scale-Hessian}.
\end{proof}

We need a uniform lifting estimate for Cauchy data on the shrinking inner sphere.

\begin{lemma}\label{lem:small-sphere-lifting}
Let $\xi$ remain in a compact subset of $\Omega$ and let $\rho\to0$.
For data $g_0$ and $g_1$ on $\partial B_\rho(\xi)$, set
\[
g_0^\rho(\theta)=g_0(\xi+\rho\theta),
\qquad
g_1^\rho(\theta)=\rho g_1(\xi+\rho\theta),
\qquad \theta\in\mathbb S^3,
\]
and let $\overline g_0^\rho$ be the average of $g_0^\rho$ on
$\mathbb S^3$.  There is a unique biharmonic function $h$ satisfying
\begin{equation*}
\begin{cases}
\Delta^2h=0&\text{in }\Omega\setminus\overline{B_\rho(\xi)},\\
h=g_0,\quad \partial_\nu h=g_1&\text{on }\partial B_\rho(\xi),\\
h=\Delta h=0&\text{on }\partial\Omega.
\end{cases}
\end{equation*}
Moreover,
\begin{equation}\label{eq:small-sphere-lifting-estimate}
\begin{aligned}
\|h\|_{H^2(\Omega\setminus B_\rho(\xi))}^2
\leq C\Bigg[
&\frac{|\overline g_0^\rho|^2}{|\log\rho|}
+\|g_0^\rho-\overline g_0^\rho\|_{H^{3/2}(\mathbb S^3)}^2+\|g_1^\rho\|_{H^{1/2}(\mathbb S^3)}^2
\Bigg].
\end{aligned}
\end{equation}
The constant is uniform in $\xi$.
\end{lemma}

\begin{proof}
Fix $R>0$ so that $B_{4R}(\xi)\Subset\Omega$ uniformly for the
allowed centers. Decompose $g_0^\rho$ and $g_1^\rho$ into spherical
harmonics. The mean of the Dirichlet datum with zero normal datum is
lifted by $-\overline g_0^\rho b_\alpha^{-1}V_{\rho,\xi}$, where
$V_{\rho,\xi}$ is the capacity profile in Lemma
\ref{lem:biharmonic-capacity}. After the harmless normalization by
$b_\alpha$, its squared $H^2$ norm is bounded by
\[
C\frac{|\overline g_0^\rho|^2}{|\log\rho|}.
\]

It remains to lift the zero-mean Dirichlet datum and the full normal
datum. On $B_R(\xi)\setminus B_\rho(\xi)$, solve the biharmonic
ordinary differential equation in every spherical harmonic sector.
For degree $k\geq1$, the radial factors are
$r^k,r^{-k-2},r^{k+2},r^{-k}$. In the degree-zero sector the remaining
normal datum is lifted by the factors $1,\log r,r^2,r^{-2}$. Solving
the four boundary equations at $r=\rho$ and $r=R$ gives a linear
boundary-to-energy map whose norm is uniform after the inner traces
are rescaled to the unit sphere. Consequently the annular lifting
$h_{\rm ann}$ satisfies
\[
\|h_{\rm ann}\|_{H^2(B_R(\xi)\setminus B_\rho(\xi))}^2
\leq C\left(
\|g_0^\rho-\overline g_0^\rho\|_{H^{3/2}(\mathbb S^3)}^2
+\|g_1^\rho\|_{H^{1/2}(\mathbb S^3)}^2\right).
\]
The value and normal derivative of $h_{\rm ann}$ on $\partial B_R$
are controlled by the same right-hand side. A fixed-domain
biharmonic extension on $\Omega\setminus B_R(\xi)$ imposes
$h=\Delta h=0$ on $\partial\Omega$ and has uniformly bounded energy.
Adding the mean lifting gives
\eqref{eq:small-sphere-lifting-estimate}.

If the data vanish, testing the biharmonic equation with $h$ gives
$\int|\Delta h|^2=0$: the inner Cauchy data remove the inner boundary
terms and the outer Navier data remove the outer ones. Hence $h=0$.
This proves uniqueness and completes the proof.
\end{proof}

The passage from fixed core radii to a growing core radius uses the
following quantitative form of the limiting inverse. It also records
the contribution of the nonlocal tail outside the expanding ball.

\begin{lemma}\label{lem:expanding-core-scale}
Fix $C_0>0$. There are $p>4$, $q_0>0$, and $\sigma>0$ such that the
following estimate holds. Let
\[
\Phi_{i,\varepsilon}^{\boldsymbol t}
:=\partial_{t_i}\phi_\varepsilon(\boldsymbol t,\boldsymbol\xi),
\qquad
\Psi_{i,\varepsilon}^{\boldsymbol t}(z)
:=\Phi_{i,\varepsilon}^{\boldsymbol t}
\left(\xi_i+\frac z{\mu_i}\right)+b_\alpha,
\]
where $|\boldsymbol t|\leq C_0\varepsilon^\tau
|\log\varepsilon|$. Uniformly for
$2\leq R\leq\varepsilon^{-\sigma/(8q_0)}$,
\begin{equation}\label{eq:expanding-core-estimate}
\begin{aligned}
\sup_{|z|\leq2R}\Big(&|\Psi_{i,\varepsilon}^{\boldsymbol t}(z)|
+(1+|z|)|\nabla\Psi_{i,\varepsilon}^{\boldsymbol t}(z)|\\
&+(1+|z|)^2|D^2\Psi_{i,\varepsilon}^{\boldsymbol t}(z)|\Big)
\leq C\left[(1+R)^{q_0}
(\varepsilon^\sigma+|\boldsymbol t|)+R^{-2}\right].
\end{aligned}
\end{equation}
\end{lemma}

\begin{proof}
We give the localization argument because the operator is nonlocal.
After the change of variables $x=\xi_i+z/\mu_i$, differentiation of
the projected equation gives, in $B_{8R}$,
\begin{equation}\label{eq:expanding-core-equation}
\mathcal L_\infty\Psi_{i,\varepsilon}^{\boldsymbol t}
=h_{i,\varepsilon}^{\boldsymbol t}
+\mathcal E_{i,\varepsilon}^{\boldsymbol t}
[\Psi_{i,\varepsilon}^{\boldsymbol t}]
+\mathcal T_{i,\varepsilon,R}^{\boldsymbol t}
+\sum_{k=0}^4d_{i,k,\varepsilon}^{\boldsymbol t}FZ_k.
\end{equation}
Here $h_{i,\varepsilon}^{\boldsymbol t}$ contains the variation of
$K$, the regular Green terms, and the differentiated residual;
$\mathcal E_{i,\varepsilon}^{\boldsymbol t}$ is the difference
between the rescaled local coefficients and those of
$\mathcal L_\infty$; and $\mathcal T_{i,\varepsilon,R}^{\boldsymbol t}$
is the part of the Riesz integral with $|y|>8R$. Propositions
\ref{prop:residual} and \ref{prop:linear-solvability}, together with
Lemma \ref{lem:weighted-Riesz}, give
\begin{align}
\|h_{i,\varepsilon}^{\boldsymbol t}\|_{L^p(B_{8R})}
+\|\mathcal E_{i,\varepsilon}^{\boldsymbol t}
[\Psi_{i,\varepsilon}^{\boldsymbol t}]\|_{L^p(B_{8R})}
&\leq C(1+R)^{q_0}
(\varepsilon^\sigma+|\boldsymbol t|),
\notag\\
\|\mathcal T_{i,\varepsilon,R}^{\boldsymbol t}\|_{L^p(B_{8R})}
&\leq CR^{-4}.
\label{eq:expanding-core-tail}
\end{align}
For the last estimate, use the boundedness of the differentiated
correction and
\[
\int_{|y|>8R}|y|^{-\alpha}e^{U(y)}\,dy\leq CR^{-4}.
\]
The estimates in \eqref{eq:expanding-core-tail}, the differentiated
orthogonality conditions, and $\int FZ_k=0$ imply
\begin{equation}\label{eq:expanding-core-moments}
\left|\int_{\mathbb R^4}FZ_k
\Psi_{i,\varepsilon}^{\boldsymbol t}\,dz\right|
+|d_{i,k,\varepsilon}^{\boldsymbol t}|
\leq C\left[(1+R)^{q_0}
(\varepsilon^\sigma+|\boldsymbol t|)+R^{-4}\right].
\end{equation}

For completeness, we record the localized inverse estimate used here.
If $\psi$ is bounded and $\chi_R$ equals one in $B_{4R}$ and vanishes
outside $B_{8R}$, with $|D^j\chi_R|\leq CR^{-j}$, the global spectral
inverse in Theorem \ref{thm:limit-input}, applied to
$\chi_R\psi$ after its five kernel components are removed, and the
interior $W^{4,p}$ estimate give
\begin{equation}\label{eq:localized-limiting-inverse}
\begin{aligned}
&\sup_{|z|\leq2R}
\bigl(|\psi|+(1+|z|)|\nabla\psi|
+(1+|z|)^2|D^2\psi|\bigr)\\
&\quad\leq C(1+R)^{q_0}
\left(\|\mathcal L_\infty\psi\|_{L^p(B_{8R})}
+\sum_{k=0}^4\left|\int FZ_k\psi\right|\right)
+CR^2\|[\mathcal L_\infty,\chi_R]\psi\|_{H^{-2}}
+CR^{-2}\|\psi\|_\infty.
\end{aligned}
\end{equation}
We spell out the nonlocal part of the commutator.  Write
\[
\mathcal R_\alpha\psi(z)
=e^{U(z)}\int_{\mathbb R^4}
\frac{e^{U(y)}\psi(y)}{|z-y|^\alpha}\,dy.
\]
The multiplication term in $\mathcal L_\infty$ commutes with
$\chi_R$, while
\[
[\mathcal R_\alpha,\chi_R]\psi(z)
=e^{U(z)}\int_{\mathbb R^4}
\frac{e^{U(y)}(\chi_R(z)-\chi_R(y))\psi(y)}
{|z-y|^\alpha}\,dy.
\]
Split the last integral according to
$B_{4R}$, $B_{8R}\setminus B_{4R}$, and
$\mathbb R^4\setminus B_{8R}$.  Since
$e^U(z)\leq C(1+|z|)^{-(8-\alpha)}$ and
$\|\nabla\chi_R\|_\infty\leq C/R$, the near part is estimated by the
mean value theorem and the far part by
$\int_{|y|>4R}|y|^{-\alpha}e^{U(y)}\,dy\leq CR^{-4}$.
Consequently, for the exponent $p>4$ fixed above,
\[
\|[\mathcal R_\alpha,\chi_R]\psi\|_{L^p(\mathbb R^4)}
\leq CR^{-4}\|\psi\|_\infty.
\]
The differential commutator
$[\Delta^2,\chi_R]\psi$ contains derivatives of $\chi_R$ of orders
one through four and is supported in
$B_{8R}\setminus B_{4R}$.  By duality and interpolation,
\[
\|[\Delta^2,\chi_R]\psi\|_{H^{-2}}
\leq C\left(
R^{-1}\|\psi\|_{H^2(B_{8R}\setminus B_{4R})}
+R^{-2}\|\psi\|_{L^2(B_{8R}\setminus B_{4R})}\right).
\]
To make the absorption explicit, let
\[
E_j=\|\psi\|_{H^2(B_{2^{j+1}R}\setminus B_{2^jR})}.
\]
A rescaled interior estimate on two neighboring annuli, followed by
Young's inequality, gives
\[
E_j\leq \frac14E_{j+1}
+C(1+2^jR)^{q_0}
\|\mathcal L_\infty\psi\|_{L^p(B_{2^{j+2}R})}
+C(2^jR)^{-2}\|\psi\|_\infty,
\]
with the five kernel coefficients added to the middle term.  The
constant $1/4$ is obtained after enlarging the outer annulus once;
it is independent of $j$ and $R$.  Iteration up to the last annulus
meeting the support of $\chi_R$, and then summation of the resulting
geometric series, absorbs the annular terms and leaves
$CR^{-2}\|\psi\|_\infty$.  Together with the preceding estimate for
the Riesz commutator, this proves
\eqref{eq:localized-limiting-inverse} with a fixed power $q_0$.

Apply \eqref{eq:localized-limiting-inverse} to
$\Psi_{i,\varepsilon}^{\boldsymbol t}$ and use
\eqref{eq:expanding-core-equation}--
\eqref{eq:expanding-core-moments}. After increasing $q_0$ once, the
terms containing $R^{-4}$ and the cutoff commutator are bounded by
$CR^{-2}$. This proves \eqref{eq:expanding-core-estimate}.
\end{proof}

The following comparison identifies the capacitary tail of a scale
tangent and supplies the precise order of its exterior energy.

\begin{lemma}\label{lem:scale-tangent-exterior}
There are $\sigma>0$ and a choice of $R_\varepsilon$ satisfying
\eqref{eq:R-epsilon-choice} such that, uniformly for separated
configurations,
\begin{equation}\label{eq:scale-tangent-growing-ball}
\begin{aligned}
\omega_\varepsilon:=\sup_{|z|\leq2R_\varepsilon}
\Bigg(&\left|\Phi_{i,\varepsilon}
\left(\xi_i+\frac z{\mu_i}\right)+b_\alpha\right|\\
&+(1+|z|)\left|\nabla_z\Phi_{i,\varepsilon}
\left(\xi_i+\frac z{\mu_i}\right)\right|\\
&+(1+|z|)^2\left|D_z^2\Phi_{i,\varepsilon}
\left(\xi_i+\frac z{\mu_i}\right)\right|\Bigg)
=o(|\log\varepsilon|^{-1/2}).
\end{aligned}
\end{equation}
For this choice, let $\widehat Z_{i,0}$ be defined by
\eqref{eq:dilation-test}. Then
\begin{equation}\label{eq:scale-tangent-exterior-comparison}
\left\|
Y_{i,\varepsilon}-\widehat Z_{i,0}
\right\|_{H^2(\Omega\setminus
B_{2\rho_{i,\varepsilon}}(\xi_i))}
=o(|\log\varepsilon|^{-1/2}).
\end{equation}
Consequently,
\begin{equation}\label{eq:scale-tangent-exterior-energy}
\int_{\Omega\setminus
B_{2\rho_{i,\varepsilon}}(\xi_i)}
|\Delta Y_{i,\varepsilon}|^2\,dx
=
\frac{8\pi^2b_\alpha^2}{|\log\varepsilon|}
+o(|\log\varepsilon|^{-1}),
\end{equation}
and, for $i\neq j$,
\begin{equation}\label{eq:scale-tangent-cross-energy}
\int_\Omega
\Delta Y_{i,\varepsilon}\Delta Y_{j,\varepsilon}\,dx
=O(|\log\varepsilon|^{-2})
+o(|\log\varepsilon|^{-1}).
\end{equation}
For every fixed $C_0>0$, the same conclusions hold uniformly when the
base point $\boldsymbol t=0$ is replaced by
$|\boldsymbol t|\leq C_0\varepsilon^\tau|\log\varepsilon|$, with
$\Phi_{i,\varepsilon}$, $Y_{i,\varepsilon}$, $\mu_i$, and
$\rho_{i,\varepsilon}$ replaced by their values at that base point.
\end{lemma}

\begin{proof}
Lemma \ref{lem:expanding-core-scale}, with
$\boldsymbol t=0$, gives for $2\leq R\leq
\varepsilon^{-\sigma/(8q_0)}$
\[
\begin{aligned}
\sup_{|z|\leq2R}\Bigg(&
\left|\Phi_{i,\varepsilon}
\left(\xi_i+\frac z{\mu_i}\right)+b_\alpha\right|
+(1+|z|)\left|\nabla_z\Phi_{i,\varepsilon}
\left(\xi_i+\frac z{\mu_i}\right)\right|\\
&+(1+|z|)^2\left|D_z^2\Phi_{i,\varepsilon}
\left(\xi_i+\frac z{\mu_i}\right)\right|
\Bigg)
\leq C\left[(1+R)^{q_0}\varepsilon^\sigma+R^{-2}\right].
\end{aligned}
\]
We now fix
\begin{equation}\label{eq:diagonal-radius-rate}
R_\varepsilon:=|\log\varepsilon|^{3/4}.
\end{equation}
Then $R_\varepsilon\to\infty$,
$\varepsilon R_\varepsilon\to0$, and
$\log R_\varepsilon=o(|\log\varepsilon|)$. Moreover,
\[
R_\varepsilon^{-2}|\log\varepsilon|^{1/2}
=|\log\varepsilon|^{-1}\to0
\]
and, for every fixed $q_0$,
\[
(1+R_\varepsilon)^{q_0}\varepsilon^\sigma
|\log\varepsilon|^{1/2}\to0.
\]
Thus \eqref{eq:scale-tangent-growing-ball} follows, and this choice
also satisfies \eqref{eq:R-epsilon-choice}. The uniform statement in
the lemma follows from the $\boldsymbol t$-uniform estimate in Lemma
\ref{lem:expanding-core-scale}.

Put
\[
D_{i,\varepsilon}
=\Omega\setminus B_{2\rho_{i,\varepsilon}}(\xi_i),
\qquad
w_{i,\varepsilon}
=Y_{i,\varepsilon}-\widehat Z_{i,0}.
\]
Let $g_{0,i,\varepsilon}$ and $g_{1,i,\varepsilon}$ be the Dirichlet
and normal traces of $w_{i,\varepsilon}$ on the inner sphere. By
\eqref{eq:scale-tangent-growing-ball},
\begin{equation*}
\begin{aligned}
&|\overline g_{0,i,\varepsilon}|
+\|g_{0,i,\varepsilon}^{\rho}-
\overline g_{0,i,\varepsilon}\|_{H^{3/2}(\mathbb S^3)}
+\|g_{1,i,\varepsilon}^{\rho}\|_{H^{1/2}(\mathbb S^3)}
=o(|\log\varepsilon|^{-1/2}).
\end{aligned}
\end{equation*}
Let $h_{i,\varepsilon}$ be the lifting from Lemma
\ref{lem:small-sphere-lifting}. Then
\begin{equation}\label{eq:scale-tangent-boundary-lifting}
\|h_{i,\varepsilon}\|_{H^2(D_{i,\varepsilon})}
=o(|\log\varepsilon|^{-1/2}).
\end{equation}

We next estimate the equation for the remaining difference. Let
\[
A_\varepsilon(x)
:=\varepsilon^{8-\alpha}K(x)e^{u_{0,\boldsymbol\xi}(x)}
\int_\Omega
\frac{K(y)e^{u_{0,\boldsymbol\xi}(y)}}{|x-y|^\alpha}\,dy
\]
and
\[
\mathcal K_\varepsilon[\psi](x)
:=\varepsilon^{8-\alpha}K(x)e^{u_{0,\boldsymbol\xi}(x)}
\int_\Omega
\frac{K(y)e^{u_{0,\boldsymbol\xi}(y)}\psi(y)}
{|x-y|^\alpha}\,dy.
\]
Since $\widehat Z_{i,0}$ and $h_{i,\varepsilon}$ are biharmonic in
$D_{i,\varepsilon}$, the function
$q_{i,\varepsilon}:=w_{i,\varepsilon}-h_{i,\varepsilon}$ satisfies
\begin{equation}\label{eq:scale-tangent-difference-equation}
\Delta^2q_{i,\varepsilon}
=f_{i,\varepsilon}^{\rm lin}
+f_{i,\varepsilon}^{\rm mult}
\quad\text{in }D_{i,\varepsilon},
\end{equation}
where
\[
f_{i,\varepsilon}^{\rm lin}
=A_\varepsilon Y_{i,\varepsilon}
+\mathcal K_\varepsilon[Y_{i,\varepsilon}],
\qquad
f_{i,\varepsilon}^{\rm mult}
=\sum_{a=1}^{5m}
\left[(\partial_{t_i}c_a)\Psi_a
+c_a\partial_{t_i}\Psi_a\right].
\]
On the inner boundary,
$q_{i,\varepsilon}=\partial_\nu q_{i,\varepsilon}=0$; on
$\partial\Omega$,
$q_{i,\varepsilon}=\Delta q_{i,\varepsilon}=0$. These are the mixed
conditions used in the energy identity. The $H^{-2}$ norm below is
the dual norm of this mixed energy space.

The two terms in $f_{i,\varepsilon}^{\rm lin}$ must be kept together.
After the change of variables $x=\xi_i+z/\mu_i$, their leading
profile is
\[
FZ_0+e^U\int_{\mathbb R^4}
\frac{e^U(y)Z_0(y)}{|z-y|^\alpha}\,dy
=\Delta^2Z_0.
\]
Let $v$ satisfy the same homogeneous mixed boundary conditions as
$q_{i,\varepsilon}$. After two integrations by parts, the inner
boundary terms vanish because $v=\partial_\nu v=0$. On the outer
boundary, $v=0$, while
$\Delta Z_{i,0}=O(\mu_i^{-2})$ because the center remains a fixed
positive distance from $\partial\Omega$. The trace theorem gives
\[
\left|\int_{D_{i,\varepsilon}}\Delta^2Z_{i,0}\,v\,dx\right|
\leq
\left(CR_\varepsilon^{-2}+C\mu_i^{-2}\right)
\|v\|_{H^2(D_{i,\varepsilon})}.
\]
The difference between the exact pair and this limiting tail contains
the error in \eqref{eq:scale-tangent-growing-ball}, a derivative of
$K$, a regular Green term, a truncated Riesz tail, or a separated-core
interaction. Lemmas \ref{lem:weighted-Riesz} and
\ref{lem:separated-core-estimates} therefore give
\begin{equation*}
\|f_{i,\varepsilon}^{\rm lin}\|_{H^{-2}(D_{i,\varepsilon})}
\leq C\bigl(R_\varepsilon^{-2}+\mu_i^{-2}+\omega_\varepsilon
+\varepsilon^\tau\bigr)
=o(|\log\varepsilon|^{-1/2}).
\end{equation*}
The differentiated multiplier system and Proposition
\ref{prop:linear-solvability} give
$|c_a|+|\partial_{t_i}c_a|\leq C\varepsilon^\tau$. The test functions
and their scale derivatives act uniformly on $H^2$, and hence
\begin{equation}\label{eq:scale-difference-multiplier-estimate}
\|f_{i,\varepsilon}^{\rm mult}\|_{H^{-2}(D_{i,\varepsilon})}
\leq C\varepsilon^\tau
=o(|\log\varepsilon|^{-1/2}).
\end{equation}

Testing \eqref{eq:scale-tangent-difference-equation} with
$q_{i,\varepsilon}$, using the mixed-boundary coercivity and
\eqref{eq:scale-tangent-boundary-lifting}--
\eqref{eq:scale-difference-multiplier-estimate}, gives
\eqref{eq:scale-tangent-exterior-comparison}. Equations
\eqref{eq:dilation-test-energy},
\eqref{eq:dilation-test-cross-energy}, and Cauchy's inequality then
give \eqref{eq:scale-tangent-exterior-energy} and
\eqref{eq:scale-tangent-cross-energy}.
\end{proof}

We next control the residual scale equations together with their
center derivatives. This estimate is needed in the mixed block of the
reduced Hessian.

\begin{lemma}\label{lem:scale-defect-C1}
Define
\begin{equation*}
\boldsymbol r_\varepsilon(\boldsymbol\xi)
:=\nabla_{\boldsymbol t}\widetilde J_\varepsilon
(0,\boldsymbol\xi).
\end{equation*}
Then
\begin{equation}\label{eq:scale-defect-C1}
\|\boldsymbol r_\varepsilon\|_{C^1(\mathcal C_{m,\eta}(\Omega))}
\leq C\varepsilon^\tau.
\end{equation}
The derivative is taken with respect to the original center
variables.
\end{lemma}

\begin{proof}
At $\boldsymbol t=0$, the projected equation has the form
\begin{equation}\label{eq:scale-defect-projected-equation}
J_\varepsilon'(u_{0,\boldsymbol\xi})
=\sum_{a=1}^{5m}c_a(0,\boldsymbol\xi)
\Psi_a(0,\boldsymbol\xi).
\end{equation}
Propositions \ref{prop:residual} and
\ref{prop:nonlinear-correction} give
$|c_a(0,\boldsymbol\xi)|\leq C\varepsilon^\tau$. Since
$r_{i,\varepsilon}=J_\varepsilon'(u_{0,\boldsymbol\xi})
[Y_{i,\varepsilon}]$, the scale rows of the Gram matrix and Lemma
\ref{lem:scale-corrector} imply
$|r_{i,\varepsilon}|\leq C\varepsilon^\tau$.

We now differentiate in an original center variable. Put
\[
X_{j,\ell,\varepsilon}
:=\partial_{\xi_{j,\ell}}u_{0,\boldsymbol\xi}.
\]
The chain rule gives the exact identity
\begin{equation*}
\partial_{\xi_{j,\ell}}r_{i,\varepsilon}
=J_\varepsilon''(u_{0,\boldsymbol\xi})
[X_{j,\ell,\varepsilon},Y_{i,\varepsilon}]
+J_\varepsilon'(u_{0,\boldsymbol\xi})
[\partial_{\xi_{j,\ell}}Y_{i,\varepsilon}].
\end{equation*}
The two terms must be expanded together. We describe the cancellation
in the finite multiplier system, which also keeps track of the
unnormalized factor $\mu_j$ in $X_{j,\ell,\varepsilon}$.

Test \eqref{eq:scale-defect-projected-equation} with all functions
$\widehat Z_{p,k}$. This gives
\[
\mathcal A_\varepsilon\boldsymbol c
=\boldsymbol d_\varepsilon,
\qquad
(\mathcal A_\varepsilon)_{(p,k),a}
=\int_\Omega\Psi_a\widehat Z_{p,k}\,dx,
\]
where $\boldsymbol d_\varepsilon$ is obtained by pairing the
projected residual with the same tests. Differentiate this system
with respect to $\xi_{j,\ell}$. In the $i$-th scale row, the terms
which carry a factor $\mu_j$ have, after rescaling, one of the two
leading forms
\[
\mu_j\int_{\mathbb R^4}FZ_0Z_\ell\,dz,
\qquad
\mu_j\mathcal Q_\infty(Z_0,Z_\ell).
\]
Both vanish. If the derivative acts on the dependence of
$\mu_i(\varepsilon,\boldsymbol\xi)$ on the center, the leading
profile is a dilation mode. Its order-one contribution vanishes by
$\mathcal Q_\infty(Z_0,Z_0)=0$ together with the constant-mode
cancellation in Lemma \ref{lem:scale-corrector}. Thus the terms that
remain in a scale row contain at least one of the following factors:
\[
\mu_i^{-1},\qquad
\varepsilon^\tau,\qquad
\varepsilon^\alpha,
\]
or a derivative of a smooth Green coefficient multiplied by an odd
core moment. The last contribution is zero. Propositions
\ref{prop:residual} and \ref{prop:linear-solvability}, together with
Lemmas \ref{lem:weighted-Riesz} and
\ref{lem:separated-core-estimates}, therefore give
\begin{equation*}
\left|
\partial_{\xi_{j,\ell}}d_{i,0}
-(\partial_{\xi_{j,\ell}}\mathcal A_\varepsilon
\boldsymbol c)_{i,0}
\right|
\leq C\varepsilon^\tau.
\end{equation*}
The normalized Gram matrix is uniformly invertible. Its off-diagonal
blocks are $o(1)$, and its diagonal blocks converge to
$(\int FZ_kZ_\ell)_{0\leq k,\ell\leq4}$. Solving the differentiated system
shows that every scale component of
$\partial_{\xi_{j,\ell}}\boldsymbol c$ is
$O(\varepsilon^\tau)$.

Finally differentiate
$r_{i,\varepsilon}=
\sum_a c_a\int_\Omega\Psi_aY_{i,\varepsilon}\,dx$. The derivatives
of the scale row have the same leading odd terms displayed above.
They vanish, while all remaining terms are bounded by
$C\varepsilon^\tau$. Hence
$|\partial_{\xi_{j,\ell}}r_{i,\varepsilon}|
\leq C\varepsilon^\tau$, uniformly in all indices. This proves
\eqref{eq:scale-defect-C1}.
\end{proof}

Before computing the reduced Hessian, we isolate the contribution of
the orthogonal correction in the position variables.

\begin{lemma}\label{lem:position-correction-C2}
Let
\[
\mathcal R_\varepsilon(\boldsymbol\xi)
:=J_\varepsilon(W_{0,\boldsymbol\xi}
+\phi_\varepsilon(0,\boldsymbol\xi))
-J_\varepsilon(W_{0,\boldsymbol\xi}).
\]
Then, uniformly on compact subsets of
$\mathcal C_{m,\eta}(\Omega)$,
\begin{equation}\label{eq:position-correction-C2}
\|\mathcal R_\varepsilon\|_{C^2}=o(1).
\end{equation}
\end{lemma}

\begin{proof}
Taylor's formula in the normal variable gives
\begin{equation}\label{eq:position-correction-Taylor}
\begin{aligned}
\mathcal R_\varepsilon
={}&J_\varepsilon'(W_{0,\boldsymbol\xi})
[\phi_\varepsilon]
+\int_0^1(1-s)
J_\varepsilon''(W_{0,\boldsymbol\xi}+s\phi_\varepsilon)
[\phi_\varepsilon,\phi_\varepsilon] \,ds.
\end{aligned}
\end{equation}
The zero-order estimate follows from Propositions \ref{prop:residual}
and \ref{prop:nonlinear-correction}. We discuss the center derivatives,
where the normalization requires care.

Let $\xi_a=\xi_{i(a),\ell(a)}$ and
$\xi_b=\xi_{i(b),\ell(b)}$ denote two original center coordinates, and
set
\[
D_a:=\mu_{i(a)}^{-1}\partial_{\xi_a},
\qquad
D_b:=\mu_{i(b)}^{-1}\partial_{\xi_b}.
\]
Propositions \ref{prop:linear-solvability} and
\ref{prop:nonlinear-correction} give uniform bounds for one and two
normalized derivatives of $\phi_\varepsilon$. An unnormalized
derivative of the ansatz has the decomposition
\[
\partial_{\xi_{i,\ell}}W_{0,\boldsymbol\xi}
=-\mu_i Z_{i,\ell}+B_{i,\ell},
\]
where $B_{i,\ell}$ contains the derivative of the matching scale and
the smooth Green terms and is uniformly bounded in the weighted
parameter norm.

Differentiate \eqref{eq:position-correction-Taylor} twice. After the
terms containing derivatives of the integration parameter are
combined, the potentially non-small contributions are linear
combinations of
\begin{align*}
I_{ab}^{(1)}&=J_\varepsilon''(W)
[\partial_{\xi_a}W,\partial_{\xi_b}\phi],\qquad
I_{ab}^{(2)}=J_\varepsilon'(W)
[\partial_{\xi_a\xi_b}^2\phi],\\
I_{ab}^{(3)}&=J_\varepsilon''(W)
[\partial_{\xi_a}\phi,\partial_{\xi_b}\phi],\qquad
I_{ab}^{(4)}=\int_0^1J_\varepsilon'''(W+s\phi)
[\partial_{\xi_a}W,\phi,
\partial_{\xi_b}\phi]\,ds,
\end{align*}
together with the symmetric terms obtained by interchanging $a$ and
$b$. We estimate these terms core by core.

Suppose first that both derivatives act at the same core. Insert
$\partial_{\xi_a}W=-\mu_iZ_{i,\ell}+B_{i,\ell}$ and rescale.
The coefficient of the apparent order-$\mu_i$ term in
$I_{ab}^{(1)}$ is
\[
\mathcal Q_\infty
(Z_\ell,D_b\phi_\varepsilon),
\]
which vanishes after the differentiated orthogonality condition is
used, because $\mathcal L_\infty Z_\ell=0$. If a derivative falls
on a radial coefficient, the resulting integrand is odd in the
$\ell$-th core variable. Terms with two leading translation factors
are differentiated versions of the same kernel identity and cancel
between the two parts of the Choquard Hessian. The remaining
same-core terms contain either a normalized derivative of
$\phi_\varepsilon$, a coefficient error from Proposition
\ref{prop:residual}, or one factor $\mu_i^{-1}$. Hence
\[
|I_{ab}^{(1)}|+|I_{ab}^{(2)}|
\leq C\left(\varepsilon^\tau|\log\varepsilon|^2
+\mu_i^{-1}\right).
\]
The quadratic terms $I_{ab}^{(3)}$ and $I_{ab}^{(4)}$ are bounded by
the Hardy--Littlewood--Sobolev inequality, the Adams estimate, and
\eqref{eq:correction-sharp}; they satisfy the same bound.

If the derivatives act on distinct cores, every term contains a
separated Riesz interaction or a smooth Green interaction after one
odd first moment has vanished. Lemma
\ref{lem:separated-core-estimates} and Proposition
\ref{prop:residual} give
\[
\sum_{j=1}^4|I_{ab}^{(j)}|
\leq C\left(\varepsilon^\alpha
+\varepsilon^\tau|\log\varepsilon|^2\right).
\]
The neck and exterior pieces are estimated by Lemma
\ref{lem:weighted-Riesz}; they contain either a bubble tail of order
$\max_i\mu_i^{-1}$ or the residual factor $\varepsilon^\tau$.
Combining all regions, for every multi-index $|a|\leq2$,
\[
|\partial_{\boldsymbol\xi}^a
\mathcal R_\varepsilon|
\leq C\left(
\varepsilon^\tau|\log\varepsilon|^2
+\varepsilon^\alpha
+\max_i\mu_i^{-1}
\right)=o(1).
\]
This proves \eqref{eq:position-correction-C2}.
\end{proof}

The next proposition gives all Hessian blocks required by the reduced
equations and by the inertia calculation.

\begin{proposition}\label{prop:scale-Hessian}
Put $\gamma_\alpha:=8\pi^2b_\alpha^2$. Uniformly for separated
configurations,
\begin{equation}\label{eq:scale-Hessian}
\frac{\partial^2\widetilde J_\varepsilon}
{\partial t_i\partial t_j}(0,\boldsymbol\xi)
=\frac{\gamma_\alpha}{|\log\varepsilon|}\delta_{ij}
+o(|\log\varepsilon|^{-1}),
\end{equation}
\begin{equation}\label{eq:mixed-Hessian-zero-scale}
\frac{\partial^2\widetilde J_\varepsilon}
{\partial t_i\partial\xi_{j,\ell}}(0,\boldsymbol\xi)
=O(\varepsilon^\tau),
\end{equation}
and
\begin{equation}\label{eq:position-Hessian}
D_{\boldsymbol\xi}^2\widetilde J_\varepsilon(0,\boldsymbol\xi)
=-M_\alpha D^2\mathcal F_m(\boldsymbol\xi)+o(1).
\end{equation}
In addition, for every fixed $C>0$,
\begin{equation}\label{eq:uniform-scale-Hessian}
\sup_{|\boldsymbol t|\leq
C\varepsilon^\tau|\log\varepsilon|}
\left\|D_{\boldsymbol t\boldsymbol t}^2
\widetilde J_\varepsilon(\boldsymbol t,\boldsymbol\xi)
-\frac{\gamma_\alpha}{|\log\varepsilon|}I_m\right\|
=o(|\log\varepsilon|^{-1}).
\end{equation}
On the same shrinking box,
\begin{equation}\label{eq:mixed-Hessian}
\|D_{\boldsymbol t\boldsymbol\xi}^2
\widetilde J_\varepsilon(\boldsymbol t,\boldsymbol\xi)\|
\leq \frac{C}{|\log\varepsilon|}+C\varepsilon^\tau.
\end{equation}
All derivatives are taken in the original variables
$(\boldsymbol t,\boldsymbol\xi)$.
\end{proposition}

\begin{proof}
Let
\[
q_{\varepsilon,\boldsymbol\xi}(v,w)
:=J_\varepsilon''(u_{0,\boldsymbol\xi})[v,w].
\]
The order-one core contribution in the scale--scale block vanishes by
\eqref{eq:core-cancellation}. The remaining core profile is $Z_0$,
and $\mathcal Q_\infty(Z_0,Z_0)=0$. Lemma
\ref{lem:scale-tangent-exterior} gives
\begin{align*}
q_{\varepsilon,\boldsymbol\xi}
(Y_{i,\varepsilon},Y_{i,\varepsilon})
&=\frac{\gamma_\alpha}{|\log\varepsilon|}
+o(|\log\varepsilon|^{-1}),
\\
q_{\varepsilon,\boldsymbol\xi}
(Y_{i,\varepsilon},Y_{j,\varepsilon})
&=o(|\log\varepsilon|^{-1}),
\qquad i\ne j.
\end{align*}
Differentiating the reduced energy twice gives
\begin{equation}\label{eq:scale-Hessian-identity}
\partial_{t_it_j}^2\widetilde J_\varepsilon
=q_{\varepsilon,\boldsymbol\xi}
(Y_{i,\varepsilon},Y_{j,\varepsilon})
+J_\varepsilon'(u_{0,\boldsymbol\xi})
[\partial_{t_it_j}^2u].
\end{equation}
The last term is $O(\varepsilon^\tau)$ by the projected equation,
Proposition \ref{prop:nonlinear-correction}, and the twice
differentiated orthogonality conditions. This proves
\eqref{eq:scale-Hessian}.

We next prove the uniform version. Fix
$|\boldsymbol t|\leq
C\varepsilon^\tau|\log\varepsilon|$ and set
$Y_{i,\varepsilon}^{\boldsymbol t}
:=\partial_{t_i}u_{\boldsymbol t,\boldsymbol\xi}$. The $\boldsymbol t$-uniform statement in Lemma
\ref{lem:scale-tangent-exterior}, based on Lemma
\ref{lem:expanding-core-scale}, applies at this base parameter. All
residual and inverse estimates are uniform because
$\boldsymbol t=o(1)$ and
$e^{2b_\alpha t_i}=1+O(\varepsilon^\tau|\log\varepsilon|)$. The
corresponding capacity radius satisfies
\[
|\log\rho_{i,\varepsilon}(\boldsymbol t)|
=|\log\varepsilon|-\log R_\varepsilon+O(1)+t_i
=|\log\varepsilon|+o(|\log\varepsilon|).
\]
Consequently,
\[
\frac{8\pi^2b_\alpha^2}
{|\log\rho_{i,\varepsilon}(\boldsymbol t)|}
=\frac{\gamma_\alpha}{|\log\varepsilon|}
+o(|\log\varepsilon|^{-1})
\]
uniformly in the shrinking box. The core cancellation differs from
the one at $\boldsymbol t=0$ by
$O(|\boldsymbol t|)=o(|\log\varepsilon|^{-1})$, and the projected
residual term is $O(\varepsilon^\tau)=o(|\log\varepsilon|^{-1})$.
The off-diagonal capacity and distinct-core terms are also
$o(|\log\varepsilon|^{-1})$. Repeating
\eqref{eq:scale-Hessian-identity} at the base point
$\boldsymbol t$ proves \eqref{eq:uniform-scale-Hessian}.

By definition,
$D_{\boldsymbol\xi}\boldsymbol r_\varepsilon
=D_{\boldsymbol t\boldsymbol\xi}^2
\widetilde J_\varepsilon(0,\boldsymbol\xi)$. Hence
\eqref{eq:mixed-Hessian-zero-scale} follows from Lemma
\ref{lem:scale-defect-C1}. The same differentiated capacity and
residual estimates, now based at $s\boldsymbol t$, give
\[
\sup_{|\boldsymbol t|\leq
C\varepsilon^\tau|\log\varepsilon|}
\|D_{\boldsymbol\xi}D_{\boldsymbol t\boldsymbol t}^2
\widetilde J_\varepsilon(\boldsymbol t,\boldsymbol\xi)\|
\leq \frac{C}{|\log\varepsilon|}+C\varepsilon^\tau.
\]
Taylor's formula and \eqref{eq:mixed-Hessian-zero-scale} prove
\eqref{eq:mixed-Hessian}.

Finally, Proposition \ref{prop:energy-expansion} gives
\[
D_{\boldsymbol\xi}^2J_\varepsilon(W_{0,\boldsymbol\xi})
=-M_\alpha D^2\mathcal F_m(\boldsymbol\xi)
+O(\varepsilon^\tau).
\]
Lemma \ref{lem:position-correction-C2} shows that the orthogonal
correction changes this block by $o(1)$. This proves
\eqref{eq:position-Hessian}.
\end{proof}

\subsection{Reduced equations and multiplier elimination}

The reduced gradient is related to the multiplier vector by the following matrix identity.

\begin{lemma}\label{lem:multipliers-gradient}
There is a $5m\times5m$ matrix
$\mathcal M_\varepsilon(\boldsymbol t,\boldsymbol\xi)$ such that
\begin{equation}\label{eq:gradient-multiplier}
\nabla_{(\boldsymbol t,\boldsymbol\xi)}
\widetilde J_\varepsilon
=\mathcal M_\varepsilon\boldsymbol c.
\end{equation}
After the position rows are divided by the corresponding $\mu_i$, the matrix converges to a block diagonal matrix with positive definite blocks
\[
\left(\int_{\mathbb R^4}FZ_kZ_\ell\,dz\right)_{0\leq k,\ell\leq4}.
\]
It is invertible for small $\varepsilon$. Consequently,
\begin{equation}\label{eq:gradient-multiplier-equivalence}
\nabla\widetilde J_\varepsilon=0
\quad\Longleftrightarrow\quad
c_{i,k}=0\ \hbox{for all }i,k.
\end{equation}
\end{lemma}

\begin{proof}
The projected equation gives, in $\mathcal H'$,
\begin{equation*}
J_\varepsilon'(u_{\boldsymbol t,\boldsymbol\xi})
=\sum_{i=1}^m\sum_{k=0}^4
c_{i,k}\chi_iF_{\mu_i,\xi_i}Z_{i,k}.
\end{equation*}
Let $p_a$ be one of the variables $t_i$ or $\xi_{i,\ell}$.
Differentiate \eqref{eq:reduced-energy}:
\begin{equation*}
\partial_{p_a}\widetilde J_\varepsilon
=\sum_{i,k}c_{i,k}\mathcal M_{a,(i,k)},
\qquad
\mathcal M_{a,(i,k)}
=\int_\Omega\chi_iF_{\mu_i,\xi_i}Z_{i,k}
\partial_{p_a}u_{\boldsymbol t,\boldsymbol\xi}\,dx.
\end{equation*}
This defines $\mathcal M_\varepsilon$ and proves
\eqref{eq:gradient-multiplier}.

Consider first $p_a=t_j$. By Lemma \ref{lem:scale-corrector},
\[
\partial_{t_j}u(\xi_j+z/\mu_j)=Z_0(z)+o(1)
\]
locally and the rescaling at every other core tends to zero. Hence
\begin{equation*}
\mathcal M_{t_j,(i,k)}
=\delta_{ij}\int_{\mathbb R^4}FZ_0Z_k\,dz+o(1).
\end{equation*}
For $p_a=\xi_{j,\ell}$, divide the row by $\mu_j$. Differentiation of
the bubble gives $-Z_\ell$ up to the fixed parameter orientation. The
differentiated orthogonality conditions imply that the correction has
no extra limiting kernel component. Therefore
\begin{equation*}
\mu_j^{-1}\mathcal M_{\xi_{j,\ell},(i,k)}
=\pm\delta_{ij}\int_{\mathbb R^4}FZ_\ell Z_k\,dz+o(1).
\end{equation*}
The off-diagonal terms are $o(1)$ by separation and the weighted
estimates. Lemma \ref{lem:kernel-Gram} shows that the limiting block is
diagonal and positive. Thus the normalized matrix is invertible for
small $\varepsilon$. Equation
\eqref{eq:gradient-multiplier-equivalence} follows from
\eqref{eq:gradient-multiplier}.
\end{proof}

The preceding estimates yield the reduced equations and their Hessian blocks.

\begin{proposition}\label{prop:reduced-equations}
For
$|\boldsymbol t|\leq C\varepsilon^\tau|\log\varepsilon|$,
\begin{equation}\label{eq:reduced-scale-equation}
\nabla_{\boldsymbol t}\widetilde J_\varepsilon
=\boldsymbol r_\varepsilon(\boldsymbol\xi)
+\mathcal A_\varepsilon(\boldsymbol t,\boldsymbol\xi)
\boldsymbol t,
\end{equation}
where
\begin{equation}\label{eq:scale-remainder}
\|\boldsymbol r_\varepsilon\|_{C^1}
\leq C\varepsilon^\tau
\end{equation}
and
\begin{equation}\label{eq:scale-matrix-uniform}
\mathcal A_\varepsilon(\boldsymbol t,\boldsymbol\xi)
=\frac{\gamma_\alpha}{|\log\varepsilon|}I_m
+o(|\log\varepsilon|^{-1})
\end{equation}
uniformly in the shrinking box. Moreover,
\begin{equation}\label{eq:reduced-position-equation}
\nabla_{\boldsymbol\xi}\widetilde J_\varepsilon
=-M_\alpha\nabla\mathcal F_m(\boldsymbol\xi)+o(1)
\end{equation}
uniformly in $C^1$ on compact subsets of the configuration space.
In the original variables
$(\boldsymbol t,\boldsymbol\xi)$, the Hessian has the block form
\begin{equation}\label{eq:reduced-Hessian-block}
D^2\widetilde J_\varepsilon
=\begin{pmatrix}
\displaystyle\frac{\gamma_\alpha}{|\log\varepsilon|}I_m
+o(|\log\varepsilon|^{-1})
& O(|\log\varepsilon|^{-1})+O(\varepsilon^\tau)\\[1.5ex]
O(|\log\varepsilon|^{-1})+O(\varepsilon^\tau)
&-M_\alpha D^2\mathcal F_m(\boldsymbol\xi)+o(1)
\end{pmatrix}.
\end{equation}
\end{proposition}

\begin{proof}
At $\boldsymbol t=0$, Proposition \ref{prop:scale-Hessian} gives
the scale Hessian and the mixed derivatives. The projected residual
equation gives
\begin{equation*}
\nabla_{\boldsymbol t}\widetilde J_\varepsilon(0,\boldsymbol\xi)
=\boldsymbol r_\varepsilon(\boldsymbol\xi),
\qquad
\|\boldsymbol r_\varepsilon\|_{C^1}\leq C\varepsilon^\tau.
\end{equation*}
For $\boldsymbol t$ in the shrinking box, use the integral form of
Taylor's formula:
\begin{equation*}
\nabla_{\boldsymbol t}\widetilde J_\varepsilon(\boldsymbol t,\boldsymbol\xi)
=
\boldsymbol r_\varepsilon(\boldsymbol\xi)
+\left[
\int_0^1
D_{\boldsymbol t}^2\widetilde J_\varepsilon
(s\boldsymbol t,\boldsymbol\xi)\,ds
\right]\boldsymbol t.
\end{equation*}
Equation \eqref{eq:uniform-scale-Hessian} gives
\eqref{eq:scale-matrix-uniform}. This proves
\eqref{eq:reduced-scale-equation}.

The position equation follows by differentiating
\eqref{eq:energy-expansion}. The correction in the position variables is controlled by Lemma
\ref{lem:position-correction-C2}. Its contribution is $o(1)$ in
$C^1$. The Hessian block expansion follows from
\eqref{eq:uniform-scale-Hessian}, \eqref{eq:mixed-Hessian}, and
\eqref{eq:position-Hessian}.
\end{proof}

\subsection{Critical points of the reduced functional and proof of Theorem \ref{thm:existence}}

\begin{proof}[Proof of Theorem \ref{thm:existence}]
Fix an admissible $\boldsymbol\xi$. Define
\begin{equation}\label{eq:scale-fixed-point-map}
\mathcal T_{\varepsilon,\boldsymbol\xi}(\boldsymbol t)
:=
\boldsymbol t
-\frac{|\log\varepsilon|}{\gamma_\alpha}
\nabla_{\boldsymbol t}\widetilde J_\varepsilon
(\boldsymbol t,\boldsymbol\xi).
\end{equation}
By \eqref{eq:scale-matrix-uniform},
\begin{equation*}
\|D_{\boldsymbol t}\mathcal T_{\varepsilon,\boldsymbol\xi}\|
=o(1)
\end{equation*}
uniformly in the shrinking box. Equation
\eqref{eq:scale-remainder} shows that this map sends a ball of radius
$C\varepsilon^\tau|\log\varepsilon|$ into itself when $C$ is large.
It is a contraction for small $\varepsilon$. Therefore there is a
unique solution
\begin{equation}\label{eq:t-solution}
\boldsymbol t_\varepsilon(\boldsymbol\xi)
=O(\varepsilon^\tau|\log\varepsilon|)
\end{equation}
of the scale equations. The parameter dependence in Proposition
\ref{prop:nonlinear-correction} and the implicit function theorem show
that $\boldsymbol t_\varepsilon$ is $C^1$ in $\boldsymbol\xi$.

Set
\begin{equation*}
\widehat J_\varepsilon(\boldsymbol\xi)
:=
\widetilde J_\varepsilon
(\boldsymbol t_\varepsilon(\boldsymbol\xi),\boldsymbol\xi).
\end{equation*}
Since the scale gradient vanishes,
\eqref{eq:reduced-position-equation} gives
\begin{equation}\label{eq:position-reduced-gradient}
\nabla\widehat J_\varepsilon
=-M_\alpha\nabla\mathcal F_m+o(1)
\end{equation}
uniformly on the set $D$ in Definition
\ref{def:stable-critical-point}. Degree invariance gives a critical
point $\boldsymbol\xi_\varepsilon\in D$. Shrinking $D$ around the
prescribed isolated stable point gives
\eqref{eq:center-convergence-intro}.

Lemma \ref{lem:multipliers-gradient} implies that all multipliers
vanish. Hence
\begin{equation*}
u_\varepsilon
=W_{\boldsymbol t_\varepsilon,\boldsymbol\xi_\varepsilon}
+\phi_\varepsilon
\end{equation*}
solves \eqref{eq:perturbed-problem}. The right-hand side is positive.
Set $v_\varepsilon=-\Delta u_\varepsilon$. Then
$-\Delta v_\varepsilon>0$ with $v_\varepsilon=0$ on
$\partial\Omega$. The Dirichlet maximum principle gives
$v_\varepsilon>0$. Applying it once more to
$-\Delta u_\varepsilon=v_\varepsilon$ gives
$u_\varepsilon>0$. We record the regularity argument. Lemma
\ref{lem:Adams-consequence} gives $e^{u_\varepsilon}\in L^q(\Omega)$
for every finite $q$. Choose $q>4/(4-\alpha)$. H\"older's inequality
shows that
\[
\sup_{x\in\Omega}\int_\Omega
\frac{K(y)e^{u_\varepsilon(y)}}{|x-y|^\alpha}\,dy<\infty.
\]
Hence the right-hand side of \eqref{eq:perturbed-problem} belongs to
$L^p(\Omega)$ for every finite $p$. Navier $L^p$ estimates give
$u_\varepsilon\in W^{4,p}(\Omega)$. Taking $p>4$ yields
$u_\varepsilon\in C^{3,\beta}(\overline\Omega)$. The density $K e^{u_\varepsilon}$ belongs to
$C^{0,\beta}(\overline\Omega)$ for some $\beta\in(0,1)$.
Lemma \ref{lem:boundary-Riesz-holder} therefore shows that its Riesz
potential belongs to $C^{0,\gamma}(\overline\Omega)$ for every
$\gamma<\min\{1,4-\alpha\}$. The right-hand side of
\eqref{eq:perturbed-problem} has the same regularity. Boundary
Schauder estimates for the Navier problem then yield
$u_\varepsilon\in C^{4,\gamma}(\overline\Omega)$. Thus the solution is
classical.

The scale formula follows from
\eqref{eq:leading-scale}, \eqref{eq:modulated-scale}, and
\eqref{eq:t-solution}. Estimate \eqref{eq:solution-expansion-intro}
follows from \eqref{eq:correction-sharp}.

The estimate \eqref{eq:solution-expansion-intro} also implies
$e^{u_\varepsilon}=e^{W_{\boldsymbol t_\varepsilon,
\boldsymbol\xi_\varepsilon}}(1+o(1))$ uniformly in $\Omega$. Using
$|e^s-1|\leq C|s|$ for small $s$, the bounded total nonlinear mass,
and the Hardy--Littlewood--Sobolev inequality, we obtain
\[
\|\mathcal N_\varepsilon(u_\varepsilon)
-\mathcal N_\varepsilon(W_{\boldsymbol t_\varepsilon,
\boldsymbol\xi_\varepsilon})\|_{L^1(\Omega)}=o(1).
\]
Let $\zeta\in C(\overline\Omega)$. In the $i$-th core, use
$x=\xi_{i,\varepsilon}+z/\mu_{i,\varepsilon}$ and
\eqref{eq:core-nonlinearity}. Dominated convergence gives
\begin{equation}\label{eq:source-core-limit}
\int_{B_{\eta/4}(\xi_{i,\varepsilon})}
\zeta(x)\mathcal N_\varepsilon(u_\varepsilon)(x)\,dx
\longrightarrow
M_\alpha\zeta(\xi_i^*).
\end{equation}
The neck and outer integrals tend to zero by
\eqref{eq:residual-small} and \eqref{eq:theta-L1}. Summing
\eqref{eq:source-core-limit} proves
\eqref{eq:measure-convergence-intro}.

Green representation gives
\begin{equation*}
u_\varepsilon(x)
=\int_\Omega
G(x,y)\mathcal N_\varepsilon(u_\varepsilon)(y)\,dy.
\end{equation*}
Fix a compact set disjoint from the limiting centers. Split the
integral into small peak balls and their complement. On a peak ball,
the $x$-derivatives of $G(x,y)$ are smooth and the first moment of the
rescaled source tends to zero. On the complement, the weighted
estimate gives convergence to zero in $L^p$ for every finite $p$.
Differentiation in $x$, followed by the local Hölder estimate for the
Riesz potential, gives \eqref{eq:outer-convergence-intro} for the
range \eqref{eq:outer-beta-range}. Near $\partial\Omega$ the same
argument uses the Navier boundary estimates.

Finally, evaluate \eqref{eq:projection-expansion} at
$x=\xi_{i,\varepsilon}$. The self term is
$2b_\alpha\log\mu_{i,\varepsilon}
+M_\alpha H(\xi_{i,\varepsilon},\xi_{i,\varepsilon})+o(1)$.
By \eqref{eq:pair-projection}, every other bubble contributes
$M_\alpha G(\xi_{i,\varepsilon},\xi_{j,\varepsilon})+o(1)$.
Insert \eqref{eq:scales-intro}. This proves
\eqref{eq:peak-height-intro}.
\end{proof}

\section{Local uniqueness and nondegeneracy}

\subsection{Uniform modulation}

We first establish a uniform modulation decomposition near the bubble manifold.

\begin{lemma}\label{lem:modulation}
Fix a compact subset of $\mathcal C_{m,\eta}(\Omega)$. There is $r_0>0$, independent of $\varepsilon$, such that every function $u$ satisfying
\[
\inf_{\boldsymbol t,\boldsymbol\xi}
\left(\|u-W_{\boldsymbol t,\boldsymbol\xi}\|_*
+\|u-W_{\boldsymbol t,\boldsymbol\xi}\|_\mathcal H\right)<r_0
\]
has, after a permutation of the peaks, a unique decomposition
\[
u=W_{\boldsymbol t,\boldsymbol\xi}+\phi
\]
with $\phi$ satisfying \eqref{eq:orthogonality}. The parameter map is $C^2$. The natural position coordinates are
$\mu_i(\eta_i-\xi_i)$.
\end{lemma}

\begin{proof}
Let $\mathcal G_{i,k}(\boldsymbol t,\boldsymbol\xi,u)$ be the left-hand side of
\eqref{eq:orthogonality} with
$\phi=u-W_{\boldsymbol t,\boldsymbol\xi}$. Differentiate with respect to $t_j$ and the scaled coordinate
$\mu_j\xi_{j,\ell}$. At a bubble configuration, the Jacobian converges to a block diagonal matrix. Each block is, up to fixed signs,
\[
\left(\int_{\mathbb R^4}FZ_kZ_\ell\,dz\right)_{0\leq k,\ell\leq4}.
\]
It is invertible by Lemma \ref{lem:kernel-Gram}. The convergence is uniform. The implicit function theorem gives the parameters in a ball whose radius is independent of $\varepsilon$. Separation makes different labelings disjoint. Only permutations remain.
\end{proof}

An exact solution in the modulation neighborhood automatically satisfies the sharper smallness estimate below.

\begin{lemma}\label{lem:uniqueness-bootstrap}
Let
\[
v_\varepsilon
=W_{\boldsymbol s,\boldsymbol\eta}+\psi
\]
be an exact solution of \eqref{eq:perturbed-problem}. Assume that
$\psi$ satisfies \eqref{eq:orthogonality},
$|\boldsymbol s|\leq
C_0\varepsilon^\tau|\log\varepsilon|$, and
\[
\|\psi\|_*+\|\psi\|_{\mathcal H}\leq r_1.
\]
If $r_1$ is sufficiently small, then
\begin{equation}\label{eq:uniqueness-bootstrap}
\|\psi\|_*+\|\psi\|_{\mathcal H}
\leq
C\bigl(\varepsilon^\tau+|\boldsymbol s|\bigr)
\leq C\varepsilon^\tau|\log\varepsilon|.
\end{equation}
\end{lemma}

\begin{proof}
Since $v_\varepsilon$ is an exact solution, its orthogonal remainder
satisfies
\[
\mathcal L_\varepsilon\psi
=-\mathcal R_{\boldsymbol s,\boldsymbol\eta}
+\mathcal Q_\varepsilon(\psi)
\]
with no multiplier term. Proposition \ref{prop:linear-estimate},
Proposition \ref{prop:residual}, and Lemma
\ref{lem:nonlinear-estimates} give
\[
\|\psi\|_*+\|\psi\|_{\mathcal H}
\leq
C\bigl(\varepsilon^\tau+|\boldsymbol s|\bigr)
+C\bigl(\|\psi\|_*+\|\psi\|_{\mathcal H}\bigr)^2.
\]
Choose $r_1$ so that $Cr_1\leq1/2$. Absorb the quadratic term.
This proves the first inequality in
\eqref{eq:uniqueness-bootstrap}. The scale assumption gives the
second one.
\end{proof}

The modulation and reduced equations imply local uniqueness in the natural bubbling class.

\begin{proposition}\label{prop:local-uniqueness}
If $D^2\mathcal F_m(\boldsymbol\xi^*)$ is nonsingular, the solution
from Theorem \ref{thm:existence} is locally unique among the
$m$-bubble solutions satisfying
\eqref{eq:uniqueness-scale-class}--\eqref{eq:branch-distance}.
\end{proposition}

\begin{proof}
Let
$v_\varepsilon=W_{\boldsymbol s,\boldsymbol\eta}+\psi$ be a second
solution in the class of the proposition. Lemma
\ref{lem:modulation} fixes its parameters, up to a permutation, by
the same orthogonality conditions as the constructed branch.
The distance assumption and
\eqref{eq:solution-expansion-intro} imply
$\|\psi\|_*+\|\psi\|_{\mathcal H}<r_1$ after $r$ is reduced.
Lemma \ref{lem:uniqueness-bootstrap} places $\psi$ in the shrinking
fixed-point ball. Proposition \ref{prop:nonlinear-correction} now
gives uniqueness of the orthogonal correction for fixed parameters.
Thus
\begin{equation}\label{eq:uniqueness-correction-identification}
\psi=\phi_\varepsilon(\boldsymbol s,\boldsymbol\eta).
\end{equation}
Since $v_\varepsilon$ is an exact solution, all its multipliers
vanish. Lemma \ref{lem:multipliers-gradient} implies
\begin{equation*}
\nabla\widetilde J_\varepsilon
(\boldsymbol s,\boldsymbol\eta)=0.
\end{equation*}

The scale assumption \eqref{eq:uniqueness-scale-class} places
$\boldsymbol s$ in the shrinking box. The contraction argument in
\eqref{eq:scale-fixed-point-map} gives the unique scale vector
$\boldsymbol t_\varepsilon(\boldsymbol\eta)$. Hence
\begin{equation}\label{eq:uniqueness-scale-identification}
\boldsymbol s=\boldsymbol t_\varepsilon(\boldsymbol\eta).
\end{equation}
After the scales are eliminated, the position gradient is the
$C^1$ perturbation in \eqref{eq:position-reduced-gradient}.
Nonsingularity of $D^2\mathcal F_m(\boldsymbol\xi^*)$ and the inverse
function theorem give a unique zero in the fixed neighborhood
selected by \eqref{eq:branch-distance}. Thus
$\boldsymbol\eta=\boldsymbol\xi_\varepsilon$, after the same
permutation. Equations
\eqref{eq:uniqueness-correction-identification} and
\eqref{eq:uniqueness-scale-identification} now give
$v_\varepsilon=u_\varepsilon$.
\end{proof}

\subsection{Normal and tangent blocks}

Let $u_\varepsilon$ be the exact solution. Its quadratic form is
\begin{equation*}
\begin{aligned}
\mathcal Q_\varepsilon(v,w)
={}&\int_\Omega\Delta v\Delta w\,dx\\
&-\frac{\varepsilon^{8-\alpha}}2
\int_\Omega\int_\Omega
\frac{K(x)K(y)e^{u_\varepsilon(x)+u_\varepsilon(y)}}{|x-y|^\alpha}
(v(x)+v(y))(w(x)+w(y))\,dxdy.
\end{aligned}
\end{equation*}
The two coefficient functions in the local part belong to every
finite $L^q(\Omega)$, uniformly along the constructed branch. Since
$\mathcal H\hookrightarrow L^q(\Omega)$ is compact for every finite
$q$, multiplication by these coefficients is compact from
$\mathcal H$ to $\mathcal H'$. The same compactness for the Riesz
term follows by first using the compact embedding in each variable
and then applying Lemma \ref{lem:HLS-domain}. Thus the exact
linearized realization is a compact perturbation of the Navier
biharmonic operator; it is self-adjoint with compact resolvent, and
its variational Morse index equals its spectral Morse index.
Let $\mathcal X_\varepsilon$ be the subspace of $\mathcal H$ satisfying the five orthogonality conditions at each peak. Let $\mathcal T_\varepsilon$ be the span of the $5m$ tangent
vectors
\[
Y_{i,\varepsilon}=\partial_{t_i}u,
\qquad
X_{i,\ell,\varepsilon}=\partial_{\xi_{i,\ell}}u.
\]
The normalized vectors
$\mu_i^{-1}X_{i,\ell,\varepsilon}$ are used only in the
independence argument. The finite-dimensional Hessian is written in
the original variables $(\boldsymbol t,\boldsymbol\xi)$.

At an exact solution, the tangent and normal spaces are exactly orthogonal for the quadratic form.

\begin{lemma}\label{lem:tangent-normal-reduction}
At the exact parameter
$p_\varepsilon=(\boldsymbol t_\varepsilon,\boldsymbol\xi_\varepsilon)$,
one has
\begin{equation}\label{eq:tangent-normal-direct-sum}
\mathcal H=\mathcal X_\varepsilon\oplus\mathcal T_\varepsilon.
\end{equation}
The restriction of $\mathcal Q_\varepsilon$ to
$\mathcal X_\varepsilon$ is nondegenerate and defines an isomorphism
from $\mathcal X_\varepsilon$ to its dual. Moreover,
\begin{equation}\label{eq:exact-tangent-normal-orthogonality}
\mathcal Q_\varepsilon(z,x)=0
\qquad
\text{for every }z\in\mathcal T_\varepsilon,\quad
x\in\mathcal X_\varepsilon.
\end{equation}
In the tangent basis,
\begin{equation}\label{eq:tangent-Hessian-exact}
\mathcal Q_\varepsilon
(\partial_{p_a}u_\varepsilon,\partial_{p_b}u_\varepsilon)
=
\partial_{p_ap_b}^2
\widetilde J_\varepsilon(p_\varepsilon).
\end{equation}
\end{lemma}

\begin{proof}
The matrix obtained by applying the $5m$ orthogonality functionals
to the $5m$ tangent vectors converges, after normalization of the
position columns, to the block diagonal Gram matrix in Lemma
\ref{lem:kernel-Gram}. It is invertible. This proves
\eqref{eq:tangent-normal-direct-sum}.

For the normal block, regard the projected operator as the map from
$\mathcal X_\varepsilon$ to $\mathcal X_\varepsilon'$ induced by its
bilinear form. Proposition \ref{prop:linear-estimate} gives a uniform
inverse for the corresponding form at
$W_{\boldsymbol t_\varepsilon,\boldsymbol\xi_\varepsilon}$. Since
$\|u_\varepsilon-W_{\boldsymbol t_\varepsilon,
\boldsymbol\xi_\varepsilon}\|_{L^\infty}\to0$, the two local
coefficients and the nonlocal kernel in the exact second variation
converge in operator norm from $\mathcal H$ to $\mathcal H'$ by the
Hardy--Littlewood--Sobolev inequality. A Neumann-series argument
therefore gives an isomorphism on $\mathcal X_\varepsilon$.

Write the projected equation along the solution manifold as
\begin{equation}\label{eq:projected-first-variation}
J_\varepsilon'(u(p))
=\sum_{\beta=1}^{5m}c_\beta(p)\Psi_\beta(p)
\quad\hbox{in }\mathcal H',
\end{equation}
where $\Psi_\beta$ are the functions
$\chi_iF_{\mu_i,\xi_i}Z_{i,k}$. At the exact parameter,
$c_\beta(p_\varepsilon)=0$. Differentiate
\eqref{eq:projected-first-variation} with respect to $p_a$ and pair
the result with $x\in\mathcal X_\varepsilon$. We obtain
\begin{equation*}
\begin{aligned}
\mathcal Q_\varepsilon(\partial_{p_a}u_\varepsilon,x)
={}&\sum_\beta
\partial_{p_a}c_\beta(p_\varepsilon)
\int_\Omega\Psi_\beta(p_\varepsilon)x\,dx +\sum_\beta c_\beta(p_\varepsilon)
\int_\Omega\partial_{p_a}\Psi_\beta(p_\varepsilon)x\,dx=0.
\end{aligned}
\end{equation*}
The first sum vanishes by the definition of
$\mathcal X_\varepsilon$. The second vanishes because all multipliers
are zero. This proves
\eqref{eq:exact-tangent-normal-orthogonality}. Finally,
\eqref{eq:tangent-Hessian-exact} is Lemma
\ref{lem:Hessian-identity}.
\end{proof}

The block decomposition now gives nondegeneracy of the constructed solution.

\begin{proposition}\label{prop:nondegenerate-branch}
If $D^2\mathcal F_m(\boldsymbol\xi^*)$ is nonsingular, then the exact linearized operator at $u_\varepsilon$ has trivial kernel for all sufficiently small $\varepsilon$.
\end{proposition}

\begin{proof}
By Lemma \ref{lem:tangent-normal-reduction}, the exact quadratic form
is block diagonal with respect to
$\mathcal H=\mathcal X_\varepsilon\oplus\mathcal T_\varepsilon$.
The normal block is invertible. It remains to examine the tangent
Hessian.

At $p_\varepsilon$, Proposition \ref{prop:reduced-equations} gives
\begin{equation*}
D^2\widetilde J_\varepsilon(p_\varepsilon)
=\begin{pmatrix}
A_\varepsilon&B_\varepsilon\\
B_\varepsilon^T&C_\varepsilon
\end{pmatrix},
\end{equation*}
where
\begin{equation*}
\begin{aligned}
A_\varepsilon
&=\frac{\gamma_\alpha}{|\log\varepsilon|}I_m
+o(|\log\varepsilon|^{-1}),\\
\|B_\varepsilon\|
&=O(|\log\varepsilon|^{-1})+O(\varepsilon^\tau),\\
C_\varepsilon
&=-M_\alpha D^2\mathcal F_m(\boldsymbol\xi^*)+o(1).
\end{aligned}
\end{equation*}
The position block is invertible with uniformly bounded inverse. Its
Schur complement satisfies
\begin{equation*}
A_\varepsilon-B_\varepsilon C_\varepsilon^{-1}B_\varepsilon^T
=
\frac{\gamma_\alpha}{|\log\varepsilon|}I_m
+o(|\log\varepsilon|^{-1}).
\end{equation*}
It is positive definite. Thus the tangent block is invertible. Exact
orthogonality of the two blocks proves that the full linearized
operator has trivial kernel.
\end{proof}

\section{Morse index}

\subsection{Inertia of the normal block}

For the rest of this subsection, set
\begin{equation}\label{eq:positive-normal-subspace}
\mathcal Y_\varepsilon
:=\left\{
v\in\mathcal X_\varepsilon:
\int_\Omega\chi_iF_{\mu_i,\xi_i}v\,dx=0,
\quad 1\leq i\leq m
\right\}.
\end{equation}
Thus functions in $\mathcal Y_\varepsilon$ are orthogonal to the
constant negative mode and to the five conformal modes at every core.

The normal Morse index requires a localization estimate adapted to
the critical four-dimensional scale. We use a logarithmically slow
partition rather than a Hardy--Rellich inequality.

\begin{lemma}\label{lem:quadratic-localization}
Let $\varepsilon_n\to0$ and let
$v_n\in\mathcal Y_{\varepsilon_n}$ satisfy
$\|v_n\|_{\mathcal H}\leq1$. For every large $R$, there are disjoint
logarithmic cutoffs $\eta_{i,n,R}$ such that
\[
\eta_{i,n,R}=1
\quad\text{in }B_{R/\mu_{i,n}}(\xi_{i,n}),
\qquad
\eta_{i,n,R}=0
\quad\text{outside }B_{R^2/\mu_{i,n}}(\xi_{i,n}),
\]
and
\begin{equation}\label{eq:quadratic-localization-cutoff}
|\nabla^k\eta_{i,n,R}(x)|
\leq
\frac{C}{|\log R|\,|x-\xi_{i,n}|^k},
\qquad k=1,2.
\end{equation}
Put
\[
v_{i,n}=\eta_{i,n,R}v_n,
\qquad
v_{0,n}=v_n-\sum_{i=1}^m v_{i,n}.
\]
Then
\begin{equation}\label{eq:quadratic-localization}
\begin{aligned}
\mathcal Q_{\varepsilon_n}(v_n,v_n)
\geq{}&
\sum_{i=1}^m
\mathcal Q_{\varepsilon_n}(v_{i,n},v_{i,n})
+(1-\delta_R)\|v_{0,n}\|_{\mathcal H}^2 -\delta_R\|v_n\|_{\mathcal H}^2-o_n(1),
\end{aligned}
\end{equation}
where $\delta_R\to0$ as $R\to\infty$. Moreover, after rescaling the
$i$-th core,
\[
\widetilde v_{i,n}(z)
:=v_{i,n}\left(\xi_{i,n}+\frac z{\mu_{i,n}}\right),
\]
one has
\begin{equation}\label{eq:quadratic-core-limit}
\mathcal Q_{\varepsilon_n}(v_{i,n},v_{i,n})
=
\mathcal Q_\infty(\widetilde v_{i,n},\widetilde v_{i,n})
+o_n(1)+o_R(1).
\end{equation}
\end{lemma}

\begin{proof}
We first localize the biharmonic energy. In the cylinder associated
with the $i$-th center, set
\[
s=\log(\mu_{i,n}|x-\xi_{i,n}|),
\qquad
\theta=\frac{x-\xi_{i,n}}{|x-\xi_{i,n}|},
\qquad
\widetilde v(s,\theta)=v(x).
\]
On every annulus,
\begin{equation}\label{eq:Emden-Fowler-biharmonic}
\int |\Delta v|^2\,dx
=\int
\left|\left(\partial_s^2+2\partial_s
+\Delta_{\mathbb S^3}\right)\widetilde v\right|^2
\,dsd\theta.
\end{equation}
Choose a fixed smooth function $\zeta$ with $\zeta=1$ on
$(-\infty,0]$ and $\zeta=0$ on $[1,\infty)$, and define
\[
\eta_{i,n,R}(x)
=\zeta\left(
\frac{\log(\mu_{i,n}|x-\xi_{i,n}|)-\log R}
{\log R}
\right).
\]
This gives \eqref{eq:quadratic-localization-cutoff}. For large $n$ the
supports are disjoint.

We record the one-dimensional estimate used for the logarithmic
cutoff. Expand in an orthonormal spherical-harmonic basis. We suppress
the multiplicity index and write
$\widetilde v_n=\sum_{k\geq0}a_{k,n}(s)Y_k(\theta)$. Put
$P_k=\partial_s^2+2\partial_s-k(k+2)$. Direct expansion gives
\[
P_k(\eta a)=\eta P_ka+2\eta'a'
+(\eta''+2\eta')a.
\]
For $k\geq2$, the roots of the characteristic polynomial are
separated from the imaginary axis uniformly in $k$. Integration by
parts therefore gives
\begin{equation}\label{eq:high-mode-logarithmic-cutoff}
\begin{aligned}
&\|P_k(\eta a_{k,n})\|_{L^2}^2
+\|P_k((1-\eta)a_{k,n})\|_{L^2}^2\leq
\|P_ka_{k,n}\|_{L^2}^2
+\frac{C}{\log R}\|P_ka_{k,n}\|_{L^2}^2.
\end{aligned}
\end{equation}

We treat $k=0,1$ separately.  The homogeneous solutions of the
pure biharmonic operators are
\[
1,\ e^{-2s}\quad(k=0),
\qquad
e^s,\ e^{-3s}\quad(k=1).
\]
They must not be confused with the conformal kernel of
$\mathcal L_\infty$.  Put
\[
I_R=(\log R,2\log R),
\qquad
I_R^*=\left(\frac12\log R,\frac52\log R\right).
\]
The following one-dimensional estimate is the low-mode substitute
for the spectral gap.  For $k=1$, the four multiplicity indices are
included in all sums:
\begin{equation}\label{eq:low-mode-variation-bound}
\begin{aligned}
&\sum_{k=0}^1\int_{I_R}
\bigl(|a_{k,n}'|^2+(1+k)^2|a_{k,n}|^2\bigr)\,ds\\
&\leq C\log R\Bigg(
\sum_{k=0}^1\|P_ka_{k,n}\|_{L^2(I_R^*)}^2
+|\mathfrak m_{*,n}|^2
+\sum_{\nu=0}^4|\mathfrak m_{\nu,n}|^2
+o_R(1)\|v_n\|_{\mathcal H}^2\Bigg)+o_n(1).
\end{aligned}
\end{equation}
where
\[
\mathfrak m_{*,n}=\int_\Omega
\chi_iF_{\mu_{i,n},\xi_{i,n}}v_n,
\qquad
\mathfrak m_{\nu,n}=\int_\Omega
\chi_iF_{\mu_{i,n},\xi_{i,n}}Z_{i,\nu}v_n,
\quad0\leq\nu\leq4.
\]

For completeness, we justify this estimate.  Variation of constants
on $I_R^*$ writes
\[
\begin{aligned}
a_{0,n}(s)
&=A_{0,n}+B_{0,n}e^{-2s}
+\int_{I_R^*}K_0(s,\sigma)P_0a_{0,n}(\sigma)\,d\sigma+r_{0,n}(s),\\
a_{1,n,\ell}(s)
&=A_{1,n,\ell}e^s+B_{1,n,\ell}e^{-3s}
+\int_{I_R^*}K_1(s,\sigma)P_1a_{1,n,\ell}(\sigma)\,d\sigma
+r_{1,n,\ell}(s),
\qquad 1\leq\ell\leq4,
\end{aligned}
\]
where the remainders are determined by the energy on the two
adjacent cylinder strips.  The Green kernels satisfy
\[
|K_0|+|\partial_sK_0|\leq C(1+|s-\sigma|),
\qquad
|K_1|+|\partial_sK_1|
\leq Ce^{-c|s-\sigma|}
\]
after the growing $e^s$ component has been separated.  The
coefficients of $e^{-2s}$ and $e^{-3s}$ are controlled by the
$H^2$-energy on the inner strip.  The coefficient of $e^s$ is
$O(\mu_{i,n}^{-1})\|v_n\|_{\mathcal H}$, because in the original
variables it is an affine harmonic function and
$\mathcal H\hookrightarrow H_0^1(\Omega)$.  Finally, evaluating the
constant and the remaining degree-zero and degree-one components
against the six weighted test functions gives
\[
|A_{0,n}|+\sum_{\ell=1}^4|A_{1,n,\ell}|
\leq C\left(
|\mathfrak m_{*,n}|+\sum_{\nu=0}^4|\mathfrak m_{\nu,n}|
+\sum_{k=0}^1\|P_ka_{k,n}\|_{L^2(I_R^*)}\right)
+o_R(1)\|v_n\|_{\mathcal H}+o_n(1).
\]
Here Lemma \ref{lem:weighted-low-mode-control} is used to pass from
the weighted core norms to the coefficients.  Substitution in the
two variation-of-constants formulas and integration over $I_R$
proves \eqref{eq:low-mode-variation-bound}.  The factor $\log R$ is
unavoidable for the constant plateau.

For $v_n\in\mathcal Y_{\varepsilon_n}$ the mass moment
$\mathfrak m_{*,n}$ and all the moments $\mathfrak m_{\nu,n}$ vanish. Since
$|\eta'|\leq C/(\log R)$ and
$|\eta''|\leq C/(\log R)^2$, equations
\eqref{eq:high-mode-logarithmic-cutoff} and
\eqref{eq:low-mode-variation-bound} imply
\begin{equation}\label{eq:cylindrical-transition-control}
\mathcal E_{i,n,R}^{\rm cut}
\leq
\frac{C}{(\log R)^{1/2}}\|v_n\|_{\mathcal H}^2
+o_n(1)+o_R(1),
\end{equation}
where $\mathcal E_{i,n,R}^{\rm cut}$ is the sum of the commutator
and cross terms produced by $\Delta(\eta_{i,n,R}v_n)$.  Indeed, the
high modes give the stronger factor $(\log R)^{-1}$ by
\eqref{eq:high-mode-logarithmic-cutoff}.  For the low modes,
Cauchy's inequality, the factor $(\log R)^{-1}$ carried by
$\eta'$, and \eqref{eq:low-mode-variation-bound} give the displayed
$(\log R)^{-1/2}$ bound.  This loss is immaterial and avoids any
unjustified sharp cutoff rate.
Summing the spherical harmonics in
\eqref{eq:Emden-Fowler-biharmonic} gives
\begin{equation}\label{eq:quadratic-local-energy}
\begin{aligned}
\int_\Omega|\Delta v_n|^2\,dx
\geq{}&
\sum_i\int_\Omega|\Delta v_{i,n}|^2\,dx
+(1-\delta_R)\int_\Omega|\Delta v_{0,n}|^2\,dx -\delta_R\|v_n\|_{\mathcal H}^2-o_n(1),
\end{aligned}
\end{equation}
with $\delta_R\to0$. This is the required critical logarithmic
localization.

We next localize the nonlocal form. Write
\[
d\Gamma_n(x,y)
:=\varepsilon_n^{8-\alpha}
\frac{K(x)K(y)e^{u_{\varepsilon_n}(x)+u_{\varepsilon_n}(y)}}
{|x-y|^\alpha}\,dxdy
\]
and
\[
\mathcal B_n(v,w)
:=\frac12\int_{\Omega\times\Omega}
(v(x)+v(y))(w(x)+w(y))\,d\Gamma_n(x,y).
\]
Then
$\mathcal Q_{\varepsilon_n}(v,w)
=\int_\Omega\Delta v\Delta w-\mathcal B_n(v,w)$. Let
\[
B_{i,n}=B_{R^2/\mu_{i,n}}(\xi_{i,n}),
\qquad
B_{i,n}^{\rm in}=B_{R/\mu_{i,n}}(\xi_{i,n}).
\]
The core expansion and Lemma
\ref{lem:separated-core-estimates} give
\begin{align}
\Gamma_n(B_{i,n}\times B_{j,n})
&\leq C\varepsilon_n^\alpha,
\qquad i\ne j,
\notag\\
\Gamma_n\left(
\left(\Omega\setminus\bigcup_iB_{i,n}^{\rm in}\right)
\times\Omega\right)
&\leq\delta_R+o_n(1).
\label{eq:Gamma-outer-mass}
\end{align}
The moment conditions and Lemma
\ref{lem:weighted-low-mode-control} give
\begin{equation*}
\sum_{i=1}^m
\int_{B_{i,n}}F_{\mu_{i,n},\xi_{i,n}}v_n^2\,dx
\leq C\|v_n\|_{\mathcal H}^2+o_n(1).
\end{equation*}
Cauchy's inequality with respect to $d\Gamma_n$ now shows that the
sum of the distinct-core, core--exterior, and transition errors is
bounded by
\begin{equation}\label{eq:quadratic-nonlocal-error}
|\mathcal E_{n,R}^{\rm nl}|
\leq(\delta_R+o_n(1))\|v_n\|_{\mathcal H}^2.
\end{equation}
For transition terms we also use
\eqref{eq:cylindrical-transition-control}. Combining
\eqref{eq:quadratic-local-energy} and
\eqref{eq:quadratic-nonlocal-error} proves
\eqref{eq:quadratic-localization}.

For \eqref{eq:quadratic-core-limit}, set
$x=\xi_{i,n}+z/\mu_{i,n}$. The local coefficients converge on every
fixed ball and the projected correction converges in
$C^{3,\beta}_{\rm loc}$. The part with $|z|>R$ is $o_R(1)$ by
\eqref{eq:Gamma-outer-mass}, and a distinct-core interaction is
$O(\varepsilon_n^\alpha)$. This proves
\eqref{eq:quadratic-core-limit}.
\end{proof}

The normal block has exactly one negative direction for each bubble.

\begin{proposition}\label{prop:normal-inertia}
For all sufficiently small $\varepsilon$, the restriction of
$\mathcal Q_\varepsilon$ to $\mathcal X_\varepsilon$ is
nondegenerate and has Morse index $m$.  The space
$\mathcal Y_\varepsilon$ in \eqref{eq:positive-normal-subspace} has
codimension $m$ in $\mathcal X_\varepsilon$, and there is $c>0$ such
that
\begin{equation}\label{eq:positive-normal-coercivity}
\mathcal Q_\varepsilon(v,v)
\geq c\|v\|_{\mathcal H}^2
\qquad
\text{for every }v\in\mathcal Y_\varepsilon.
\end{equation}
Consequently there is a
$\mathcal Q_\varepsilon$-orthogonal decomposition
\begin{equation}\label{eq:normal-inertia-decomposition}
\mathcal X_\varepsilon
=\mathcal E_\varepsilon^-\oplus\mathcal E_\varepsilon^+,
\qquad
\dim\mathcal E_\varepsilon^-=m,
\end{equation}
on which the quadratic form is uniformly negative and uniformly
positive, respectively.
\end{proposition}

\begin{proof}
We first construct $m$ negative directions.  Let
$V_{\rho_{i,\varepsilon},\xi_i}$ be the capacity minimizer in Lemma
\ref{lem:biharmonic-capacity}, extended by $-b_\alpha$ in the inner
ball, and put
\begin{equation*}
\Pi_{i,\varepsilon}
:=-b_\alpha^{-1}V_{\rho_{i,\varepsilon},\xi_i}.
\end{equation*}
Then $\Pi_{i,\varepsilon}=1$ in
$B_{\rho_{i,\varepsilon}}(\xi_i)$ and
\begin{equation}\label{eq:negative-capacity-energy}
\int_\Omega|\Delta\Pi_{i,\varepsilon}|^2\,dx
=O(|\log\varepsilon|^{-1}).
\end{equation}
The core nonlocal term converges to the whole-space constant mode.
Using \eqref{eq:negative-mode-value},
\eqref{eq:negative-capacity-energy}, and the separated-core
estimates, we obtain
\begin{equation*}
\mathcal Q_\varepsilon
(\Pi_{i,\varepsilon},\Pi_{j,\varepsilon})
=-2M_\alpha\delta_{ij}+o(1).
\end{equation*}
Project these functions onto $\mathcal X_\varepsilon$.  Choose
$a_{i;j,k}$ so that
\begin{equation*}
P_{i,\varepsilon}
:=\Pi_{i,\varepsilon}
-\sum_{j=1}^{m}\sum_{k=0}^{4}
a_{i;j,k}\widehat Z_{j,k}
\in\mathcal X_\varepsilon.
\end{equation*}
The matrix is \eqref{eq:linear-multiplier-matrix}, and
\begin{equation*}
\int_\Omega
\chi_jF_{\mu_j,\xi_j}Z_{j,k}\Pi_{i,\varepsilon}\,dx
=\delta_{ij}\int_{\mathbb R^4}FZ_k\,dz+o(1)=o(1).
\end{equation*}
Hence $a_{i;j,k}=o(1)$, and
\begin{equation*}
\mathcal Q_\varepsilon
(P_{i,\varepsilon},P_{j,\varepsilon})
=-2M_\alpha\delta_{ij}+o(1).
\end{equation*}
Thus the normal index is at least $m$.

The matrix
$\int\chi_iF_{\mu_i,\xi_i}P_{j,\varepsilon}$ converges to
$M_\alpha I_m$.  Hence the $m$ mass functionals are independent on
$\mathcal X_\varepsilon$, and
$\mathcal Y_\varepsilon$ has codimension $m$.

Suppose that \eqref{eq:positive-normal-coercivity} fails.  Then there
are $\varepsilon_n\to0$ and
$v_n\in\mathcal Y_{\varepsilon_n}$ such that
\begin{equation}\label{eq:normal-coercivity-contradiction}
\|v_n\|_{\mathcal H}=1,
\qquad
\mathcal Q_{\varepsilon_n}(v_n,v_n)\leq o_n(1).
\end{equation}
Apply Lemma \ref{lem:quadratic-localization}.  Let $v_{i,n}$ be the
localized core pieces and set
\begin{equation*}
w_{i,n}(z)
:=v_{i,n}\left(\xi_{i,n}+\frac z{\mu_{i,n}}\right).
\end{equation*}
Lemma \ref{lem:weighted-low-mode-control}, the six global moment
conditions, and the cutoff construction give
\begin{equation*}
\int_{\mathbb R^4}Fw_{i,n}^2\,dz\leq C.
\end{equation*}
Define
\begin{equation*}
a_{i,n}
:=\frac1{M_\alpha}\int_{\mathbb R^4}Fw_{i,n}\,dz,
\qquad
b_{i,k,n}
:=\frac{\int_{\mathbb R^4}FZ_kw_{i,n}\,dz}
{\int_{\mathbb R^4}FZ_k^2\,dz}.
\end{equation*}
The coefficients are small.  For example,
\begin{equation*}
\begin{aligned}
|a_{i,n}|
&\leq o_n(1)
+C\left(\int_{|z|>R}F\,dz\right)^{1/2}
\left(\int_{\mathbb R^4}Fw_{i,n}^2\,dz\right)^{1/2},
\end{aligned}
\end{equation*}
and, for $0\leq k\leq4$,
\begin{equation*}
\begin{aligned}
|b_{i,k,n}|
&\leq o_n(1)
+C\left(\int_{|z|>R}FZ_k^2\,dz\right)^{1/2}
\left(\int_{\mathbb R^4}Fw_{i,n}^2\,dz\right)^{1/2}.
\end{aligned}
\end{equation*}
Thus
\begin{equation*}
|a_{i,n}|+\sum_{k=0}^{4}|b_{i,k,n}|
\leq o_R(1)+o_n(1).
\end{equation*}
This is the required tail control; no $L^\infty$ bound is used.

Set
\begin{equation*}
\widehat w_{i,n}
:=w_{i,n}-a_{i,n}-\sum_{k=0}^{4}b_{i,k,n}Z_k.
\end{equation*}
Then the six moments of $\widehat w_{i,n}$ vanish.  Theorem
\ref{thm:limit-input} gives
\begin{equation*}
\mathcal Q_\infty(\widehat w_{i,n},\widehat w_{i,n})
\geq c_\alpha
\int_{\mathbb R^4}|\Delta\widehat w_{i,n}|^2\,dz.
\end{equation*}
Since the $Z_k$ are kernel functions and
$\mathcal Q_\infty(1,1)=-2M_\alpha$, the smallness of the coefficients
implies
\begin{equation}\label{eq:localized-direct-coercivity}
\mathcal Q_\infty(w_{i,n},w_{i,n})
\geq
\frac{c_\alpha}{2}
\int_{\mathbb R^4}|\Delta w_{i,n}|^2\,dz
-o_R(1)-o_n(1).
\end{equation}

Use \eqref{eq:quadratic-core-limit} and
\eqref{eq:localized-direct-coercivity} in
\eqref{eq:quadratic-localization}.  We obtain
\begin{equation}\label{eq:normal-global-coercivity-sequence}
\begin{aligned}
\mathcal Q_{\varepsilon_n}(v_n,v_n)
\geq
c\sum_{i=1}^{m}\|v_{i,n}\|_{\mathcal H}^2
+(1-\delta_R)\|v_{0,n}\|_{\mathcal H}^2-\delta_R-o_R(1)-o_n(1).
\end{aligned}
\end{equation}
The cylindrical localization also gives
\begin{equation}\label{eq:normal-energy-recovery}
1\leq
\sum_{i=1}^{m}\|v_{i,n}\|_{\mathcal H}^2
+\|v_{0,n}\|_{\mathcal H}^2
+\delta_R+o_n(1).
\end{equation}
First choose $R$ large.  Then let $n\to\infty$.
Equations \eqref{eq:normal-global-coercivity-sequence}--
\eqref{eq:normal-energy-recovery} contradict
\eqref{eq:normal-coercivity-contradiction}.  Hence
\eqref{eq:positive-normal-coercivity} holds.

It remains to exclude a zero mode and to obtain an orthogonal
splitting. Put
$\mathcal N_\varepsilon=\operatorname{span}
\{P_{1,\varepsilon},\ldots,P_{m,\varepsilon}\}$. The invertibility of
the mass matrix shows that
$\mathcal X_\varepsilon=\mathcal N_\varepsilon\oplus
\mathcal Y_\varepsilon$ as a topological direct sum. Moreover,
\[
\sup_{\substack{n\in\mathcal N_\varepsilon,
 y\in\mathcal Y_\varepsilon\\
\|n\|_{\mathcal H}=\|y\|_{\mathcal H}=1}}
|\mathcal Q_\varepsilon(n,y)|=o(1).
\]
Indeed, after localization at the $i$-th core,
$P_{i,\varepsilon}$ converges to the constant mode $1$, while the
mass condition defining $\mathcal Y_\varepsilon$ gives
$\int F\widetilde y_i=0$. Hence the limiting cross term is
$\mathcal Q_\infty(1,\widetilde y_i)=-2\int F\widetilde y_i=0$.
The neck, outer, and distinct-core terms tend to zero by Lemma
\ref{lem:quadratic-localization}. More precisely, the same estimates
used in Lemma \ref{lem:scale-tangent-exterior}, together with
\eqref{eq:negative-capacity-energy}, give an upper bound of the form
\[
C|\log\varepsilon|^{1/2}
\left(R_\varepsilon^{-2}+|\log\varepsilon|^{-1}
+\varepsilon^\tau\right)=o(1)
\]
for the displayed supremum. Here the factor
$|\log\varepsilon|^{1/2}$ compensates for the normalization of the
capacity profiles, and the diagonal choice in
\eqref{eq:diagonal-radius-rate} makes the bound tend to zero.

By \eqref{eq:positive-normal-coercivity}, for each
$n\in\mathcal N_\varepsilon$ there is a unique
$T_\varepsilon n\in\mathcal Y_\varepsilon$ such that
\[
\mathcal Q_\varepsilon(n+T_\varepsilon n,y)=0
\qquad\text{for every }y\in\mathcal Y_\varepsilon.
\]
The preceding cross estimate gives
$\|T_\varepsilon n\|_{\mathcal H}=o(1)
\|n\|_{\mathcal H}$. Therefore the form on the graph
\[
\mathcal E_\varepsilon^-
=\{n+T_\varepsilon n:n\in\mathcal N_\varepsilon\}
\]
has matrix $-2M_\alpha I_m+o(1)$ in the basis generated by the
$P_{i,\varepsilon}$. It is uniformly negative. Taking
$\mathcal E_\varepsilon^+=\mathcal Y_\varepsilon$ gives the
$\mathcal Q_\varepsilon$-orthogonal decomposition
\eqref{eq:normal-inertia-decomposition}. In particular the normal
block is nondegenerate and has Morse index exactly $m$.
\end{proof}

\subsection{Inertia of the tangent block}

The inertia of the tangent block is determined by the reduced Hessian.

\begin{proposition}\label{prop:tangent-inertia}
At the critical parameter, the tangent Hessian has the block expansion
\eqref{eq:reduced-Hessian-block}. Its scale Schur complement is positive definite. Therefore
\begin{equation}\label{eq:tangent-index}
\operatorname{ind}
(D^2\widetilde J_\varepsilon)
=\operatorname{ind}
(-D^2\mathcal F_m(\boldsymbol\xi^*)).
\end{equation}
The scale block of the reduced Hessian has $m$ positive eigenvalues
of order $|\log\varepsilon|^{-1}$. The translation block has the
same inertia as $-D^2\mathcal F_m(\boldsymbol\xi^*)$ after
multiplication by positive normalization factors. No assertion about
the numerical eigenvalues of the full infinite-dimensional
linearized realization is made here.
\end{proposition}

\begin{proof}
The block expansion is Proposition \ref{prop:reduced-equations}. The position block is invertible. Sylvester's law of inertia applies. The Schur-complement factorization is
\[
\begin{aligned}
&\begin{pmatrix}I&-B_\varepsilon C_\varepsilon^{-1}\\0&I\end{pmatrix}
\begin{pmatrix}A_\varepsilon&B_\varepsilon\\B_\varepsilon^T&C_\varepsilon\end{pmatrix}
\begin{pmatrix}I&0\\-C_\varepsilon^{-1}B_\varepsilon^T&I\end{pmatrix} =\begin{pmatrix}
A_\varepsilon-B_\varepsilon C_\varepsilon^{-1}B_\varepsilon^T&0\\0&C_\varepsilon
\end{pmatrix}.
\end{aligned}
\]
The first block is positive. The inertia is therefore the inertia of
$C_\varepsilon$. This matrix converges to
$-M_\alpha D^2\mathcal F_m(\boldsymbol\xi^*)$. Formula
\eqref{eq:tangent-index} follows.
\end{proof}

\begin{proof}[Proof of the Morse index formula in Theorem \ref{thm:qualitative}]
The decomposition in Lemma \ref{lem:tangent-normal-reduction} is
exactly orthogonal for $\mathcal Q_\varepsilon$. Proposition
\ref{prop:normal-inertia} contributes $m$ negative normal
directions. Proposition \ref{prop:tangent-inertia} contributes
$\operatorname{ind}(-D^2\mathcal F_m(\boldsymbol\xi^*))$ tangent
directions. Therefore
\begin{equation}\label{eq:Morse-formula-proof}
\operatorname{ind}(u_\varepsilon)
=m+\operatorname{ind}
(-D^2\mathcal F_m(\boldsymbol\xi^*)).
\end{equation}
Since $D^2\mathcal F_m$ acts on $4m$ variables and is nonsingular,
\begin{equation}\label{eq:Morse-sign-reversal}
\operatorname{ind}(-D^2\mathcal F_m)
=4m-\operatorname{ind}(D^2\mathcal F_m).
\end{equation}
Equations \eqref{eq:Morse-formula-proof} and
\eqref{eq:Morse-sign-reversal} prove
\eqref{eq:Morse-formula-intro}--\eqref{eq:Morse-equivalent}.
Propositions \ref{prop:nondegenerate-branch} and
\ref{prop:local-uniqueness} complete the proof of Theorem
\ref{thm:qualitative}.
\end{proof}

\begin{proof}[Proof of Corollary \ref{cor:single-peak}]
For $m=1$, a strict local maximum of $\mathcal F_1$ has negative definite Hessian. Therefore
$-D^2\mathcal F_1$ is positive and contributes no negative tangent direction. The total index is one. At a strict local minimum,
$-D^2\mathcal F_1$ has four negative eigenvalues. The total index is five.
\end{proof}

\section{Technical estimates}

We collect the weighted Riesz-potential bound used throughout the proof.

\begin{lemma}\label{lem:weighted-Riesz}
Let $0<\alpha<4$ and $q>4$. For every $\mu\geq1$, $\xi,x\in\mathbb R^4$, and $s\in\mathbb R$,
\begin{equation}\label{eq:weighted-Riesz-estimate}
\int_{\mathbb R^4}
\frac{\mu^s}{|x-y|^\alpha(1+\mu|y-\xi|)^q}\,dy
\leq
\frac{C\mu^{s+\alpha-4}}
{(1+\mu|x-\xi|)^\alpha}.
\end{equation}
In particular,
\begin{equation}\label{eq:bubble-Riesz-bound}
\int_{\mathbb R^4}
\frac{e^{U_{\mu,\xi}(y)}}{|x-y|^\alpha}\,dy
\leq
\frac{C\mu^{\alpha/2}}
{(1+\mu|x-\xi|)^\alpha},
\end{equation}
and
\begin{equation}\label{eq:bubble-L1}
\int_{\mathbb R^4}e^{U_{\mu,\xi}(y)}\,dy
=C\mu^{-\alpha/2}.
\end{equation}
\end{lemma}

\begin{proof}
Set $z=\mu(x-\xi)$ and $w=\mu(y-\xi)$. The left-hand side of
\eqref{eq:weighted-Riesz-estimate} equals
\[
\mu^{s+\alpha-4}
\int_{\mathbb R^4}
\frac{dw}{|z-w|^\alpha(1+|w|)^q}.
\]
For $|z|\leq2$, the last integral is bounded. For $|z|>2$, split into
$|w|\leq|z|/2$, $|z-w|\leq|z|/2$, and the remaining region. The first and third parts are bounded by
$C|z|^{-\alpha}$. In the second region, $|w|\geq|z|/2$, so the integral is bounded by
$C|z|^{-q}|z|^{4-\alpha}\leq C|z|^{-\alpha}$. This proves the first estimate. The other two follow by inserting the explicit bubble and changing variables.
\end{proof}

The next estimates quantify the small interaction of distinct cores.

\begin{lemma}\label{lem:separated-core-estimates}
If $i\ne j$ and $x\in B_{\eta/4}(\xi_i)$, then
\begin{equation*}
\int_{B_{\eta/4}(\xi_j)}
\frac{K(y)e^{W(y)}}{|x-y|^\alpha}\,dy
\leq C\mu_j^{4-\alpha}.
\end{equation*}
The self contribution in the $i$-th core is comparable with
\[
\mu_i^4(1+\mu_i|x-\xi_i|)^{-\alpha}.
\]
Consequently the relative cross error is $O(\varepsilon^\alpha)$ on bounded core scales. Moreover,
\begin{equation}\label{eq:separated-double}
\varepsilon^{8-\alpha}
\int_{B_{\eta/4}(\xi_i)}
\int_{B_{\eta/4}(\xi_j)}
\frac{K(x)K(y)e^{W(x)+W(y)}}{|x-y|^\alpha}\,dxdy
=O(\varepsilon^\alpha).
\end{equation}
\end{lemma}

\begin{proof}
On the $j$-th core,
$W=U_{\mu_j,\xi_j}+A_j+O(1)$ and
$e^{A_j}\leq C\mu_j^{b_\alpha}$. Separation makes the Riesz kernel bounded. Formula
\eqref{eq:bubble-L1} gives
\[
\int_{B_j}\frac{K(y)e^{W(y)}}{|x-y|^\alpha}\,dy
\leq C\mu_j^{b_\alpha}\mu_j^{-\alpha/2}
=C\mu_j^{4-\alpha}.
\]
The self estimate follows from \eqref{eq:bubble-Riesz-bound} and the matching factor. For the double integral, use
$\int_{B_i}e^W\leq C\mu_i^{4-\alpha}$ at both peaks and then
$\mu_i\asymp\varepsilon^{-1}$.
\end{proof}

The kernel moments form a diagonal positive Gram matrix.

\begin{lemma}\label{lem:kernel-Gram}
With $F=F_{1,0}$ as above,
\begin{equation*}
\int_{\mathbb R^4}FZ_k\,dz=0,
\qquad 0\leq k\leq4,
\end{equation*}
and
\begin{equation*}
\int_{\mathbb R^4}FZ_kZ_\ell\,dz=0\quad(k\ne\ell),
\qquad
\int_{\mathbb R^4}FZ_k^2\,dz>0.
\end{equation*}
\end{lemma}

\begin{proof}
The translation identities follow from oddness. The dilation mode is radial. With $t=r^2$,
\[
\int_{\mathbb R^4}FZ_0
=C\int_0^\infty\frac{r^3(1-r^2)}{(1+r^2)^5}\,dr
=\frac C2\left(
\int_0^\infty\frac{t}{(1+t)^5}\,dt
-\int_0^\infty\frac{t^2}{(1+t)^5}\,dt
\right)=0.
\]
Both beta integrals equal $1/12$. Mixed products vanish by parity and rotational symmetry. Diagonal products are positive and finite.
\end{proof}

The final linear estimate below rules out loss of compactness in the
annular necks between the core and the outer region.

\begin{lemma}\label{lem:remove-defects}
Let $\varepsilon_n\to0$ and
$\mu_{i,n}\asymp\varepsilon_n^{-1}$. Let $v_n$ be uniformly bounded
solutions of projected linear equations with
\begin{equation*}
\|h_n\|_{**}+\sum_{i,k}|c_{i,k,n}|\longrightarrow0.
\end{equation*}
Assume that every core rescaling tends to zero in
$C^{3,\beta}_{\rm loc}(\mathbb R^4)$ for some $\beta\in(0,1)$ and
that $v_n\to0$ in $C^2_{\rm loc}$ away from the limiting centers.
Then
\begin{equation}\label{eq:neck-conclusion}
\|v_n\|_{L^\infty(\Omega)}\longrightarrow0.
\end{equation}
\end{lemma}

\begin{proof}
Fix one peak and suppress its index. For fixed $R>1$ and small
$\rho>0$, let
\[
A_n(R,\rho)
=\{R/\mu_n<|x-\xi_n|<\rho\}.
\]
Write $\Delta^2v_n=g_n$ there and decompose the annulus into the
dyadic rings
\[
A_{n,\nu}
=\{2^\nu R/\mu_n<|x-\xi_n|
<2^{\nu+1}R/\mu_n\}.
\]
The weighted source estimate gives
\begin{equation}\label{eq:neck-source-bound}
\int_{A_{n,\nu}}|g_n|\,dx
\leq
C2^{-\nu(4-\tau)}R^{-(4-\tau)}
+o_n(1)2^{-4\nu}.
\end{equation}
The same calculation after multiplication by
$|x-\xi_n|^j$, $j=1,2$, gives the corresponding moment bounds. For
the nonlocal linear terms, apply Lemma \ref{lem:weighted-Riesz} in
the integration variable and sum over the separated cores.

Define the logarithmic charge
\[
q_n(\rho):=\int_{B_\rho(\xi_n)}g_n(y)\,dy.
\]
Evaluate the global Green representation at
$x=\xi_n+z/\mu_n$, with $z$ fixed. On $B_\rho(\xi_n)$ use
\[
-\log|x-y|
=\log\mu_n-\log|z-\mu_n(y-\xi_n)|.
\]
For fixed $R$ and $\rho$, the second term is absolutely summable over
the dyadic rings by \eqref{eq:neck-source-bound}. The contribution of
the fixed rescaled core tends to zero by the limiting equation, and
the contribution of $\Omega\setminus B_\rho(\xi_n)$ tends to zero by
the outer convergence. Therefore
\begin{equation}\label{eq:neck-charge}
|\log\mu_n|\,|q_n(\rho)|
=o_n(1)+o_R(1)+o_\rho(1),
\end{equation}
where the limits are understood in the order
$n\to\infty$, $R\to\infty$, and $\rho\to0$.

Let $b_n$ solve
\begin{equation*}
\begin{cases}
\Delta^2b_n=0&\text{in }A_n(R,\rho),\\
b_n=v_n,\quad \Delta b_n=\Delta v_n
&\text{on both boundary spheres}.
\end{cases}
\end{equation*}
We give the interpolation estimate in detail. Put
$a=R/\mu_n$ and $b=\rho$. For a spherical harmonic of degree
$l\geq1$, write
\[
b_{n,l}(r)=A_lr^l+B_lr^{-l-2}+C_lr^{l+2}+D_lr^{-l}.
\]
Since
\[
\Delta(r^{l+2}Y_l)=4(l+2)r^lY_l,
\qquad
\Delta(r^{-l}Y_l)=-4lr^{-l-2}Y_l,
\]
the coefficient vector is determined by
\begin{equation}\label{eq:neck-interpolation-matrix}
\begin{pmatrix}
a^l&a^{-l-2}&a^{l+2}&a^{-l}\\
0&0&4(l+2)a^{l+2}&-4la^{-l}\\
b^l&b^{-l-2}&b^{l+2}&b^{-l}\\
0&0&4(l+2)b^{l+2}&-4lb^{-l}
\end{pmatrix}
\begin{pmatrix}A_l\\B_l\\C_l\\D_l\end{pmatrix}
=
\begin{pmatrix}
v_{n,l}(a)\\a^2\Delta v_{n,l}(a)\\
v_{n,l}(b)\\b^2\Delta v_{n,l}(b)
\end{pmatrix}.
\end{equation}
The lower right $2\times2$ determinant is
\[
16l(l+2)a^{-l}b^{-l}
\bigl(b^{2l+2}-a^{2l+2}\bigr),
\]
and is nonzero. Solving first for $C_l,D_l$ and then for $A_l,B_l$
gives, uniformly when $0<a<b$ and $b/a\geq2$,
\begin{equation}\label{eq:neck-interpolation-estimate}
\begin{aligned}
\sup_{a\leq r\leq b}|b_{n,l}(r)|
\leq C\Big(&|v_{n,l}(a)|+|v_{n,l}(b)|+(1+l)^{-2}
\bigl(|a^2\Delta v_{n,l}(a)|
+|b^2\Delta v_{n,l}(b)|\bigr)\Big).
\end{aligned}
\end{equation}
For fixed $R$ and $\rho$, the condition $b/a\geq2$ holds for all
large $n$.  Estimate \eqref{eq:neck-interpolation-estimate} and the
same elimination after tangential differentiation prove the
corresponding estimates in the spherical Sobolev norms. The
$C^{3,\beta}$ convergence of the inner rescaling and the $C^2$
outer convergence therefore imply that the sum of all nonzero
spherical harmonics in $b_n$ is
$o_n(1)+o_R(1)+o_\rho(1)$ in $L^\infty$.

It remains to control the spherical mean.  Set
$\overline v_n(r)=|\mathbb S^3|^{-1}\int_{\mathbb S^3}
v_n(\xi_n+r\theta)\,d\theta$.  The radial flux identity is
\begin{equation}\label{eq:neck-radial-flux}
|\mathbb S^3|r^3\partial_r\Delta\overline v_n(r)
=\int_{B_r(\xi_n)}g_n(y)\,dy.
\end{equation}
Integrating \eqref{eq:neck-radial-flux} twice between
$a=R/\mu_n$ and $b=\rho$, and using the boundary values at $a$ and
$b$, yields
\begin{equation}\label{eq:neck-zero-mode-estimate}
\begin{aligned}
\sup_{a\leq r\leq b}|\overline v_n(r)|
\leq C\Big(&|\overline v_n(a)|+|\overline v_n(b)|
+|a^2\Delta\overline v_n(a)|
+|b^2\Delta\overline v_n(b)|\\
&+|\log(b/a)|\,|q_n(\rho)|
+\sum_{\nu\geq0}(1+\nu)
\int_{A_{n,\nu}}|g_n|\,dx\Big).
\end{aligned}
\end{equation}
The fixed core convergence controls the first and third terms, and
the outer convergence controls the second and fourth terms.  By
\eqref{eq:neck-charge},
$|\log(b/a)|\,|q_n(\rho)|
=o_n(1)+o_R(1)+o_\rho(1)$.  The dyadic sum in
\eqref{eq:neck-zero-mode-estimate} is
$O(R^{-(4-\tau)})+o_n(1)$ by
\eqref{eq:neck-source-bound}.  Hence
\begin{equation}\label{eq:neck-mean-small}
\sup_{a\leq r\leq b}|\overline v_n(r)|
=o_n(1)+o_R(1)+o_\rho(1).
\end{equation}

Let $b_n^\perp$ be the sum of the nonzero spherical harmonics of
$b_n$.  Estimate \eqref{eq:neck-interpolation-estimate} and the same
elimination after tangential differentiation give the corresponding
bounds in the spherical Sobolev norms.  The
$C^{3,\beta}$ convergence of the inner rescaling and the $C^2$
outer convergence therefore imply
\begin{equation}\label{eq:neck-homogeneous-small}
\|b_n^\perp\|_{L^\infty(A_n(R,\rho))}
=o_n(1)+o_R(1)+o_\rho(1).
\end{equation}

Put
\[
w_n^\perp
=v_n-\overline v_n-b_n^\perp.
\]
This function has zero Navier data on both boundary spheres, has zero
spherical mean, and satisfies the nonzero-mode part of
$\Delta^2w_n^\perp=g_n$.  Let $G_{n,R,\rho}^\perp$ denote the
annular Navier Green kernel with its degree-zero component removed.
The same radial basis as in
\eqref{eq:neck-interpolation-matrix} gives the following quantitative
bound.  If $x$ and $y$ belong to the dyadic rings with indices
$\nu$ and $k$, respectively, the degree $l\geq1$ contribution is
bounded by
\[
C(1+l)^{-2}2^{-(l+1)|\nu-k|}.
\]
Summation in $l$ yields
\begin{equation}\label{eq:neck-Green-remainder}
|G_{n,R,\rho}^\perp(x,y)|
\leq C2^{-|\nu-k|},
\end{equation}
uniformly in $n$, $R$, and $\rho$.  Combining
\eqref{eq:neck-Green-remainder} with
\eqref{eq:neck-source-bound} and the first two moment bounds gives
\[
\sup_{A_n(R,\rho)}|w_n^\perp|
\leq o_n(1)+CR^{-(4-\tau)}+o_\rho(1).
\]
Together with \eqref{eq:neck-mean-small} and
\eqref{eq:neck-homogeneous-small},
\[
\sup_{A_n(R,\rho)}|v_n|
\leq o_n(1)+CR^{-(4-\tau)}+o_\rho(1).
\]
First let $n\to\infty$, then $R\to\infty$, and finally
$\rho\to0$.  Repeating the argument at every peak and combining it
with the core and outer convergence proves
\eqref{eq:neck-conclusion}.
\end{proof}

We finally record the standard Hessian identity on the projected solution manifold.

\begin{lemma}\label{lem:Hessian-identity}
Let $p\mapsto u(p)$ be the $C^2$ projected solution manifold and let
$\mathcal J(p)=J_\varepsilon(u(p))$. At an exact solution,
\begin{equation}\label{eq:abstract-Hessian-identity}
\partial_{p_ap_b}^2\mathcal J
=\mathcal Q_\varepsilon(\partial_{p_a}u,\partial_{p_b}u).
\end{equation}
The same identity holds after the scale variables have been eliminated.
\end{lemma}

\begin{proof}
Differentiate
$\partial_{p_a}\mathcal J=J_\varepsilon'(u)[\partial_{p_a}u]$. At an exact solution the term containing
$J_\varepsilon'(u)$ vanishes after a second differentiation. This gives
\eqref{eq:abstract-Hessian-identity}. The chain rule gives the statement on the reduced graph.
\end{proof}

\noindent{\bf Acknowledgements}

The authors were supported by the National Natural Science Foundation of China (No. 12371121) and the Fundamental Research Funds for the Central Universities (No. SWU-KF26002).

\noindent{\bf Data availability statement}

Data availability is not applicable to this article as no datasets were generated or analysed during the current study.

\noindent{\bf Conflict of interest}

The authors declare that there is no conflict of interest.

\end{document}